\documentclass[11pt]{amsart}
\usepackage[linkcolor=blue, citecolor = blue]{hyperref}
\hypersetup{colorlinks=true}

\usepackage{amsthm,amsfonts,amsmath,amssymb,amscd} 

\usepackage{verbatim}
\usepackage{float}
\usepackage{wrapfig}
\usepackage{mathtools}
\usepackage[color=orange!20, linecolor=orange]{todonotes}

\usepackage{ytableau}
\usepackage{array}

\usepackage{geometry}
\usepackage{enumitem}

\usepackage{tikz}
\usetikzlibrary{decorations.markings}
\tikzstyle{vertex}=[circle, draw, inner sep=0pt, minimum size=4pt]
\newcommand{\vertex}{\node[vertex]}

\usetikzlibrary{arrows,automata}
\usepackage{tikz, tikz-3dplot, pgfplots}
\usepackage{tkz-graph}
\usetikzlibrary[positioning,patterns]

\newtheorem{theorem}{Theorem}[section]
\newtheorem{proposition}[theorem]{Proposition}
\newtheorem{lemma}[theorem]{Lemma}
 
\theoremstyle{definition}
\newtheorem{corollary}[theorem]{Corollary}
\newtheorem{definition}[theorem]{Definition}
\newtheorem{example}[theorem]{Example}
\newtheorem{question}[theorem]{Question}
\newtheorem{problem}[theorem]{Problem}

\theoremstyle{remark}
\newtheorem{remark}[theorem]{Remark}

\newcommand{\sgn}{\mathrm{sgn}}

\title
{
Duality and deformations of stable Grothendieck polynomials
}
\author{Damir Yeliussizov}
\address{Kazakh-British Technical University, 59 Tole bi st. Almaty, Kazakhstan}
\address{Department of Mathematics, UCLA, Los Angeles, CA 90095 }
\email{yeldamir@gmail.com}

\begin{document}

\begin{abstract}
Stable Grothendieck polynomials can be viewed as a K-theory analog of Schur polynomials. We extend stable Grothendieck polynomials to a two-parameter version, which we call canonical stable Grothendieck functions. These functions have the same structure constants (with scaling) as stable Grothendieck polynomials, and (composing with parameter switching) are self-dual under the standard involutive ring automorphism. 
We study various properties of these functions, including combinatorial formulas, Schur expansions, Jacobi-Trudi type identities, and associated Fomin-Greene operators. 
\end{abstract}

\maketitle


\section{Introduction}
Stable Grothendieck polynomials are certain symmetric power series first studied by Fomin and Kirillov \cite{fk, fk1}. These functions arise as a stable limit of Grothendieck polynomials introduced by Lascoux and Sch\"utzenberger \cite{ls}. The stable Grothendieck polynomial $G_{\lambda}(x_1, x_2, \ldots)$ (corresponding to a Grassmannian permutation) can be viewed as a deformation and K-theory analog of the Schur function $s_{\lambda}(x_1, x_2, \ldots)$ {(while the Grothendieck polynomial is an analog of  Schubert polynomial).} 

As a symmetric function, $G_{\lambda}$ has many similarities with $s_{\lambda}$. For example, it can be defined by the following `bi-alternant' formula \cite{in, kir1, ms}
$$
G_{\lambda}(x_1, \ldots, x_n) = \frac{\det\left[x_i^{\lambda_j + n - j}(1 - x_i)^{j - 1} \right]_{1 \le i,j \le n}}{\prod_{1 \le i < j \le n}(x_i - x_j)}
$$
and has a combinatorial presentation given by the generating series
$$
G_{\lambda}(x_1, x_2, \ldots) = \sum_{T} (-1)^{|T| - |\lambda|} \prod_{i \ge 1} x_i^{\# i\text{'s in }T},
$$
where the sum runs over  {\it set-valued tableaux} of shape $\lambda$; a generalization of semi-standard Young tableaux (SSYT), where boxes may contain sets of integers \cite{buch}.

Let $\Lambda$ be the ring of symmetric functions in infinitely many variables $x = (x_1, x_2, \ldots)$. 
Denote by $\hat\Lambda$ the completion of $\Lambda$ which includes infinite linear combinations of the basis elements (in some distinguished basis of $\Lambda$, e.g. Schur functions). 

Note that $G_{\lambda} \in \hat\Lambda$, for instance $G_{(1)} = e_1 - e_2 + e_3 - \cdots$, where $e_k$ is the $k$th elementary symmetric function. It is remarkable that the product of stable Grothendieck polynomials has a finite decomposition
\begin{equation}\label{prodg}
G_{\lambda} G_{\mu} = \sum_{\nu} (-1)^{|\nu| - |\lambda| - |\mu|}c_{\lambda \mu}^{\nu}{G}_{\nu}, \quad c_{\lambda \mu}^{\nu} \in \mathbb{Z}_{\ge 0}, \quad |\nu| \ge |\lambda| + |\mu|.
\end{equation}
This result was proved by Buch \cite{buch} and he described the coefficients $c_{\lambda \mu}^{\nu}$ combinatorially using set-valued tableaux, generalizing the Littlewood-Richardson rule for Schur functions. 
Buch studied the Grothendieck ring of the Grassmannian $\mathrm{Gr}(k, \mathbb{C}^n)$ of $k$-planes in $\mathbb{C}^n$ as the quotient ring $\Gamma/\left\langle G_{\lambda}, \lambda \not\subseteq (n-k)^k \right\rangle$, where $\Gamma = \bigoplus_{\lambda} \mathbb{Z} \cdot G_{\lambda}$ is a ring with a basis of Grothendieck polynomials ($(n-k)^k$ is the partition of rectangular shape $k \times (n - k)$).

By the fundamental duality isomorphism of the Grassmannian $\mathrm{Gr}(k, \mathbb{C}^n) \cong \mathrm{Gr}(n-k, \mathbb{C}^n)$, structure constants have the symmetry $c^{\nu}_{\lambda \mu} = c^{\nu'}_{\lambda' \mu'}$ where $\lambda'$ denotes the conjugate of $\lambda$. Therefore the involutive linear map $\tau : \hat\Lambda \to \hat\Lambda$ (or $\Gamma \to \Gamma$; the completion $\hat\Gamma$ of $\Gamma$ coincides with $\hat\Lambda$) defined on bases by setting $\tau(G_{\lambda}) = G_{\lambda'}$, is a ring homomorphism. The standard involutive ring automorphism $\omega$ (which maps $e_k$ to $h_k$) extended on $\hat\Lambda$ by mapping the Schur bases $s_{\lambda}$ to $s_{\lambda'}$, does not lead to self-duality for Grothendieck basis, $\omega(G_{\lambda}) = J_{\lambda}  \ne G_{\lambda'}$ \cite{lp}. So another family $\{J_{\lambda}\}$ has the same structure constants. We first ask the following question. 

\begin{question}
Is there a family $\{\widetilde{G}_{\lambda}\}$ which has the same structure constants as $\{{G}_{\lambda}\}$ and is self-dual under the standard involution $\omega$, $\omega(\widetilde{G}_{\lambda}) = \widetilde{G}_{\lambda'}$? 
\end{question}

Our aim is to deform stable Grothendieck polynomials to adjust it to the canonical involution $\omega$. We will give a concrete construction of these symmetric functions and state the following basic result. 

\begin{theorem}\label{t1}
There is an automorphism $\phi : \hat\Lambda \to \hat\Lambda$ satisfying $\omega =  \phi \tau \phi^{-1}$.  
\end{theorem}

Hence $\omega\phi(G_{\lambda}) = \phi(G_{\lambda'})$ and the symmetric function $ \widetilde{G}_{\lambda} := \phi(G_{\lambda})\in \hat\Lambda$ has the same ring structure and satisfies the duality $\omega(\widetilde{G}_{\lambda}) = \widetilde{G}_{\lambda'}$. As we will see, $\phi$ is given explicitly by the substitution $x \mapsto 2x/(2 + x)$.

There is a comultiplication $\Delta : \Gamma \to \Gamma \otimes \Gamma$ given by
\begin{equation} 
\Delta(G_{\nu}) = \sum_{\lambda, \mu} (-1)^{|\lambda| + |\mu| - |\nu|}d^{\nu}_{\lambda \mu} G_{\lambda} \otimes G_{\mu}, \qquad d^{\nu}_{\lambda \mu} \in \mathbb{Z}_{\ge 0}.
\end{equation}
Both product \eqref{prodg} and coproduct $\Delta$ make $\Gamma$ a commutative and cocommutative bialgebra \cite{buch}. The completion $\hat\Gamma$ of $\Gamma$ is a Hopf algebra \cite{lp} with the antipode given by $S(G_{\lambda}(x)) = \omega(G_{\lambda}(-x))$ \cite{patrias}. 

The coproduct of $\Gamma$ is compatible with the dual family for $G_{\lambda}$ via the Hall inner product. These dual stable Grothendieck polynomials $g_{\lambda} \in \Lambda$ were explicitly described by Lam and Pylyavskyy \cite{lp} using reverse plane partitions. We have
$$
g_{\lambda} g_{\mu} = \sum_{\nu} (-1)^{|\lambda| + |\mu| - |\nu|}d^{\nu}_{\lambda \mu} g_{\nu}.
$$
The constants $d^{\nu}_{\lambda \mu}$ are also symmetric up to diagram transpositions, $d^{\nu}_{\lambda \mu} = d^{\nu'}_{\lambda' \mu'}$ and similarly $\omega(g_{\lambda}) \ne g_{\lambda'}$. The dual polynomials $\{ g_{\lambda}\}$ is a (non-homogeneous) $\mathbb{Z}$-basis of $\Lambda$ and the linear map $\bar\tau : \Lambda \to \Lambda$ given by $\bar\tau(g_{\lambda}) = g_{\lambda'}$ is a ring homomorphism. 
Similarly, there is an automorphism $\bar\phi : \Lambda \to \Lambda$ satisfying $\omega = \bar\phi \bar\tau \bar\phi^{-1}$ and we introduce the polynomials $\widetilde{g}_{\lambda} \in \Lambda$ dual to $\widetilde{G}_{\lambda}$ via the Hall inner product, so they also satisfy the duality $\omega(\widetilde{g}_{\lambda}) = \widetilde{g}_{\lambda'}.$ Once we specify the elements $\widetilde{g}_{(1)}, \widetilde{g}_{(2)}, \ldots$ as (free) generators of (the polynomial ring) $\Lambda$, they characterize the automorphism $\bar\phi$ (and hence the dual family of the $G$-functions).

The functions $\widetilde{G}_{\lambda}$ can be extended to $\widetilde{G}_{w}$ for any permutation $w \in S_n$, similarly as stable Grothendieck polynomials $G_{w}$ arise. Here we have the duality $\omega(\widetilde{G}_{w}) = \widetilde{G}_{w^{-1}},$ as well as $\omega(\widetilde{G}_{w}) = \widetilde{G}_{w_0w w_0},$ where $w_0 \in S_n$ is the longest permutation.

More generally, the focus of this paper is on two-parameter versions of stable Grothendieck polynomials, the dual symmetric functions $G^{(\alpha, \beta)}_{\lambda}(x_1, x_2, \ldots),$ $g^{(\alpha, \beta)}_{\lambda}(x_1, x_2, \ldots)$. We call them the {\it canonical stable Grothendieck functions} and the {\it dual canonical stable Grothendieck polynomials}. In special cases they correspond to:
\begin{itemize} 
\item[(a)] Schur functions $s_{\lambda} = G^{(0,0)}_{\lambda} = g^{(0,0)}_{\lambda}$; the case $\alpha + \beta = 0$ corresponds to certain deformed Schur functions;
\item[(b)] stable Grothendieck polynomials $G_{\lambda} = G^{(0,-1)}_{\lambda}$ and its dual $g_{\lambda} = g^{(0,1)}_{\lambda}$;
\item[(c)] weak stable Grothendieck polynomials $J_{\lambda} = \omega(G_{\lambda}) = G^{(-1,0)}_{\lambda'}$, and its dual $j_{\lambda} = \omega(g_{\lambda}) = g^{(1,0)}_{\lambda'}$;
\item[(d)] and the functions discussed above: $\widetilde{G}_{\lambda} = G^{(-1/2, -1/2)}_{\lambda}$, $\widetilde{g}_{\lambda} = g^{(1/2, 1/2)}_{\lambda}$. 
\end{itemize}
The functions $G^{(\alpha, \beta)}_{\lambda} \in \hat\Lambda$, $g^{(\alpha, \beta)}_{\lambda} \in \Lambda$ are non-homogeneous symmetric, 
$$
G^{(\alpha, \beta)}_{\lambda} = s_{\lambda} + \text{ higher degree terms},
\qquad
g^{(\alpha, \beta)}_{\lambda} = s_{\lambda} + \text{ lower degree terms}.
$$
The families $\{g^{(\alpha, \beta)}_{\lambda} \},$ $\{G^{(-\alpha, -\beta)}_{\lambda} \}$ are dual via the Hall inner product $\langle g^{(\alpha, \beta)}_{\lambda}, G^{(-\alpha, -\beta)}_{\mu}\rangle = \delta_{\lambda, \mu}$, and the involution acts on them as follows: 
$$\omega(G^{(\alpha, \beta)}_{\lambda}) = G^{(\beta, \alpha)}_{\lambda'},\qquad \omega(g^{(\alpha, \beta)}_{\lambda}) = g^{(\beta, \alpha)}_{\lambda'}.$$
Multiplication of the $G^{(\alpha, \beta)}_{\lambda}$ is governed by the same structure constants (up to scaling) as for the polynomials $G_{\lambda}$, we have
$$
G^{(\alpha, \beta)}_{\lambda} G^{(\alpha, \beta)}_{\mu} = \sum_{\nu} (\alpha + \beta)^{|\nu| - |\lambda| - |\mu|} c^{\nu}_{\lambda\mu} G^{(\alpha, \beta)}_{\nu}
$$
and the comultiplication $\Delta$ is given by
$$
\Delta(G^{(\alpha, \beta)}_{\nu}) = \sum_{\lambda, \mu} (\alpha + \beta)^{|\lambda| + |\mu| -|\nu|} d^{\nu}_{\lambda \mu} G^{(\alpha, \beta)}_{\lambda} \otimes G^{(\alpha, \beta)}_{\mu}.
$$
There are dual Hopf algebra structures with these dual bases parametrized by $\alpha, \beta$ and similar properties. 

In some sense, the first new parameter $\alpha$ in $G^{(\alpha, \beta)}_{\lambda}$ uncovers duality (under the involution $\omega$) of the $\beta$-Grothendieck polynomials $G^{(0, \beta)}_{\lambda}$. As we will see, construction of the canonical version $G^{(\alpha, \beta)}_{\lambda}$ corresponds to an appropriate substitution of variables. For the dual polynomials $g^{(\alpha, \beta)}_{\lambda}$, description appear to be more complicated, especially its combinatorial presentation.
Combining the `unifying' and duality (conjugation) properties described above, the reason for calling these symmetric functions as {\it canonical} is also the following. In the specialization $(\alpha, \beta) = (q, q^{-1})$, the elements $\{g^{(\alpha, \beta)}_{\rho} : \rho \text{ is a single row or column} \}$ (under the involution $\omega$) admit a similar characterization as the Kazhdan-Lusztig canonical bases. The elements $\{g^{(\alpha, \beta)}_{(k)} :  k\in \mathbb{Z}_{> 0} \}$ then characterize the dual functions $g^{(\alpha, \beta)}_{\lambda}$ as generators for $\Lambda$. 

Let us summarize some further results about the functions $G^{(\alpha, \beta)}_{\lambda}$, $g^{(\alpha, \beta)}_{\lambda}$.  

Combinatorial formulas. They are based on several new types of tableaux:
\begin{itemize}
\item[-] {\it Hook-valued tableaux for $G^{(\alpha, \beta)}_{\lambda}$.} Each box of such tableau contains a semistandard {\it hook}, the hooks then `weakly increase' in rows and `striclty increase' in columns (Section \ref{hvt}). This presentation combines set-valued tableaux \cite{buch} (sets are single column hooks) and weak set-valued (or multiset-valued) tableaux given in \cite{lp} for description of $J_{\lambda} = \omega(G_{\lambda})$ (multisets are single row hooks).

\item[-] {\it Rim border tableaux for $g^{(\alpha, \beta)}_{\lambda}$.} These tableaux are constructed via a special decomposition of reverse plane partitions (RPP) into {\it rim hooks} on borders of the same entries (Section~\ref{rrt}). Here we also describe equivalent objects, called {\it lattice forests on RPP}. Similarly, for $\alpha \beta = 0$ combinatorics of $g^{(\alpha, \beta)}_{\lambda}$ corresponds to the known combinatorial presentations of the dual stable Grothendieck polynomials $g_{\lambda}$ and $j_{\lambda}=\omega(g_{\lambda})$ given in \cite{lp}. For $\alpha = 0$, we have generating series running over RPP with a special (column) weight and for $\beta = 0$ they run over SSYT with a special (row) weight. 
\end{itemize}
Combinatorial formulas are accompanied with various Pieri type formulas (Section \ref{spieri}) for multiplying $G^{(\alpha, \beta)}_{\lambda}, g^{(\alpha, \beta)}_{\lambda}$ on the functions $e_{k},$ $h_k,$ $G^{(\alpha, \beta)}_{(k)},$ $g^{(\alpha, \beta)}_{(k)},$ $G^{(\alpha, \beta)}_{(1^k)},$ $g^{(\alpha, \beta)}_{(1^k)}.$ 

We prove the duality $\omega(G^{(\alpha, \beta)}_{\lambda}) = G^{(\beta, \alpha)}_{\lambda'}$ using the method of Fomin and Greene \cite{fg} on noncommutative Schur functions. Our approach is based on Schur operators (Section \ref{dfg}), which also gives a way to define the functions indexed by skew shapes. 

In section \ref{schur} we present Schur expansions and related combinatorics. 
In particular, we show that $g^{(\alpha, \beta)}_{\lambda}$ are Schur-positive (i.e. their transition coefficients are polynomials in $\alpha, \beta$ with positive integer coefficients). We give combinatorial and determinantal formulas for connection constants. 

Jacobi-Trudi type determinantal identities (Section \ref{jt}). To obtain them, we use information given in Schur expansions. Using determinantal formulas for connection constants we obtain new determinantal identities via the Cauchy-Binet formula. In essence, this method gives combinatorial proofs of corresponding identities, which we discuss in detail for $\alpha = 0$.

In section \ref{ggw} we extend $G^{(\alpha, \beta)}_{\lambda}$ to the functions $G^{(\alpha, \beta)}_{w}$ indexed by permutations $w \in S_n$.
In section \ref{skl} we put the canonical stable Grothendieck polynomials $g^{(\alpha, \beta)}_{\lambda}, G^{(\alpha, \beta)}_{\lambda}$ (and more generally bases of the ring of symmetric functions) in context of the Kazhdan-Lusztig theory of canonical bases, we give corresponding characterization and discuss some related problems. 

{
\subsection*{Acknowledgements} I am grateful to Pavlo Pylyavskyy for stimulation of this project, many helpful discussions and insightful suggestions. Initial stages of this work began while I was visiting the IMA (Institute for Mathematics and its Applications), and the major part was completed during my visit in the Department of Mathematics at MIT. 
I~am thankful to Richard Stanley and Alexander Postnikov for their hospitality at MIT and for helpful conversations. I~am grateful to Askar Dzhumadil'daev for his support and helpful discussions, parts of this work were reported in his seminar. I also thank Sergey Fomin and Victor Reiner for their comments, and the referees for helpful suggestions. 
}

\section{Preliminaries and background on Grothendieck polynomials}\label{prelim}

\subsection{Symmetric functions}  
We assume some familiarity with basic theory, e.g. \cite{macdonald, ec2}. Let $\Lambda$ be the ring of symmetric functions in infinitely many variables $x = (x_1, x_2, \ldots)$. The elementary symmetric function $e_k$ and complete homogeneous symmetric function $h_k$ are given by 
$$e_k = \sum_{i_1 < \cdots < i_k} x_{i_1} \cdots x_{i_k}, \qquad h_k = \sum_{i_1 \le \cdots \le i_k} x_{i_1} \cdots x_{i_k}.$$
The ring $\Lambda$ is a polynomial ring generated by $e_1, e_2, \ldots$ (or $h_1, h_2, \ldots$) and the standard involutive ring automorphism $\omega : \Lambda \to \Lambda$ maps $e_k$ to $h_k$. 
\footnote{We will usually write $F$ or $F(x)$ meaning that it is $F(x_1, x_2, \ldots)$. Similarly, $F(x/(1-\alpha x))$ refers to the function $F(x_1/(1 - \alpha x_1), x_2/(1 - \alpha x_2), \ldots)$. If the function $F(x)$ is of a single variable $x$ it will be stated or clear from the context.}

Classical bases of $\Lambda$ are indexed by partitions. A {\it partition} is a sequence $\lambda = (\lambda_1 \ge \lambda_2 \ge \ldots)$ of nonnegative integers with only finitely many nonzero terms. The weight of a partition $\lambda$ is the sum $|\lambda| = \lambda_1 + \lambda_2 + \cdots.$ Any partition $\lambda$ can be viewed as a {\it Young diagram} which contains $\lambda_1$ boxes in the first row, $\lambda_2$ boxes in the second row and so on; equivalently it is the set $\{(i, j) : 1 \le i \le \ell, 1 \le j \le \lambda_i \}$, where $\ell = \ell(\lambda)$ is the number of nonzero parts of $\lambda$. (We use English notation for Young diagrams.) The partition $\lambda'$ with transposed Young diagram, is the {\it conjugate} of $\lambda$. 

We consider $\Lambda$ over $\mathbb{Z}$, e.g. with the Schur basis $\Lambda = \bigoplus_{\lambda} \mathbb{Z} \cdot s_{\lambda}$. We have $\omega(s_{\lambda}) = s_{\lambda'}$ and $\Lambda$ is equipped with the (standard) {\it Hall inner product} $\langle\cdot, \cdot \rangle : \Lambda \times \Lambda \to \mathbb{Z}$ which makes Schur functions an orthonormal basis, $\langle s_{\lambda}, s_{\mu}\rangle = \delta_{\lambda \mu},$ where $\delta$ is the Kronecker symbol. 

Denote by $\hat\Lambda$ the completion of $\Lambda$ which consists of symmetric power series (of unbounded degree). For the basis of Schur functions $s_{\lambda}$ each element $f \in \hat\Lambda$ can uniquely be  written as (possibly an infinite sum)
$
f = \sum_{\lambda} a_{\lambda} s_{\lambda}, a_{\lambda} \in \mathbb{Z}.
$ 
The Hall inner product $\langle \cdot, \cdot \rangle$ and the ring automorphism $\omega$ extend as follows: $\langle \cdot, \cdot \rangle : \Lambda \times \hat\Lambda \to \mathbb{Z}$ by $\langle s_{\lambda}, s_{\mu} \rangle = \delta_{\lambda, \mu}$; and $\omega : \hat\Lambda \to \hat\Lambda$ by $\omega(s_{\lambda}) = s_{\lambda'}$.

\subsection{Grothendieck polynomials} 
For a  generic parameter $\beta$ (usually $\beta = \pm 1$), the {\it stable Grothendieck polynomial} $G^{\beta}_{\lambda}(x_1,  \ldots,  x_n)$ is a symmetric polynomial which can be defined as (see \cite{in, kir1, ms})
\begin{equation}
G^\beta_{\lambda}(x_1, \ldots, x_n) = \frac{\det\left[x_i^{\lambda_j + n - j}(1 - \beta x_i)^{j - 1} \right]_{1 \le i,j \le n}}{\prod_{1 \le i < j \le n}(x_i - x_j)}.
\end{equation}
Combinatorially it is presented as the generating power series 
\begin{equation}\label{svt}
G^{\beta}_{\lambda}(x_1, x_2, \ldots) = \sum_{T} \beta^{|T| - |\lambda|} \prod_{i \ge 1} x_i^{\# \text{ of $i$'s in }T},
\end{equation}
where the sum runs over {\it set-valued tableaux} \cite{buch}, that are fillings of the boxes of the Young diagram of $\lambda$ with nonempty {\it sets} of positive integers such that if we replace each set by any of its elements the resulting tableau is always a {\it semi-standard Young tableau} (SSYT), i.e. the numbers weakly increase from left to right in each row and strictly increase from top to bottom in each column. In other words, a maximal element in each box of a set-valued tableau is less or equal than any element in a box to the right (in the same row) and strictly less than any element in boxes below (in the same column). 

Obviously, $ s_{\lambda} = G^{\beta}_{\lambda}$ for $\beta = 0$ and we set $G_{\lambda} = G^{\beta}_{\lambda}$ for $\beta = -1$.  Note that
$$
\beta^{|\mu|} G^{\beta}_{\lambda} = G^{1}_{\lambda}(\beta x_1, \beta x_2, \ldots),
\quad
(-\beta)^{|\mu|} G^{\beta}_{\lambda} = G_{\lambda}(-\beta x_1, -\beta x_2, \ldots).
$$

Originally, the functions $G_{\lambda}$ arose as a stable limit of (more general) Grothendieck polynomials indexed by permutations where the partition $\lambda$ corresponds to a Grassmannian permutation \cite{fk}. We touch this setting in more detail in section \ref{ggw} when we consider $(\alpha, \beta)$-deformations of Grothendieck polynomials indexed by permutations. 


Buch \cite{buch} proved that the product of stable Grothendieck polynomials has a finite decomposition
$$
 G_{\lambda}  G_{\mu} = \sum_{\nu} (-1)^{|\nu| - |\lambda| - |\mu|} c^{\nu}_{\lambda \mu}  G_{\nu}, \qquad c^{\nu}_{\lambda \mu}\in \mathbb{Z}_{\ge 0}
$$
and described a combinatorial Littlewood-Richardson rule for $c^{\nu}_{\lambda \mu}$ using set-valued tableaux.
He related the commutative ring $\Gamma = \bigoplus_{\lambda} \mathbb{Z} \cdot G_{\lambda}$ to the K-theory $K^{\circ} \mathrm{Gr}(k, \mathbb{C}^n)$ of the Grassmannian $\mathrm{Gr}(k, \mathbb{C}^n)$ of $k$-planes in $\mathbb{C}^n$. The Grothendieck group $K^{\circ} \mathrm{Gr}(k, \mathbb{C}^n)$ has a basis $[\mathcal{O}_{\lambda}]$ that are classes of the structure sheaves of the Schubert varieties in $\mathrm{Gr}(k, \mathbb{C}^n)$ indexed by partitions $\lambda$ whose diagram fit in the rectangle $k \times (n - k) = (n-k)^k$. Then, the quotient ring $\Gamma / \langle G_{\lambda}, \lambda \not\subseteq (n-k)^k \rangle$ is isomorphic to $K^{\circ} \mathrm{Gr}(k, \mathbb{C}^n)$ via the map $G_{\lambda} \to [\mathcal{O}_{\lambda}]$; i.e. the structure constants in the Grothendieck ring with the basis $[\mathcal{O}_{\lambda}]$ are the same:
$$
[\mathcal{O}_{\lambda}] \cdot [\mathcal{O}_{\mu}] = \sum_{\nu} (-1)^{|\nu| - |\lambda| - |\mu|} c^{\nu}_{\lambda \mu} [\mathcal{O}_{\nu}].
$$

The basis $\{g_{\lambda} \}$ of $\Lambda$ dual to $\{G_{\lambda} \}$ via the Hall inner product was studied by Lam and Pylyavskyy \cite{lp}. They gave its formula as the generating series  
\begin{equation}
g_{\lambda}(x_1, x_2, \ldots) = \sum_{T} \prod_{i \ge 1} x_i^{\# \{\text{columns of $T$ containing }i\} },
\end{equation}
where the sum runs over {\it reverse plane partitions}, i.e. entries weakly increase in rows and columns, of shape $\lambda$. This basis of {\it dual stable Grothendieck polynomials} agrees with the coproduct $\Delta : \Gamma \to \Gamma \otimes \Gamma$ of $G_{\lambda}$ and addresses K-homology of Grassmannians.

The functions $G_{\lambda}, g_{\lambda}$ are not self-dual under the standard involution $\omega$. This map produces here other symmetric functions $J_{\lambda} = \omega(G_{\lambda}), j_{\lambda} = \omega(g_{\lambda})$. Combinatorially, $J_{\lambda}$ is described using {\it weak set-valued tableaux} (boxes may now contain multisets) and $j_{\lambda}$ using {\it valued-set tableaux} (which are SSYT with a special weighted decomposition) \cite{lp}.

\section{The canonical stable Grothendieck functions $G^{(\alpha, \beta)}_{\lambda}$}

Consider now the ring of symmetric functions $\Lambda$ and its completion $\hat\Lambda$ over $\mathbb{Z}[\alpha, \beta]$ for two generic parameters $\alpha, \beta$. Let $\hat\Lambda_{n}$ be a subring of $\hat\Lambda$ under specialization $x_{n+i} = 0$ for all $i \ge 1$, i.e. with finitely many variables $x_1, \ldots, x_n$. 

\begin{definition}
Let $\lambda$ be a partition. Define the {\it canonical stable Grothendieck function} $G^{(\alpha, \beta)}_{\lambda}(x_1, \ldots, x_n)$ by the formula
\begin{equation}\label{eq1}
G^{(\alpha, \beta)}_{\lambda}(x_1, \ldots, x_n) = \frac{\displaystyle \det \left[\frac{x_i^{\lambda_j + n - j} (1 + \beta x_i)^{j - 1}}{(1-\alpha x_i)^{\lambda_j}} \right]_{1 \le i,j \le n}  } {\displaystyle \det \left[{x_i^{n - j} } \right]_{1 \le i,j \le n} }.
\end{equation}
Here 
$$
\det \left[{x_i^{n - j} } \right]_{1 \le i,j \le n} = \prod_{1 \le i < j \le n} (x_i - x_j)
$$
is the Vandermonde determinant. 
\end{definition}
Since numerator and denominator of the expression above are both skew-symmetric, $G^{(\alpha, \beta)}_{\lambda}(x_1, \ldots, x_n)  \in \hat\Lambda_n$ 
is a well-defined symmetric power series in $x_1, \ldots, x_n$. Note that $G^{(\alpha, \beta)}_{\lambda}(x_1, \ldots, x_n) = 0,$ if $\ell(\lambda) > n$. It is easy to see that stability property $\hat\Lambda_{n+1} \to \hat\Lambda_{n}$ holds:
$$G^{(\alpha, \beta)}_{\lambda}(x_1, \ldots, x_n, 0) = G^{(\alpha, \beta)}_{\lambda}(x_1, \ldots, x_n).$$
Therefore we have an extended symmetric power series $G^{(\alpha, \beta)}_{\lambda}(x_1, x_2, \ldots) \in \hat\Lambda$ in infinitely many variables. From definition we obtain that 
$$
G^{(\alpha, \beta)}_{\lambda} = s_{\lambda} + \text{higher degree terms},
$$
the family $\{G^{(\alpha, \beta)}_{\lambda}\}$ forms a linearly independent set of elements of $\hat\Lambda$ and every element of $\hat\Lambda$ can uniquely be written as an infinite linear combination of these functions.

\begin{example}
We can compute (e.g. by induction on $n \to \infty$) that 
\begin{align}
1 + (\alpha + \beta)G_{(1)}^{(\alpha, \beta)}(x_1, x_2, \ldots) &= \prod_{i \ge 1}\frac{1+\beta x_i}{1 - \alpha x_i}, \label{ex1}\\ 
G_{(1)}^{(\alpha, \beta)}(x_1, x_2, \ldots) &= \sum_{i, j \ge 0} \alpha^i \beta^j s_{(i | j)},
\end{align}
where $(i | j)$ denotes the hook shape partition $(i+1, 1^{j})$.  
\end{example}

To obtain both formulas from the example, one can e.g. prove by induction on $n \to \infty$ and using \eqref{eq1} for $\lambda = (1,0,\ldots)$ that 
$$
1 + (\alpha + \beta)G^{(\alpha, \beta)}_{(1)}(x_1, \ldots, x_{n}, x) = 
\frac{1 + \beta x}{1 - \alpha x} \left( 1 + (\alpha + \beta)G^{(\alpha, \beta)}_{(1)}(x_1, \ldots, x_{n}) \right).
$$ 

As we see, the function $G^{(\alpha, \beta)}_{\lambda}$ is a certain deformation of Schur and stable Grothendieck polynomials. In special cases it corresponds to
\begin{itemize}
\item Schur polynomials $G^{(0, 0)}_{\lambda} = s_{\lambda}$;
\item stable Grothendieck polynomials $G^{(0, -1)}_{\lambda} = G_{\lambda}$, $G^{(0, \beta)}_{\lambda} = G^{\beta}_{\lambda}$.
\end{itemize}

Pivotal properties of this function are the following.

\begin{theorem} \label{omega1} The functions $G^{(\alpha, \beta)}_{\lambda}$ satisfy: 
\begin{itemize}
\item[(i)] self-duality $\omega(G^{(\alpha, \beta)}_{\lambda}) = G^{(\beta, \alpha)}_{\lambda'}$
\item[(ii)] product has the finite decomposition
$$
G^{(\alpha, \beta)}_{\lambda} G^{(\alpha, \beta)}_{\mu} = \sum_{\nu} (\alpha + \beta)^{|\nu| - |\lambda| - |\mu|} c^{\nu}_{\lambda, \mu} G^{(\alpha, \beta)}_{\nu}.
$$
\end{itemize}
\end{theorem}

The duality (i) will be proved later (in Section \ref{dfg}) using the method of Fomin and Greene \cite{fg} on noncommutative Schur functions. Specializing $\alpha + \beta = -1$ in (ii), the multiplicative structure constants of $G^{(\alpha, \beta)}_{\lambda}$ coincide with those of $G^{}_{\lambda}$. This property (ii) will be clear from Proposition \ref{prop1} below.

\subsection{Basic properties of $G^{(\alpha, \beta)}_{\lambda}$}

\begin{proposition}\label{prop1} The following formulas hold
\begin{align}
G^{(\alpha, \beta)}_{\lambda}(x_1, x_2, \ldots) &= G^{(0,\alpha+ \beta)}_{\lambda} \left( \frac{x_1}{1 - \alpha x_1}, \frac{x_2}{1 - \alpha x_2}, \ldots \right),\label{e1} \\
G^{(\alpha, \beta)}_{\lambda}\left( \frac{x_1}{1 - \beta x_1}, \frac{x_2}{1 - \beta x_2}, \ldots \right) &= G^{(\alpha+ \beta, 0)}_{\lambda}(x_1, x_2, \ldots). \label{e2}
\end{align}
\end{proposition}

\begin{proof}
It is straightforward to verify these identities for a finite variable functions by the determinantal formula \eqref{eq1}, and then extend it to infinitely many variables. 
\end{proof}

Note that as a relation between the stable Grothendieck polynomials $G_{\lambda}$ and the weak stable Grothendieck polynomials $J_{\lambda}$, similar formulas were given in \cite{pp}.

\begin{corollary}
For $\alpha + \beta = 0$ we have 
$$
G^{(\alpha, -\alpha)}_{\lambda}(x_1, x_2, \ldots) = s_{\lambda} \left( \frac{x_1}{1 - \alpha x_1}, \frac{x_2}{1 - \alpha x_2}, \ldots \right),
$$
which multiply by the usual Littlewood-Richardson rule.
\end{corollary}

Consider a special case $(\alpha, \beta) \to (\beta/2, \beta/2).$ We know that the structure constants $c^{\nu}_{\lambda \mu}$ satisfy the symmetry $c^{\nu}_{\lambda \mu} = c^{\nu'}_{\lambda' \mu'}$, and thus the linear map $\tau : \hat\Lambda \to \hat\Lambda$ given by $\tau(G^{\beta}_{\lambda}) = G^{\beta}_{\lambda'}$ is a ring homomorphism.

The {function} $\widetilde{G}^\beta_{\lambda} := G^{(\beta/2, \beta/2)}_{\lambda}$ is self-dual under the standard involution,
$\omega(\widetilde{G}^\beta_{\lambda}) = \widetilde{G}^\beta_{\lambda'}.$ 
Hence combining this with Proposition \ref{prop1}, the map $\phi : \hat\Lambda \to \hat\Lambda$ which sends $G^{\beta}_{\lambda} \mapsto \widetilde{G}^\beta_{\lambda}$ is an automorphism given explicitly by the substitution of variables $x_i \mapsto \frac{x_i}{1 - \frac{\beta}{2} x_i}$ and we have $\omega = \phi \tau \phi^{-1}$, thus confirming Theorem \ref{t1} for $\beta = -1$.

\begin{remark}
The symmetry of $c^{\nu}_{\lambda \mu}$ also implies that there is another automorphism which maps $G^{(\alpha, \beta)}_{\lambda} \mapsto G^{(\alpha, \beta)}_{\lambda'}$, and it coincides with the standard involution $\omega$ for $\alpha = \beta$.
\end{remark}

\subsection{The elements $G^{(\alpha, \beta)}_{(k)}, G^{(\alpha, \beta)}_{(1^k)}$}
Define the generating series
\begin{align}
H^{(\alpha, \beta)}(x;t) &:= 1 + (\alpha + \beta + t) \sum_{k \ge 1} G^{(\alpha, \beta)}_{(k)} t^{k - 1},\\
E^{(\alpha, \beta)}(x;t) &:= 1 + (\alpha + \beta + t) \sum_{k \ge 1} G^{(\alpha, \beta)}_{(1^k)} t^{k - 1}.
\end{align}

\begin{proposition}
We have
\begin{align}
H^{(\alpha, \beta)}(x;t) &= \prod_{j \ge 1} \frac{1 + \beta x_j}{1 - (\alpha + t) x_j}, \qquad
E^{(\alpha, \beta)}(x;t) =\prod_{j \ge 1} \frac{1 + (\beta + t) x_j}{1 - \alpha x_j}.
\end{align}
\end{proposition}

\begin{proof}
It is known that (e.g. \cite{lenart}) 
\begin{align}
G^\beta_{(k)} = \sum_{i \ge 0} \beta^i s_{(k-1 | i)}, \qquad G^\beta_{(1^k)} = \sum_{i \ge 0} \beta^i \binom{i + k - 1}{i} e_{i + k}.
\end{align}
The needed identities can be derived by standard manipulations and the substitutions $\beta \to \alpha + \beta$, $x_j \to \frac{x_j}{1 - \alpha x_j}$. 
\end{proof}

We now give some variations on these series. 
First note that using equation \eqref{ex1} we obtain
\begin{align}
 \sum_{k \ge 1} G^{(\alpha, \beta)}_{(k)} t^{k - 1} &= G^{(\alpha + t, \beta)}_{(1)},\qquad \sum_{k \ge 1} G^{(\alpha, \beta)}_{(1^k)} t^{k - 1} = G^{(\alpha, \beta + t)}_{(1)}.\end{align}
Let us define
\begin{align}
h^{(\alpha, \beta)}_0 &:= 1 + (\alpha + \beta) G^{(\alpha, \beta)}_{(1)}, \quad h^{(\alpha, \beta)}_k := G^{(\alpha, \beta)}_{(k)} + (\alpha + \beta) G^{(\alpha, \beta)}_{(k+1)}, \quad k > 0,\label{a13}\\
e^{(\alpha, \beta)}_0 &:= 1 + (\alpha + \beta) G^{(\alpha, \beta)}_{(1)}, \quad e^{(\alpha, \beta)}_k := G^{(\alpha, \beta)}_{(1^k)} + (\alpha + \beta) G^{(\alpha, \beta)}_{(1^{k+1})}, \quad k > 0,\label{b13}
\end{align}
or in other words,
\begin{align}
H^{(\alpha, \beta)}(x;t) &= \sum_{k \ge 0} h^{(\alpha, \beta)}_k t^k,\qquad E^{(\alpha, \beta)}(x;t) = \sum_{k \ge 0} e^{(\alpha, \beta)}_k t^k.
\end{align}
It is easy to show that
\begin{align}\label{14}
{h^{(\alpha, \beta)}_{k}}/{h^{(\alpha, \beta)}_{0}} = h_k\left(\frac{x}{1 - \alpha x} \right), \qquad {e^{(\alpha, \beta)}_{k}}/{e^{(\alpha, \beta)}_{0}} = e_k\left(\frac{x}{1 + \beta x} \right).
\end{align}
We have
$$
E^{(\alpha, \beta)}(t) H^{(-\beta, -\alpha)}(-t) = 1,
$$
which implies the following relation between the elements $e^{(\alpha, \beta)}_{k}, h^{(-\beta, -\alpha)}_{\ell}$:
\begin{align}
\sum_{i} (-1)^i e^{(\alpha, \beta)}_{i} h^{(-\beta, -\alpha)}_{n - i} = \delta_{n,0}.
\end{align}
In particular,
\begin{align}
e^{(\alpha, \beta)}_{0} h^{(-\beta, -\alpha)}_{0} = \left(1 + (\alpha + \beta) G^{(\alpha, \beta)}_{(1)} \right) \left(1 - (\alpha + \beta) G^{(-\beta, -\alpha)}_{(1)} \right) = 1.
\end{align}

\begin{proposition}
Each of the following four families is algebraically independent: \\
 $\{G^{(\alpha, \beta)}_{(k)} | k \in \mathbb{Z}_{> 0}\}$,
 $\{G^{(\alpha, \beta)}_{(1^k)} | k \in \mathbb{Z}_{> 0} \}$, and for $\alpha + \beta \ne 0$,
 $\{h^{(\alpha, \beta)}_{k} | k \in \mathbb{Z}_{\ge 0}\}$,
 $\{e^{(\alpha, \beta)}_{k} | k \in \mathbb{Z}_{\ge 0}\}$.
\end{proposition}

\begin{proof}
If there is a relation between the elements $h^{(\alpha, \beta)}_{k}$ (or $e^{(\alpha, \beta)}_{k}$), then by \eqref{14} there is a relation between the elements $h_i$ (or $e_i$) which is not true. Similarly, if there is a relation between $G^{(\alpha, \beta)}_{(k)}$ (or $G^{(\alpha, \beta)}_{(1^k)}$), then by definitions \eqref{a13}, \eqref{b13} there is a relation between 
$h^{(\alpha, \beta)}_{i}$ (or $e^{(\alpha, \beta)}_{i}$).
\end{proof}

\section{Hook-valued tableaux}\label{hvt}

In this section we describe the functions $G^{(\alpha, \beta)}_{\lambda}$ combinatorially using {\it hook-valued tableaux}. 

A Young diagram is called a {\it hook} if it has the form $(a | b) = (a+1, 1^b)$ for some $a, b \ge 0$. In this case, $a$ is called the {\it arm} of the hook and $b$ is called the {\it leg} of the hook. 

We now present a generalization of SSYT where boxes contain tableaux of hook shapes. 

Let $\max(T)$ (resp. $\min(T)$) of a tableau $T$ be the maximal (resp. minimal) number contained in $T$. For two arbitrary tableaux $T_1, T_2$ define the relations $T_1 \le T_2$ (resp. $T_1 < T_2$) if $\max(T_1) \le \min(T_2)$ (resp. $\max(T_1) < \min(T_2)$). So with these orders we can say that a sequence of nonempty tableaux is (weakly) increasing. 
 
\begin{definition}
A {\it hook-valued tableau} of shape $\lambda$ is a filling of the Young diagram of $\lambda$ with SSYTs (instead of numbers) satisfying the following properties:
\begin{itemize}
\item[(1)] each box contains one SSYT of hook shape;
\item[(2)] the hooks inside boxes weakly increase from left to right in rows and strictly increase from top to bottom in columns (with the orders defined above).
\end{itemize}
\end{definition}

An example of such tableau is given in Figure \ref{fig1}. Let $a(T)$ and $b(T)$ be the sums of all arms and legs, respectively, of hooks in $T$.  The {\it weight} of $T$ is then defined as $w_T(\alpha, \beta) = \alpha^{a(T)} \beta^{b(T)}$ and the monomial $x^T = \prod_i x_i^{a_i}$ where $a_i$ is the total number of occurrences of $i$ in $T$. Let $HT(\lambda)$ be the set of hook-valued tableaux of shape $\lambda$. 

This setting generalizes (and combines) the notions of set-valued tableaux \cite{buch} (set is a single column hook) and weak set-valued tableaux \cite{lp} (multiset is a single row hook), see Figure \ref{fig1x}. 

\begin{figure}
\ytableausetup{boxsize=1.2cm} 
{\scriptsize
\begin{ytableau}
\begin{tabular}{l} 1 1 2\\ 2\\3 \end{tabular} & \begin{tabular}{l} 3 4\\ 4 \end{tabular} & \begin{tabular}{l} 4\\ 5\\7 \end{tabular} \\
\begin{tabular}{l} 4 4 4\\ 5 \end{tabular} & 5\ 5\ 7
\end{ytableau}
}
\caption{A hook-valued tableau $T$ of shape $(3,2)$ with $w_T(\alpha, \beta) = \alpha^7\beta^6$ and $x^T = x_1^2 x_2^2 x_3^2 x_4^6 x_5^4 x_7^2$. } \label{fig1}
\end{figure}

\begin{figure}
\ytableausetup{boxsize=1.2cm} 
{\scriptsize
\begin{ytableau}
\begin{tabular}{l} 1\\ 2\\3 \end{tabular} & \begin{tabular}{l} 3\\ 4 \end{tabular} & \begin{tabular}{l} 4\\ 5\\7 \end{tabular} \\
\begin{tabular}{l} 4\\ 5 \end{tabular} & 5
\end{ytableau}
}
\qquad {\scriptsize
\begin{ytableau}
\begin{tabular}{l} 1 1 2 \end{tabular} & \begin{tabular}{l} 3 4 \end{tabular} & \begin{tabular}{l} 4 \end{tabular} \\
\begin{tabular}{l} 4 4 4 \end{tabular} & 5\ 5\ 7
\end{ytableau}
}
\caption{Examples of set-valued (the case $\alpha = 0$) and weak set-valued tableaux (the case $\beta = 0$). } \label{fig1x}
\end{figure}

\begin{theorem}\label{hook}The following formula holds
\begin{equation}
G^{(\alpha, \beta)}_{\lambda}(x_1, x_2, \ldots) = \sum_{T \in HT(\lambda)} w_T(\alpha, \beta) x^{T}.
\end{equation}
\end{theorem}

\begin{proof}
We use set-valued tableaux formula \eqref{svt} for $G^{\alpha + \beta}_{\lambda} = G^{(0, \alpha + \beta)}_{\lambda}$ and via the relation $$G^{(0, \alpha + \beta)}_{\lambda}\left(\frac{x_1}{1 - \alpha x_1}, \ldots \right) = G_{\lambda}^{(\alpha, \beta)}(x_1, \ldots)$$ from Proposition \ref{prop1} obtain hook-valued interpretation. We have
$$
G^{(0, \alpha + \beta)}_{\lambda}\left(\frac{x_1}{1 - \alpha x_1}, \ldots \right) = \sum_{T \in SVT(\lambda)} (\alpha + \beta)^{|T| - |\lambda|} \left(\frac{x}{1 - \alpha x}\right)^T,
$$
where $SVT(\lambda)$ is the set of set-valued tableaux of shape $\lambda$. Let $T \in SVT(\lambda)$. Imagine a set in a box of $T$ as a single column, where any element starting from the second row has the weight $(\alpha + \beta)$. Each element $i \in T$ has an expanded contribution $x_i/(1 - \alpha x_i) = x_i + \alpha x_i^2 + \alpha^2 x_i^3 + \cdots$ to the function and we may rewrite the last sum as
$$
G^{(\alpha, \beta)}_{\lambda} = G^{(0, \alpha + \beta)}_{\lambda}\left(\frac{x_1}{1 - \alpha x_1}, \ldots \right) = \sum_{T \in MSVT(\lambda)} (\alpha + \beta)^{|T| - |\lambda|} w_T(\alpha) x^T,
$$
where $MSVT(\lambda)$ is the set of {\it multiset-valued tableaux} (sets are now replaced by multisets) and $w_T(\alpha) = \alpha^{a_T}$ where $a_T$ is the total number of {\it extra} copies of elements in all boxes of $T$. 

We now establish a weight-preserving bijection with the hook-valued tableaux. Suppose we have a multiset-valued tableau with weights $(\alpha + \beta)$ for each (non-first) distinct elements in its box and each element $i$ has $a_i$ copies contributing the weight $\alpha^{a_i - 1}$. To create a hook, for each element $i$ we do the following: 
\begin{itemize}
\item[(i)] if $i$ is the first element in its box, put copies of $i$ in the first row (with weights $\alpha$);
\item[(ii)] otherwise there are two options:
\begin{itemize}
\item[(a)] put $i$ in the first column with the weight $\beta$ {and} put its extra copies to the first row with weights $\alpha$;

 {or}
 
\item[(b)] put $i$ with its copies in the first row with weights $\alpha$. 
\end{itemize}
\end{itemize}

It is easy to see that the procedure is reversible. Notice that the hooks inside the boxes created by the rules (i), (ii) (a), (b) weakly increase in rows and strictly increase in columns, and hence we obtain a hook-valued tableaux.
\end{proof}

By this combinatorial formula we can extend the Grothendieck functions to any skew-shape. However the way we define $G^{(\alpha, \beta)}_{\lambda/\mu}$ in the next section is based on noncommutative operators and it differs from this combinatorial formula by having different boundary conditions, we allow to put tableau elements in a boundary of $\mu$ outside of the skew shape $\lambda/\mu$. 

\section{Duality for $G^{(\alpha, \beta)}_{\lambda}$} \label{dfg}

In this section we prove the duality $\omega(G^{(\alpha, \beta)}_{\lambda}) = G^{(\beta, \alpha)}_{\lambda'}.$ We describe noncommutative operators for canonical stable Grothendieck functions based on Schur operators. With this approach we then define the functions $G^{(\alpha, \beta)}_{\lambda/\mu}$ indexed by skew shapes.   

\subsection{Noncommutative Schur functions}
We use the theory of noncommutative Schur functions developed by Fomin and Greene \cite{fg}. Refer to \cite{bf} for a more general context and review of the method. 

For a given partition $\mu$ consider the free $\mathbb{Z}[\alpha, \beta]$-module $\mathbb{Z}_\mu = \bigoplus_{\mu \subset \lambda} \mathbb{Z}[\alpha, \beta] \cdot \lambda$ of all partitions that contain $\mu$. 

Given a set $u = (u_1, \ldots, u_N)$ of linear operators $u_i : \mathbb{Z}_{\mu} \to \mathbb{Z}_{\mu}$, consider the (noncommutative) ring  $K\langle u_1, \ldots, u_N\rangle$ (over $K=\mathbb{Z}[\alpha, \beta]$ or $\mathbb{Z}$)
generated by the variables $u$. The elementary symmetric functions $e_k(u)$ and the complete homogeneous symmetric functions $h_k(u)$ ($k \ge 0$) on $u$ are defined as follows
$$
e_k(u)  := \sum_{N \ge i_1 > \ldots > i_k \ge 1} u_{i_1} \ldots u_{i_k}, \qquad h_k(u)  := \sum_{1 \le i_1 \le \ldots \le i_k \le N} u_{i_1} \ldots u_{i_k}.
$$
Then the noncommutative Schur functions $s_{\lambda}(u)$ can be defined via the Jacobi-Trudi identity\footnote{There is another way of defining $s_{\lambda}(u)$, directly converting SSYT into monomials consisting of the $u$ variables.} $s_{\lambda} = \det[e_{\lambda'_i - i + j}]$,
$$
s_{\lambda}(u) := \sum_{\sigma\in S_{\ell = \ell(\lambda')}} \sgn(\sigma) e_{\lambda'_1 + \sigma(1) - 1}(u) \ldots e_{\lambda'_\ell + \sigma(\ell) - \ell}(u).
$$

Denote $[a,b] = ab - ba$ the commutator. Let now $u = (u_1, \ldots, u_N )$ be a set of linear operators $u_i : \mathbb{Z}_{\mu} \to \mathbb{Z}_{\mu}$ satisfying the following commutation relations\footnote{The first relations can be changed to the non-local Knuth relations: $u_i u_k u_j = u_k u_i u_j, i \le j < k, |i - k| \ge 2$ and $u_j u_i u_k = u_j u_k u_i, i < j \le k, |i - k| \ge 2$ (which we do not use for our purposes).}:
\begin{equation}\label{com1}
[u_j, u_i] = 0, \quad |i- j| \ge 2;
\end{equation}
\begin{equation}\label{comx}
[u_{i+1} u_i, u_i + u_{i+1}] = 0.
\end{equation}
If these relations are satisfied, it is known that the noncommutative Schur function $s_{\lambda}(u)$ behaves like the usual Schur function \cite{fg}. In particular, the following properties hold. The noncommutative versions of symmetric functions commute:
$$
[e_{i}(u), e_{j}(u)] = [h_{i}(u), h_j(u)] = [s_{\lambda}(u), s_{\mu}(u)] = 0, \quad \forall\ i, j, \lambda, \mu
$$
as well as the series $A(x), B(x)$ defined (for a single variable $x$ commuting with the $u$) by
$$
A(x) := \cdots (1 + x u_2) (1 + x u_1), \quad B(x) := 1/(1 - x u_1) 1/(1 - x u_2) \cdots, 
$$
$$
[A(x), A(y)] = 0, \quad [B(x), B(y)] = 0, \quad [A(x), B(y)] = 0,
$$
and the noncommutative analogs of the Cauchy identities hold:
\begin{align}\label{cauchy}
\cdots A(x_2)A(x_1) &= \sum_{\lambda} s_{\lambda'}(x_1, x_2, \ldots) s_{\lambda}(u), \\ \cdots B(x_2)B(x_1) &= \sum_{\lambda} s_{\lambda}(x_1, x_2, \ldots) s_{\lambda}(u).
\end{align}

\subsection{Schur operators}
\begin{definition}[Schur operators]\label{schuro}
Define the linear operators $u_i, d_i : \mathbb{Z}_{\mu} \to \mathbb{Z}_{\mu}$, $i \in \mathbb{Z}_{>0}$ which act on bases as follows:
$$
u_i \cdot \lambda = \begin{cases}
\lambda \cup \text{box in $i$th column}, & \text{ if possible,}\\
0, & \text{ otherwise};
\end{cases}
$$
$$
d_i \cdot \lambda = \begin{cases}
\lambda - \text{box in $i$th column}, & \text{ if possible,}\\
0, & \text{ otherwise.}
\end{cases}
$$
\end{definition}

It is known \cite{fomin, fg} that both operators $u, d$ satisfy the relations \eqref{com1}; 
the following {\it local Knuth} relations (which sum to \eqref{comx})
\begin{align}
\label{com2} u_{i+1} u_i u_i = u_i u_{i+1} u_{i}, \quad u_{i+1} u_i u_{i+1} = u_{i+1} u_{i+1} u_i;
\end{align}
and the {duality} (or conjugate) relations \cite{fomin}
\begin{align}\label{1}
[d_j, u_i] = 0\ (i\ne j), \quad
d_{i+1} u_{i+1} = u_i d_i\ (i \in \mathbb{Z}_{>0}), \quad
d_1 u_1 = 1.
\end{align}
These operators build Schur functions by its tableau interpretation (e.g., \cite[Example 2.4]{fg}). 

Note that for each $i \in \mathbb{Z}_{>0},$ the operator $u_i d_i$ simply gives $1$ (identity) if the box on $i$th column is removable, and $0$ otherwise.
The elements $\{u_i d_i\}$ commute, it is easy to see that 
\begin{equation}\label{c1}
[u_i d_i, u_{i-1}d_{i-1}] = [d_{i+1}u_{i+1}, u_{i-1} d_{i-1}]= 0.
\end{equation}
Moreover, we also have 
\begin{equation}\label{c2}
[u_id_i, d_iu_i] = [u_id_i, u_{i-1}d_{i-1}] = 0.
\end{equation}

The relations \eqref{c1}, \eqref{c2} follow from \eqref{1}. We use one more type of relations (that can easily be checked on bases)
\begin{equation}\label{c3}
[u_i d_i, u_{i+1}u_i] = 0.
\end{equation}

\subsection{Operators for $G^{(\alpha, \beta)}_{\lambda}$} 
Consider now the set $u^{(\alpha, \beta)}$ of linear operators defined (on the same spaces) using Schur operators:
\begin{equation}
u^{(\alpha, \beta)}_i := u_i (1 + (\alpha + \beta) d_i) - \alpha = u_i + (\alpha + \beta)u_i d_i - \alpha, \quad i \in \mathbb{Z}_{>0}
\end{equation}

To see effect of these operators, consider the case $\alpha = 0,$ i.e. $u^{(0, \beta)}_i = u_i (1 + \beta d_i) = u_i + \beta u_i d_i$ which means the following for a diagram. If possible, it adds a single box on $i$th column; {or} just multiplies by a scalar parameter $\beta$ if the box on $i$th column is removable (without removing it). This procedure allows to construct set-valued tableaux with a parameter $\beta$. 

As we will show, these deformations of Schur operators build the functions $G^{(\alpha, \beta)}_{\lambda/\mu}$. Using the relations \eqref{com2}--\eqref{c3} for the operators $u,d$, the following statement is a straightforward check. 
\begin{lemma}
The operators $u^{(\alpha, \beta)}$ satisfy the relations \eqref{com1}, \eqref{comx}. 
\end{lemma}

Define the series (the set $u^{(\alpha, \beta)}$ is finite)
\begin{align*}
C(x) &:= \cdots \left(\frac{1 + x u^{(\alpha, \beta)}_2}{1 - \alpha x} \right) \left(\frac{1 + x  u^{(\alpha, \beta)}_1}{1 - \alpha x} \right), \quad
D(x) := \left(\frac{1 + \alpha x}{1 - x u^{(\alpha, \beta)}_1} \right) \left(\frac{1 + \alpha x}{1 - x u^{(\alpha, \beta)}_2} \right) \cdots 
\end{align*}
Then from the Lemma above we obtain 
\begin{equation}
[C(x), C(y)] = 0, \quad [D(x), D(y)] = 0.
\end{equation}

Let $\langle\cdot, \cdot \rangle: \mathbb{Z}_\mu \times \mathbb{Z}_\mu \to \mathbb{Z}[\alpha, \beta]$ be a (non-degenerate) $\mathbb{Z}[\alpha, \beta]$-bilinear pairing defined on bases by $\langle \lambda, \nu \rangle = \delta_{\lambda, \nu}$.

\begin{theorem}
We have
$$
\langle \cdots C(x_2) C(x_1) \cdot \varnothing, \lambda \rangle = G^{(\alpha, \beta)}_{\lambda}, \quad \langle \cdots D(x_2) D(x_1) \cdot \varnothing, \lambda \rangle = G^{(\beta, \alpha)}_{\lambda'}.
$$
\end{theorem}

\begin{proof}
We can rewrite
\begin{align*}
C(x) &= \overleftarrow{\prod_{i \ge 1}} \left(1 + \frac{x}{1 - \alpha x} u_i(1 + (\alpha + \beta) d_i) \right) = \overleftarrow{\prod_{i \ge 1}} \left(1 + (x + \alpha x^2 + \cdots) u_i(1 + (\alpha + \beta) d_i) \right).
\end{align*}
For each particular $x_k$ taken from the product, the term $\alpha^\ell x_k^{\ell+1} (u_i+(\alpha + \beta)u_id_i)$ means the following procedure of building the hook-valued tableau:
\begin{itemize}
\item[(a)] we add a new box on $i$th column (if possible) and put $\ell + 1$ copies of $k$ in a row inside this box (each copy except the first has weight $\alpha$);

\item[(b)] or if the last box in $i$th column is removable, then (applying $u_id_i$ means that the shape does not change) add $\ell$ copies of $k$ in a row of the hook inside this box and then add the remaining one copy of $k$ to either first row (with weight $\alpha$) or first column (with weight $\beta$).
\end{itemize}

On can see from this procedure and order of operators, that we have inequalities exactly as in hook-valued tableaux: hooks inside the boxes weakly increase in rows and strictly increase in columns. Therefore, applying the operator series 
until we obtain $\lambda$ gives the symmetric function $G^{(\alpha, \beta)}_{\lambda}$, or $\langle \cdots C(x_2) C(x_1) \cdot \varnothing, \lambda \rangle = G^{(\alpha, \beta)}_{\lambda}$.

For the second equality involving the $D$ series, we rewrite it as follows
$$
D(x) = \overrightarrow{\prod_{i \ge 1}} \left(\frac{1}{1 - \frac{x}{1 + \alpha x} u_i(1 + (\alpha + \beta) d_i)} \right).
$$
Using a similar reasoning it is not hard to see that for the series given by
$$
\widetilde{D}(t) := \overrightarrow{\prod_{i \ge 1}} \left(\frac{1}{1 - t u_i(1 + (\alpha + \beta) d_i)} \right)
$$
we have 
$$
\langle \cdots \widetilde{D}(t_2) \widetilde{D}(t_1) \cdot \varnothing, \lambda \rangle = G_{\lambda'}^{(\alpha + \beta, 0)}(t_1, t_2, \ldots ).
$$
Finally note that for $t_i = x_i/(1 + \alpha x_i)$ by Proposition \eqref{prop1} we have
$$
G_{\lambda'}^{(\alpha + \beta, 0)}(t_1, t_2, \ldots ) = G^{(\beta, \alpha)}_{\lambda'}(x_1, x_2, \ldots).
$$
\end{proof}

\begin{theorem} We have
$\omega(G^{(\alpha, \beta)}_{\lambda}) = G^{(\beta, \alpha)}_{\lambda'}.$
\end{theorem}

\begin{proof}
Let $N = \# u^{(\alpha, \beta)}$. From the previous Theorem applying the noncommutative Cauchy identities \eqref{cauchy} to the series $C, D$ we have
\begin{align*}
\omega(G^{(\alpha, \beta)}_{\lambda}) &= \omega(\langle \cdots C(x_2) C(x_1) \cdot \varnothing, \lambda \rangle)\\
&= \omega \left\langle \prod_i {(1 - \alpha x_i)^{-N}} \sum_{\nu}s_{\nu'}(x) s_{\nu}(u^{(\alpha, \beta)}) \cdot\varnothing,\lambda\right\rangle\\
&= \sum_{\nu} \omega(\prod_i {(1 - \alpha x_i)^{-N}} s_{\nu'}(x)) \left\langle s_{\nu}(u^{(\alpha, \beta)}) \cdot\varnothing,\lambda\right\rangle\\
&= \sum_{\nu} \prod_i {(1 + \alpha x_i)^N} s_{\nu}(x) \left\langle s_{\nu}(u^{(\alpha, \beta)}) \cdot\varnothing,\lambda\right\rangle\\
&=  \left\langle \prod_i {(1 + \alpha x_i)^N} \sum_{\nu} s_{\nu}(x) s_{\nu}(u^{(\alpha, \beta)}) \cdot\varnothing,\lambda\right\rangle\\
&= \langle \cdots D(x_2) D(x_1) \cdot \varnothing, \lambda \rangle\\
&= G^{(\beta, \alpha)}_{\lambda'}.
\end{align*}
\end{proof}

\subsection{Skew shapes}

For skew shapes, we can define the symmetric functions $G^{(\alpha, \beta)}_{\lambda/\mu}$ as 
\begin{equation}\label{gskew}
G^{(\alpha, \beta)}_{\lambda/\mu} := \langle \cdots C(x_2) C(x_1) \cdot \mu, \lambda \rangle,
\end{equation}
for which we may similarly obtain that
\begin{equation}
G^{(\beta, \alpha)}_{\lambda'/\mu'} = \langle \cdots D(x_2) D(x_1) \cdot \mu, \lambda \rangle
\end{equation}
and hence the duality $\omega(G^{(\alpha, \beta)}_{\lambda/\mu}) = G^{(\beta, \alpha)}_{\lambda'/\mu'}.$

This definition of $G^{(\alpha, \beta)}_{\lambda/\mu}$ for skew shapes does not match exactly the combinatorial definition (i.e. if we extend directly hook-valued tableaux for skew shapes) as it has different boundary conditions. For example, for a single variable $t$
\begin{align*}
G^{(\alpha, \beta)}_{(2)/(1)}(t) &=\langle \overleftarrow{\prod_{i \ge 1}} \left(1 + \frac{t}{1 - \alpha t} u_i(1 + (\alpha + \beta) d_i) \right)  \cdot (1), (2)\rangle \\
&=\left(1 +  \frac{(\alpha + \beta) t}{1 - \alpha t} \right) \frac{t}{1 - \alpha t} \\
&= \frac{1 + \beta t}{1 - \alpha t} \frac{t}{1 - \alpha t},
\end{align*} 
(the first factor comes from applying $u_1 d_1$ on the shape $(1)$) whereas it should be just $ \frac{t}{1 - \alpha t}$ if we compute hook-valued tableaux in the skew shape $(2)/(1)$.  For $(\alpha, \beta) = (0, -1)$ it corresponds to the function denoted as $G_{\lambda /\!\!/ \mu}$ in \cite{buch}. 

\subsection{Remarks}

\begin{remark}
For $\alpha = \beta = 0$, the operators ${u^{(\alpha, \beta)}}$ reduce to the usual Schur operators, which build Schur functions. For $\alpha + \beta = 0$, we obtain the deformations of Schur functions $s_{\lambda}(x/(1- \alpha x))$ and $\omega s_{\lambda}(x/(1- \alpha x)) = s_{\lambda'}(x/(1 + \alpha x)).$ The case $(\alpha, \beta) = (0, \pm 1)$ or $(\pm 1,0)$ corresponds to the usual stable Grothendieck polynomials $G_{\lambda}$ and its image under $\omega$, the weak stable Grothendieck polynomial $J_{\lambda}$ \cite{lp}.  
\end{remark}

\begin{remark}[On diagonal adding operators]
The functions $G^{(\alpha, \beta)}_{\lambda}$ can also be built using another types of operators $v_i$ which add boxes by diagonals, an approach used in \cite{buch, lp} for $G_{\lambda}$. 
Let us recall these operators. The box $(i,j)$ of the Young diagram of $\lambda$ lies on the $(j-i)$th diagonal. Say that $\lambda$ has an {\it inner corner} on $i$th diagonal if there is a partition $\mu$ so that $\lambda/\mu$ is a single box on $i$th diagonal. Similarly, $\lambda$ has an {\it outer corner} on $i$th diagonal if there is a partition $\mu$ so that $\mu/\lambda$ is a single box on the same diagonal. 
Define the linear operators $v_{i} : \mathbb{Z}_\mu \to \mathbb{Z}_\mu$ ($i\in \mathbb{Z}$) as follows
$$
v_i \cdot \lambda = 	\begin{cases}
					\nu,  & \text{ if } \lambda \text{ has an outer corner } \nu/\lambda \text{ on } i\text{th diagonal;}\\
					\lambda, & \text{ if } \lambda \text{ has an inner corner not contained in } \mu \text{ on $i$th diagonal;}\\
					0, & \text{ otherwise}.
				\end{cases}
$$
These operators $v_i$ satisfy the following relations: 
$
v_i^2 = v_i,$
$v_{i} v_{i+1}v_{i} = v_{i+1} v_{i} v_{i+1},$
$v_{i} v_{j} = v_{j} v_{i}, |i-j| > 1, i,j \in \mathbb{Z}$  \cite{lp, buch}.
To obtain similar properties for the functions $G^{(\alpha, \beta)}_{\lambda/\mu}$ one could play with the series
$$
E^{}(x) := \cdots  \left(\frac{1+\beta x v_1}{1 - \alpha x v_1}\right) \left(\frac{1+\beta x v_0}{1 - \alpha x v_0}\right) \left(\frac{1+\beta x v_{-1}}{1 - \alpha x v_{-1}}\right) \cdots
$$
Let $v = v_i$ for any $i$. From the fact that $v^2 = v$ it is easy to see that
\begin{equation*}\label{sw}
\frac{1 + \beta x v}{1 - \alpha x v} = \frac{1 + ((\alpha + \beta)v - \alpha)x}{1 - \alpha x}.
\end{equation*}
Hence we apply the transformation $v_{i} \to v_i' = (\alpha + \beta) v_{i} - \alpha$ and use the theory of noncommutative Schur functions. Note that the operators $v'_i$ satisfy the properties of the (generalized) Hecke algebra: 
$v'^2_i = (\beta - \alpha)v'_i + \alpha \beta,$ $v'_i v'_{i+1} v'_i = v'_{i+1} v'_{i} v'_{i+1}$, and $v'_{i} v'_{j} = v'_{j} v'_{i}$ for $|i - j| > 1$.
So Grothendieck functions, like Schur functions, can be build using both types of operators. In fact, similarly as we used Schur operators, the operators $v_i$ can be constructed via diagonal Schur operators as follows: $v_i = \bar v_i(1 + \bar d_i)$ where $\bar v_i$ adds a single box on $i$th diagonal if possible and returns $0$ otherwise; $\bar d_i$ removes a single box on $i$th diagonal if possible and returns $0$ otherwise.
\end{remark}

\begin{remark}
It is not hard to prove the following result:
Let $\phi_{\alpha} : \hat\Lambda \to \hat\Lambda$ be an automorphism given by $\phi_{\alpha} f(x) = f(x/(1 - \alpha x))$.  Then, $\phi_{\alpha} \omega \phi_{\alpha} = \omega$ or equivalently $\omega \phi_{\alpha} =\phi_{\alpha}^{-1} \omega$. In other words, let $f,g \in \hat\Lambda$ be symmetric power series satisfying $\omega f(x) = g(x).$ Then, $\omega f({x}/{(1 - \alpha x)}) = g({x}/{(1 + \alpha x)}).$ 
\end{remark}

\section{The dual canonical Grothendieck polynomials $g^{(\alpha, \beta)}_{\lambda}$}

Recall that  
$$
G^{(\alpha, \beta)}_{\lambda} = s_{\lambda} + \text{ higher degree terms}.
$$   
We now want to introduce the dual family for $\{G^{(\alpha, \beta)}_{\lambda} \}$ via the Hall inner product $\langle \cdot, \cdot \rangle : \Lambda \times \hat\Lambda \to \mathbb{Z}[\alpha, \beta]$ defined on Schur basis by $\langle s_{\lambda}, s_{\mu} \rangle = \delta_{\lambda, \mu}$. 

\begin{definition}
The {\it dual canonical stable Grothendieck polynomials} $g^{(\alpha, \beta)}_{\lambda}(x_1, x_2, \ldots)$ is the dual basis for $G^{(-\alpha, -\beta)}_{\lambda}$ via the Hall inner product, $\langle g^{(\alpha, \beta)}_{\lambda}, G^{(-\alpha, -\beta)}_{\mu} \rangle = \delta_{\lambda, \mu}$. 
\end{definition}

It is a priori clear that $g^{(\alpha, \beta)}_{\lambda} \in \Lambda$ is a symmetric function satisfying the duality $\omega(g^{(\alpha, \beta)}_{\lambda}) = g^{(\beta, \alpha)}_{\lambda'}$. We have 
$$
g^{(\alpha, \beta)}_{\lambda} = s_{\lambda} + \text{ lower degree terms}
$$
and for a finite number of variables, $g^{(\alpha, \beta)}_{\lambda}(x_1, \ldots, x_n)$ is a non-homegeneous symmetric polynomial. Later we will see that the function $g^{(\alpha, \beta)}_{\lambda}$ has a nice combinatorial description and that its expansion in the Schur basis is positive (i.e. exchange constants are polynomials in $\alpha, \beta$ with positive integer coefficients). 

The polynomials $g^{(\alpha, \beta)}_{\lambda}$ combine the dual stable Grothendieck polynomials $g^{(0, 1)}_{\lambda} = g_{\lambda}$, $g^{(1, 0)}_{\lambda} = j_{\lambda'} = \omega(g_{\lambda'})$ given in \cite{lp}; note that $g^{(0,0)}_{\lambda} = s_{\lambda}$. 

Another equivalent description is via the Cauchy identity
\begin{equation}
\sum_{\lambda} g^{(\alpha, \beta)}_{\lambda}(x_1, x_2, \ldots) G^{(-\alpha, -\beta)}_{\lambda}(y_1, y_2, \ldots) = \prod_{i,j} \frac{1}{1 - x_i y_j}. 
\end{equation}

Let $\Gamma^{(\alpha, \beta)} = \bigoplus_{\lambda} \mathbb{Z}[\alpha, \beta] \cdot G^{(\alpha, \beta)}_{\lambda}$. The polynomials $g^{(\alpha, \beta)}_{\lambda}$ have multiplicative structure constants as by the comultiplication $\Delta : \Gamma^{(-\alpha, -\beta)} \to \Gamma^{(-\alpha, -\beta)} \otimes \Gamma^{(-\alpha, -\beta)}$ given by 
\begin{align}
\Delta(G^{(-\alpha, -\beta)}_{\nu}) &= \sum_{\lambda, \mu} (-\alpha - \beta)^{|\nu| - |\lambda| - |\mu|} d^{\nu}_{\lambda \mu} G^{(-\alpha, -\beta)}_{\lambda} \otimes G^{(-\alpha, -\beta)}_{\mu}.
\end{align}

In the case $(\alpha, \beta) \to (\beta/2, \beta/2)$ we obtain the polynomials $\widetilde{g}^\beta_{\lambda} := g^{(\beta/2, \beta/2)}_{\lambda}$ which form a basis of $\Lambda$ satisfying $\omega(\widetilde{g}^{\beta}_{\lambda}) = \widetilde{g}^{\beta}_{\lambda'}$. The map $\bar\phi : \Lambda \to \Lambda$ which sends $\widetilde{g}^{\beta}_{\lambda}$ to $g^{\beta}_{\lambda}$, is a ring automorphism. If we define the linear map $\bar\tau : \Lambda \to \Lambda$ given by $\bar\tau({g}^{\beta}_{\lambda}) = {g}^{\beta}_{\lambda'}$, then by the symmetry $d^{\nu}_{\lambda \mu} = d^{\nu'}_{\lambda' \mu'}$, this map is a ring automorphism. Then we have $\omega = \bar\phi \bar\tau \bar \phi^{-1}$ and $\Lambda \cong \mathbb{Z}[\beta][\widetilde{g}^{\beta}_{(1)}, \widetilde{g}^{\beta}_{(2)}, \ldots]$ as a polynomial ring. Once we specify these generators $\widetilde{g}^{\beta}_{(k)}$ we can say that the automorphism $\bar\phi$ (and hence $\phi$ for $\widetilde{G}^{\beta}_{\lambda}$) is unique. Note also that another automorphism which sends $g^{(\alpha, \beta)}_{\lambda}$ to $g^{(\alpha, \beta)}_{\lambda'}$ coincides with the canonical involution $\omega$ for $\alpha = \beta$. 

\section{Combinatorial formulas for $g^{(\alpha, \beta)}_{\lambda}$ and rim border tableaux}\label{rrt}

In this section we give combinatorial formulas for the dual canonical Grothendieck polynomials $g^{(\alpha, \beta)}_{\lambda}$. 

A {\it reverse plane partition} (RPP) is a filling of a Young diagram so that each box contains a single positive integer and numbers weakly increase in rows (from left to right) and columns (from top to bottom). A {\it rim hook} (or {\it ribbon}) is a connected skew shape which contains no $2 \times 2$ square. 

\begin{definition}[Rim border tableaux (RBT)] \label{rrpp}
Let $T$ be an RPP. For each integer $i$ written in $T$ let $T_i$ be the (skew shape) part of $T$ containing all elements $i$. Define the {\it border} $R_i$ ($R_i \subset T_i$) consisting of all boxes $b$ of $T_i$ for which no box on the same diagonal above and to the left of $b$ is in $T_i$. For example, in Figure \ref{on} (a), the borders $R_1, R_2, R_3$ are shadowed. Let $I_i = T_i \setminus R_i$ be the {\it inner part} of $T_i$. 
\begin{figure}[h]
{\scriptsize
\begin{tikzpicture}[scale = 0.4]
	      \fill[fill=black!15] (0.5,1.5) rectangle (12.5,0.5);
	      \fill[fill=black!15] (0.5,1.5) rectangle (1.5,-4.5);
	      \fill[fill=black!15] (4.5,0.5) rectangle (7.5,-0.5);
	      \fill[fill=black!15] (8.5,0.5) rectangle (9.5,-0.5);
	      \fill[fill=black!15] (3.5,-0.5) rectangle (5.5,-1.5);
	      \fill[fill=black!15] (6.5,-0.5) rectangle (9.5,-1.5);
	      \fill[fill=black!15] (2.5,-1.5) rectangle (4.5,-2.5);
	      \fill[fill=black!15] (6.5,-1.5) rectangle (7.5,-2.5);
	      \fill[fill=black!15] (0.5,-2.5) rectangle (5.5,-3.5);
	      \fill[fill=black!15] (0.5,-3.5) rectangle (3.5,-4.5);

	      \node (1) at (   1,  1) {1};
	      \node (2) at (   2,  1) {1};
	      \node (3) at (   3,  1) {1};
	      \node (4) at (   4,  1) {1};
	      \node (5) at (   5,  1) {1};
	      \node (6) at (   6,  1) {1};
	      \node (7) at (   7,  1) {2};
	      \node (8) at (   8,  1) {2};
	      \node (9) at (   9,  1) {2};
	      \node (10) at (   10,  1) {2};
	      \node (11) at (   11,  1) {3};
	      \node (12) at (   12,  1) {3};
	      
	      \node (a) at (-0.5, 1) {\normalsize(a)};

	      \node (21) at (   1,  0) {1};
	      \node (22) at (   2,  0) {1};
	      \node (23) at (   3,  0) {1};
	      \node (24) at (   4,  0) {1};
	      \node (25) at (   5,  0) {2};
	      \node (26) at (   6,  0) {2};
	      \node (27) at (   7,  0) {2};
	      \node (28) at (   8,  0) {2};
	      \node (29) at (   9,  0) {3};

	      \node (31) at (   1,  -1) {1};
	      \node (32) at (   2,  -1) {1};
	      \node (33) at (   3,  -1) {1};
	      \node (34) at (   4,  -1) {2};
	      \node (35) at (   5,  -1) {2};
	      \node (36) at (   6,  -1) {2};
	      \node (37) at (   7,  -1) {3};
	      \node (38) at (   8,  -1) {3};
	      \node (39) at (   9,  -1) {3};

	      \node (41) at (   1,  -2) {1};
	      \node (42) at (   2,  -2) {1};
	      \node (43) at (   3,  -2) {2};
	      \node (44) at (   4,  -2) {2};
	      \node (45) at (   5,  -2) {2};
	      \node (46) at (   6,  -2) {2};
	      \node (47) at (   7,  -2) {3};
	      \node (48) at (   8,  -2) {3};

	      \node (51) at (   1,  -3) {2};
	      \node (52) at (   2,  -3) {2};
	      \node (53) at (   3,  -3) {3};
	      \node (54) at (   4,  -3) {3};
	      \node (55) at (   5,  -3) {3};

	      \node (61) at (   1,  -4) {3};
	      \node (62) at (   2,  -4) {3};
	      \node (63) at (   3,  -4) {3};

	      \draw[very thick] (0.5,1.5) to (6.5,1.5);
	      \draw[very thick] (6.5,1.5) to (6.5,0.5);
	      \draw[very thick] (6.5,0.5) to (1.5,0.5);
	      \draw[very thick] (1.5,0.5) to (1.5,-2.5);
	      \draw[very thick] (1.5,-2.5) to (0.5,-2.5);
	      \draw[very thick] (0.5,-2.5) to (0.5,1.5);

	      \draw[very thick] (6.5,1.5) to (10.5,1.5);
	      \draw[very thick] (10.5,1.5) to (10.5,0.5);
	      \draw[very thick] (10.5,0.5) to (7.5,0.5);
	      \draw[very thick] (7.5,0.5) to (7.5,-0.5);
	      \draw[very thick] (7.5,-0.5) to (5.5,-0.5);
	      \draw[very thick] (5.5,-0.5) to (5.5,-1.5);
	      \draw[very thick] (5.5,-1.5) to (5.5-1,-1.5);
	      \draw[very thick] (5.5-1,-1.5) to (5.5-1,-1.5-1);
	      \draw[very thick] (5.5-1,-1.5-1) to (5.5-1-2,-1.5-1);
	      \draw[very thick] (5.5-1-2,-1.5-1) to (5.5-1-2,-1.5-1+1);
	      \draw[very thick] (5.5-1-2,-1.5-1+1) to (5.5-1-2+1,-1.5-1+1);
	      \draw[very thick] (5.5-1-2+1,-1.5-1+1) to (5.5-1-2+1,-1.5-1+1+1) to (5.5-1-2+1+1,-1.5-1+1+1) to (5.5-1-2+1+1,-1.5-1+1+1+1);

	      \draw[very thick] (0.5,-2.5) to (0.5,-2.5-1) to (0.5+2,-2.5-1) to (0.5+2,-2.5-1+1) to (0.5+2-1,-2.5-1+1);
		\draw[very thick] (0.5, -3.5) to (0.5, -4.5) to (0.5+3, -4.5) to (0.5+3, -4.5+1) to (0.5+3+2, -4.5+1) to (0.5+3+2, -4.5+1+1) to (0.5+3+2-1, -4.5+1+1);		
		\draw[very thick] (6.5, -0.5) to (6.5, -2.5) to (6.5+1, -2.5) to (6.5+1, -2.5+1) to (6.5+1+1, -2.5+1) to (6.5+1+1+1, -2.5+1) to (6.5+1+1+1, -2.5+1+2) ;
		\draw[very thick] (6.5+1+1+1-1, -2.5+1+2) to (6.5+1+1+1-1, -2.5+1+2-1) to (6.5+1+1+1-1-1, -2.5+1+2-1);
		\draw[] (6.5, -2.5) to (6.5-1, -2.5);
		\draw[] (8.5, -2.5+1) to (8.5, -2.5) to (8.5-1, -2.5);
		\draw[very thick] (10.5, 0.5) to (10.5+2, 0.5) to (10.5+2, 0.5+1) to (10.5+2-2, 0.5+1);
		
\end{tikzpicture}
\quad
\begin{tikzpicture}[scale = 0.4]
	      \node (7) at (   7,  1) {2};
	      \node (8) at (   8,  1) {2};
	      \node (9) at (   9,  1) {2};
	      \node (10) at (   10,  1) {2};

	      \node (25) at (   5,  0) {2};
	      \node (26) at (   6,  0) {2};
	      \node (27) at (   7,  0) {2};
	      \node (28) at (   8,  0) {\color{gray}{\bf 2}};

	      \node (34) at (   4,  -1) {2};
	      \node (35) at (   5,  -1) {2};
	      \node (36) at (   6,  -1) {\color{gray}{\bf 2}};

	      \node (43) at (   3,  -2) {2};
	      \node (44) at (   4,  -2) {2};
	      \node (45) at (   5,  -2) {\color{gray}{\bf 2}};
	      \node (46) at (   6,  -2) {\color{gray}{\bf 2}};

	      \node (51) at (   1,  -3) {2};
	      \node (52) at (   2,  -3) {2};
	      
	       \node (b) at (-0.5, 1) {\normalsize(b)};
	       \node (bb) at (0.5, -4.33) {};

	      \draw[very thick] (6.5,1.5) to (6.5,0.5);
	      \draw[very thick] (6.5,0.5) to (4.5,0.5);
	      \draw[very thick] (1.5,-2.5) to (0.5,-2.5);

	      \draw[very thick] (6.5,1.5) to (10.5,1.5);
	      \draw[very thick] (10.5,1.5) to (10.5,0.5);
	      \draw[very thick] (10.5,0.5) to (7.5,0.5);
	      \draw[very thick] (7.5,0.5) to (7.5,-0.5);
	      \draw[very thick] (7.5,-0.5) to (5.5,-0.5);
	      \draw[very thick] (5.5,-0.5) to (5.5,-1.5);
	      \draw[very thick] (5.5,-1.5) to (5.5-1,-1.5);
	      \draw[very thick] (5.5-1,-1.5) to (5.5-1,-1.5-1);
	      \draw[very thick] (5.5-1,-1.5-1) to (5.5-1-2,-1.5-1);
	      \draw[very thick] (5.5-1-2,-1.5-1) to (5.5-1-2,-1.5-1+1);
	      \draw[very thick] (5.5-1-2,-1.5-1+1) to (5.5-1-2+1,-1.5-1+1);
	      \draw[very thick] (5.5-1-2+1,-1.5-1+1) to (5.5-1-2+1,-1.5-1+1+1) to (5.5-1-2+1+1,-1.5-1+1+1) to (5.5-1-2+1+1,-1.5-1+1+1+1);

	      \draw[very thick] (0.5,-2.5) to (0.5,-2.5-1) to (0.5+2,-2.5-1) to (0.5+2,-2.5-1+1) to (0.5+2-1,-2.5-1+1);
	      
	      \draw[thick] (8.5,1.5) to (8.5,0.5);
	      \draw[thick] (5.5,0.5) to (5.5,-0.5);
\end{tikzpicture}

\begin{tikzpicture}[scale = 0.4]
	      \fill[fill=black!5] (0.5,1.5) rectangle (12.5,0.5);
	      \fill[fill=black!5] (0.5,1.5) rectangle (1.5,-4.5);
	      \fill[fill=black!5] (4.5,0.5) rectangle (7.5,-0.5);
	      \fill[fill=black!5] (8.5,0.5) rectangle (9.5,-0.5);
	      \fill[fill=black!5] (3.5,-0.5) rectangle (5.5,-1.5);
	      \fill[fill=black!5] (6.5,-0.5) rectangle (9.5,-1.5);
	      \fill[fill=black!5] (2.5,-1.5) rectangle (4.5,-2.5);
	      \fill[fill=black!5] (6.5,-1.5) rectangle (7.5,-2.5);
	      \fill[fill=black!5] (0.5,-2.5) rectangle (5.5,-3.5);
	      \fill[fill=black!5] (0.5,-3.5) rectangle (3.5,-4.5);

	      \node (1) at (   1,  1) {1};
	      \node (2) at (   2,  1) {1};
	      \node (3) at (   3,  1) {1};
	      \node (4) at (   4,  1) {1};
	      \node (5) at (   5,  1) {1};
	      \node (6) at (   6,  1) {1};
	      \node (7) at (   7,  1) {2};
	      \node (8) at (   8,  1) {2};
	      \node (9) at (   9,  1) {2};
	      \node (10) at (   10,  1) {2};
	      \node (11) at (   11,  1) {3};
	      \node (12) at (   12,  1) {3};

	      \node (c) at (-0.5, 1) {\normalsize(c)};

	      \node (21) at (   1,  0) {1};
	      \node (22) at (   2,  0) {1};
	      \node (23) at (   3,  0) {1};
	      \node (24) at (   4,  0) {1};
	      \node (25) at (   5,  0) {2};
	      \node (26) at (   6,  0) {2};
	      \node (27) at (   7,  0) {2};
	      \node (28) at (   8,  0) {2};
	      \node (29) at (   9,  0) {3};

	      \node (31) at (   1,  -1) {1};
	      \node (32) at (   2,  -1) {1};
	      \node (33) at (   3,  -1) {1};
	      \node (34) at (   4,  -1) {2};
	      \node (35) at (   5,  -1) {2};
	      \node (36) at (   6,  -1) {2};
	      \node (37) at (   7,  -1) {3};
	      \node (38) at (   8,  -1) {3};
	      \node (39) at (   9,  -1) {3};

	      \node (41) at (   1,  -2) {1};
	      \node (42) at (   2,  -2) {1};
	      \node (43) at (   3,  -2) {2};
	      \node (44) at (   4,  -2) {2};
	      \node (45) at (   5,  -2) {2};
	      \node (46) at (   6,  -2) {2};
	      \node (47) at (   7,  -2) {3};
	      \node (48) at (   8,  -2) {3};

	      \node (51) at (   1,  -3) {2};
	      \node (52) at (   2,  -3) {2};
	      \node (53) at (   3,  -3) {3};
	      \node (54) at (   4,  -3) {3};
	      \node (55) at (   5,  -3) {3};

	      \node (61) at (   1,  -4) {3};
	      \node (62) at (   2,  -4) {3};
	      \node (63) at (   3,  -4) {3};

	      \draw[very thick] (0.5,1.5) to (6.5,1.5);
	      \draw[very thick] (6.5,1.5) to (6.5,0.5);
	      \draw[very thick] (6.5,0.5) to (1.5,0.5);
	      \draw[very thick] (1.5,0.5) to (1.5,-2.5);
	      \draw[very thick] (1.5,-2.5) to (0.5,-2.5);
	      \draw[very thick] (0.5,-2.5) to (0.5,1.5);

	      \draw[thick] (3.5,1.5) to (3.5,0.5);
	      \draw[thick] (8.5,1.5) to (8.5,0.5);
	      \draw[thick] (11.5,1.5) to (11.5,0.5);
	      \draw[thick] (5.5,0.5) to (5.5,-0.5);
	      \draw[thick] (1.5,-3.5) to (1.5,-4.5);
	      \draw[thick] (7.5,-0.5) to (7.5,-1.5);

	      \draw[very thick] (6.5,1.5) to (10.5,1.5);
	      \draw[very thick] (10.5,1.5) to (10.5,0.5);
	      \draw[very thick] (10.5,0.5) to (7.5,0.5);
	      \draw[very thick] (7.5,0.5) to (7.5,-0.5);
	      \draw[very thick] (7.5,-0.5) to (5.5,-0.5);
	      \draw[very thick] (5.5,-0.5) to (5.5,-1.5);
	      \draw[very thick] (5.5,-1.5) to (5.5-1,-1.5);
	      \draw[very thick] (5.5-1,-1.5) to (5.5-1,-1.5-1);
	      \draw[very thick] (5.5-1,-1.5-1) to (5.5-1-2,-1.5-1);
	      \draw[very thick] (5.5-1-2,-1.5-1) to (5.5-1-2,-1.5-1+1);
	      \draw[very thick] (5.5-1-2,-1.5-1+1) to (5.5-1-2+1,-1.5-1+1);
	      \draw[very thick] (5.5-1-2+1,-1.5-1+1) to (5.5-1-2+1,-1.5-1+1+1) to (5.5-1-2+1+1,-1.5-1+1+1) to (5.5-1-2+1+1,-1.5-1+1+1+1);

	      \draw[very thick] (0.5,-2.5) to (0.5,-2.5-1) to (0.5+2,-2.5-1) to (0.5+2,-2.5-1+1) to (0.5+2-1,-2.5-1+1);
		\draw[very thick] (0.5, -3.5) to (0.5, -4.5) to (0.5+3, -4.5) to (0.5+3, -4.5+1) to (0.5+3+2, -4.5+1) to (0.5+3+2, -4.5+1+1) to (0.5+3+2-1, -4.5+1+1);		
		\draw[very thick] (6.5, -0.5) to (6.5, -2.5) to (6.5+1, -2.5) to (6.5+1, -2.5+1) to (6.5+1+1, -2.5+1) to (6.5+1+1+1, -2.5+1) to (6.5+1+1+1, -2.5+1+2) ;
		\draw[very thick] (6.5+1+1+1-1, -2.5+1+2) to (6.5+1+1+1-1, -2.5+1+2-1) to (6.5+1+1+1-1-1, -2.5+1+2-1);
		\draw[] (6.5, -2.5) to (6.5-1, -2.5);
		\draw[] (8.5, -2.5+1) to (8.5, -2.5) to (8.5-1, -2.5);
		\draw[very thick] (10.5, 0.5) to (10.5+2, 0.5) to (10.5+2, 0.5+1) to (10.5+2-2, 0.5+1);
\end{tikzpicture}
}

\caption{(a) An RPP $T$ with borders (gray) and inner parts (white); (b) the part $T_2$ with the border $R_2$ decomposed into four rim hooks; (c) A rim border tableaux (RBT) $T$ with $w_T(\alpha, \beta) = \alpha^{14} \beta^{9} (\alpha + \beta)^{11}$ and $x^T = x_1^2 x_2^4 x_3^6$.}\label{on}
\end{figure}
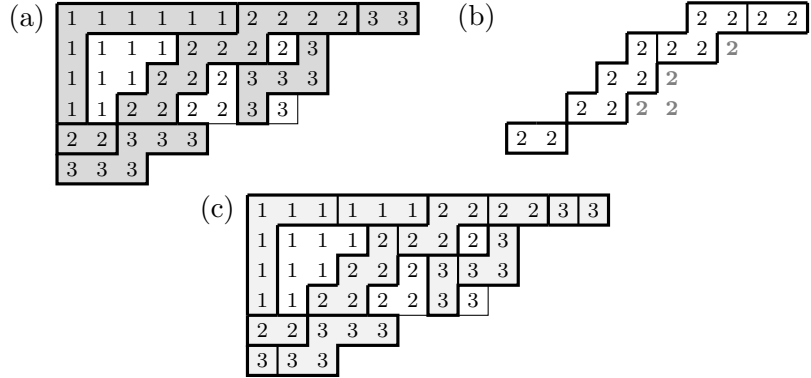
The inner parts $I_1, I_2, I_3$ correspond to white parts of $T$ in Figure \ref{on} (a). 
Let us then arbitrarily partition each border $R_i$ into rim hooks by some {\it vertical} cuts (see Figure \ref{on} (b)). 
We call this resulting tableau a {\it rim border tableau} (RBT). Example of the resulting RBT is given in Figure \ref{on} (c).  
\end{definition}

Let $RBT(\lambda)$ be the set of all RBT of shape $\lambda$. For each element $T \in RBT(\lambda)$ define the $(\alpha, \beta)$-weight $w_T(\alpha, \beta) = \alpha^{wt} \beta^{ht} (\alpha + \beta)^{in}$, where $wt$ is the sum of {\it width}$-1$ of all rim hooks in $T$, $ht$ is the sum of {\it height}$-1$ of all rim hooks in $T$, and $in$ is the total number of boxes in inner parts of $T$. The corresponding monomial $x^T = \prod_{i} x_i^{a_i}$ is defined so that $a_i$ is the number of rim hooks in $T$ containing $i$. See Figure \ref{on} (c) for an RBT $T$ with $w(T) = \alpha^{14} \beta^{9} (\alpha + \beta)^{11}$ and $x^T = x_1^2 x_2^4 x_3^6$. 

We now state that generating series for these rim border tableaux define combinatorial presentation of the dual polynomials $g^{(\alpha, \beta)}_{\lambda}$. 

\begin{theorem}\label{cg}
The dual polynomials   
$g^{(\alpha, \beta)}_{\lambda}$ satisfy the following formula
\begin{equation*}
g^{(\alpha, \beta)}_{\lambda}(x_1, x_2, \ldots) = \sum_{T \in RBT(\lambda)} w_T(\alpha, \beta) x^T.
\end{equation*}
\end{theorem}

We will prove this theorem in next section after giving Pieri and branching formulas for the functions $G^{(\alpha, \beta)}_{\lambda}, g^{(\alpha, \beta)}_{\lambda}$. We will also show that the polynomials $g^{(\alpha, \beta)}_{\lambda}$ are Schur-positive (for $\alpha, \beta > 0$, see Table 1 with some examples). 
Recall that by definition $g^{(\alpha, \beta)}_{\lambda}(x_1, x_2, \ldots)$ is a symmetric function satisfying the duality $\omega(g^{(\alpha, \beta)}_{\lambda}) = g^{(\beta, \alpha)}_{\lambda'}$.

\begin{figure}
{\scriptsize
\begin{tikzpicture}[scale = 0.4]
	      \node (1) at (   1,  1) {1};
	      \node (2) at (   2,  1) {1};
	      \node (3) at (   3,  1) {1};
	      \node (4) at (   4,  1) {1};
	      \node (5) at (   5,  1) {1};
	      \node (6) at (   6,  1) {1};
	      \node (7) at (   7,  1) {2};
	      \node (8) at (   8,  1) {2};
	      \node (9) at (   9,  1) {2};
	      \node (10) at (   10,  1) {2};
	      \node (11) at (   11,  1) {3};
	      \node (12) at (   12,  1) {3};

	      \node (c) at (-0.5, 1) {\normalsize(a)};

	      \node (21) at (   1,  0) {1};
	      \node (22) at (   2,  0) {1};
	      \node (23) at (   3,  0) {1};
	      \node (24) at (   4,  0) {1};
	      \node (25) at (   5,  0) {2};
	      \node (26) at (   6,  0) {2};
	      \node (27) at (   7,  0) {2};
	      \node (28) at (   8,  0) {2};
	      \node (29) at (   9,  0) {3};

	      \node (31) at (   1,  -1) {1};
	      \node (32) at (   2,  -1) {1};
	      \node (33) at (   3,  -1) {1};
	      \node (34) at (   4,  -1) {2};
	      \node (35) at (   5,  -1) {2};
	      \node (36) at (   6,  -1) {2};
	      \node (37) at (   7,  -1) {3};
	      \node (38) at (   8,  -1) {3};
	      \node (39) at (   9,  -1) {3};

	      \node (41) at (   1,  -2) {1};
	      \node (42) at (   2,  -2) {1};
	      \node (43) at (   3,  -2) {2};
	      \node (44) at (   4,  -2) {2};
	      \node (45) at (   5,  -2) {2};
	      \node (46) at (   6,  -2) {2};
	      \node (47) at (   7,  -2) {3};
	      \node (48) at (   8,  -2) {3};

	      \node (51) at (   1,  -3) {2};
	      \node (52) at (   2,  -3) {2};
	      \node (53) at (   3,  -3) {3};
	      \node (54) at (   4,  -3) {3};
	      \node (55) at (   5,  -3) {3};

	      \node (61) at (   1,  -4) {3};
	      \node (62) at (   2,  -4) {3};
	      \node (63) at (   3,  -4) {3};
	      
	      \draw[very thick] (0.5, 1.5) to (0.5, -4.5) to (3.5, -4.5) to (3.5, -3.5) to (5.5, -3.5) to (5.5, -2.5) to (8.5, -2.5) to (8.5, -1.5) to (9.5, -1.5) to (9.5, 0.5) to (12.5, 0.5) to (12.5, 1.5) to (0.5, 1.5);
	      
	      \draw[very thick] (1.5, 1.5) to (1.5, -4.5);
	      \draw[very thick] (2.5, 1.5) to (2.5, -4.5);
	      \draw[very thick] (3.5, 1.5) to (3.5, -3.5);
	      \draw[very thick] (4.5, 1.5) to (4.5, -3.5);
	      \draw[very thick] (5.5, 1.5) to (5.5, -2.5);
	      \draw[very thick] (6.5, 1.5) to (6.5, -2.5);
	      \draw[very thick] (7.5, 1.5) to (7.5, -2.5);
	      \draw[very thick] (8.5, 1.5) to (8.5, -1.5);
	      \draw[very thick] (9.5, 1.5) to (9.5, 0.5);
	      \draw[very thick] (10.5, 1.5) to (10.5, 0.5);
	      \draw[very thick] (11.5, 1.5) to (11.5, 0.5);
	      
	      \draw[very thick] (0.5, -2.5) to (5.5, -2.5);
	      \draw[very thick] (0.5, -3.5) to (2.5, -3.5);
	      \draw[very thick] (2.5, -1.5) to (3.5, -1.5);
	      \draw[very thick] (3.5, -0.5) to (4.5, -0.5);
	      \draw[very thick] (4.5, 0.5) to (6.5, 0.5);
	      \draw[very thick] (6.5, -0.5) to (8.5, -0.5);
	      \draw[very thick] (8.5, 0.5) to (9.5, 0.5);	      
\end{tikzpicture}
\qquad 
\begin{tikzpicture}[scale = 0.4]
	      \node (1) at (   1,  1) {1};
	      \node (2) at (   2,  1) {1};
	      \node (3) at (   3,  1) {1};
	      \node (4) at (   4,  1) {1};
	      \node (5) at (   5,  1) {1};
	      \node (6) at (   6,  1) {1};
	      \node (7) at (   7,  1) {2};
	      \node (8) at (   8,  1) {2};
	      \node (9) at (   9,  1) {2};

	      \node (c) at (-0.5, 1) {\normalsize(b)};

	      \node (51) at (   1,  0) {2};
	      \node (52) at (   2,  0) {2};
	      \node (53) at (   3,  0) {2};
	      \node (54) at (   4,  0) {3};
	      \node (55) at (   5,  0) {3};
	      \node (55) at (   6,  0) {3};
	      \node (55) at (   7,  0) {3};

	      \node (61) at (   1,  -1) {3};
	      \node (62) at (   2,  -1) {3};
	      \node (63) at (   3,  -1) {3};

	      \draw[very thick] (0.5,1.5) to (6.5,1.5) to (6.5, 0.5) to (0.5, 0.5) to (0.5, 1.5);
	      \draw[very thick] (2.5, 0.5) to (2.5, 1.5);
	      
	      \draw[very thick] (6.5, 0.5) to (9.5, 0.5) to (9.5, 1.5) to (6.5, 1.5);
	      \draw[very thick] (8.5, 0.5) to (8.5, 1.5);
	      \draw[very thick] (0.5, 0.5) to (0.5, -0.5) to (3.5, -0.5) to (3.5, 0.5);
	      
	      \draw[very thick] (3.5, -0.5) to (7.5, -0.5) to (7.5, 0.5);
	      \draw[very thick] (5.5, -0.5) to (5.5, 0.5);
	      \draw[very thick] (0.5, -0.5) to (0.5, -1.5) to (3.5, -1.5) to (3.5, -0.5);
	      \draw[very thick] (1.5, -0.5) to (1.5, -1.5);
	      
	      	      \node (66) at (   3,  -4.25) { };

\end{tikzpicture}
}
\caption{(a) a reverse plane partition with the monomial weight $x_1^6 x_2^{10} x_3^{10}$; (b) a valued-set tableaux with the monomial weight $x_1^{2} x_2^{3} x_3^4$.}\label{ext}
\end{figure}
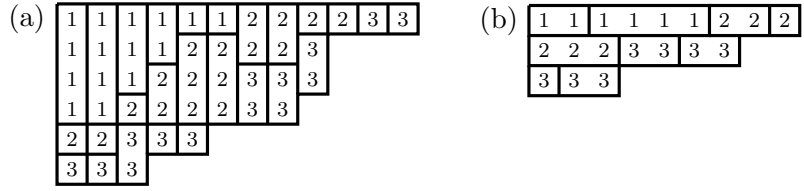

Let us now look at some special cases. 
\begin{itemize}
\item If $\alpha = 0, \beta = 1$, RBT's of nonzero weight correspond to RPP whose monomial weight is given by $x^T = \prod_i x_i^{a_i}$ where $a_i$ is the number of columns which contain $i$, see Figure \ref{ext}~(a). Therefore, $g^{(0,1)}_{\lambda} = g_{\lambda}$ recovers combinatorial presentation given in \cite{lp} for dual stable Grothendieck polynomials. 

\item For $\alpha = 1, \beta = 0$, nonzero weight RBT will be SSYT, their borders are just horizontal strips; each horizontal strip consisting of the same element split into several parts which account the monomial weight; this setting corresponds to the valued-set tableaux given in \cite{lp} for the polynomials $j_{\lambda} = \omega(g_{\lambda})$, see Figure \ref{ext}~(b). 

\item For $\alpha = \beta = 0$, RBT will count only SSYT with the usual monomial weight given as for Schur functions. 

\item One more interesting case arises when $\alpha + \beta = 0$, i.e. we sum over tableaux which have no inner parts. Here we consider a kind of {\it rim tableaux} with a special signed weight. More details on this case will be discussed in the final section.

\end{itemize}

The function $g^{(\alpha, \beta)}_{\lambda/\mu}$ extended for any skew shape $\lambda/\mu$ can be defined by this combinatorial formula (other ways are compatible with this definition). 

\subsection{Lattice forests on plane partitions}
We give one more equivalent combinatorial definition for RBT. From any $T \in RBT(\lambda)$ we construct the following {\it forest-type} structure. Let us put in the center of each box a vertex and then connect (in a chain manner) the vertices inside each rim hook in $T$. Every vertex in inner parts of $T$ has two options: to be connected by a vertical edge to the vertex in the upper box or by a horizontal edge to the vertex in the left box. This will produce a certain {\it lattice forest} on RPP where each tree has the same label.\footnote{These forests have special structure, so not all lattice forests will correspond to the objects that we define here.} The weight $\alpha^{a} \beta^b$ will correspond to the total number $a$ of horizontal edges and $b$ vertical edges. 
See Figure \ref{on2}.
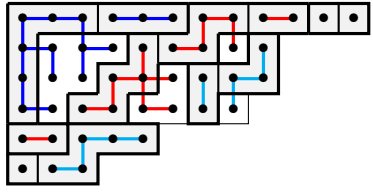
\begin{figure}
{\scriptsize
\begin{tikzpicture}[scale = 0.4]
	      \fill[fill=black!5] (0.5,1.5) rectangle (12.5,0.5);
	      \fill[fill=black!5] (0.5,1.5) rectangle (1.5,-4.5);
	      \fill[fill=black!5] (4.5,0.5) rectangle (7.5,-0.5);
	      \fill[fill=black!5] (8.5,0.5) rectangle (9.5,-0.5);
	      \fill[fill=black!5] (3.5,-0.5) rectangle (5.5,-1.5);
	      \fill[fill=black!5] (6.5,-0.5) rectangle (9.5,-1.5);
	      \fill[fill=black!5] (2.5,-1.5) rectangle (4.5,-2.5);
	      \fill[fill=black!5] (6.5,-1.5) rectangle (7.5,-2.5);
	      \fill[fill=black!5] (0.5,-2.5) rectangle (5.5,-3.5);
	      \fill[fill=black!5] (0.5,-3.5) rectangle (3.5,-4.5);

	       \draw[very thick, blue] (1,1) to (2,1);
	      \draw[very thick, blue] (2,1) to (3,1);
	      \draw[very thick, blue] (4,1) to (5,1) to (6,1);
	      \draw[very thick, red] (6,0) to (7,0) to (7,1) to (8,1) to (8,0);
	      \draw[very thick, red] (9,1) to (10,1);
	      \draw[very thick, blue] (1,1) to (1,0) to (1,-1) to (1,-2);
	      \draw[very thick, blue] (2,0) to (1,0);
	      \draw[very thick, blue] (3,0) to (3,1);
	      \draw[very thick, blue] (4,0) to (3,0);
	      \draw[very thick, blue] (2,-1) to (2,0);
	      \draw[very thick, blue] (3,-1) to (3,0);
	      \draw[very thick, blue] (2,-2) to (1,-2);
	      \draw[very thick, red] (5,0) to (5,-1) to (4,-1) to (4,-2) to (3,-2);
	      \draw[very thick, cyan] (9,0) to (9,-1) to (8,-1) to (8,-2);
	      \draw[very thick, cyan] (7,-2) to (7,-1);
	      \draw[very thick, red] (6,-1) to (5,-1) to (5,-2) to (6,-2);
	      \draw[very thick, red] (1,-3) to (2,-3);
	      \draw[very thick, cyan] (2,-4) to (3,-4) to (3,-3) to (4,-3) to (5,-3);
	      
	      \draw[very thick] (0.5,1.5) to (6.5,1.5);
	      \draw[very thick] (6.5,1.5) to (6.5,0.5);
	      \draw[very thick] (6.5,0.5) to (1.5,0.5);
	      \draw[very thick] (1.5,0.5) to (1.5,-2.5);
	      \draw[very thick] (1.5,-2.5) to (0.5,-2.5);
	      \draw[very thick] (0.5,-2.5) to (0.5,1.5);

	      \draw[thick] (3.5,1.5) to (3.5,0.5);
	      \draw[thick] (8.5,1.5) to (8.5,0.5);
	      \draw[thick] (11.5,1.5) to (11.5,0.5);
	      \draw[thick] (5.5,0.5) to (5.5,-0.5);
	      \draw[thick] (1.5,-3.5) to (1.5,-4.5);
	      \draw[thick] (7.5,-0.5) to (7.5,-1.5);

	      \draw[very thick] (6.5,1.5) to (10.5,1.5);
	      \draw[very thick] (10.5,1.5) to (10.5,0.5);
	      \draw[very thick] (10.5,0.5) to (7.5,0.5);
	      \draw[very thick] (7.5,0.5) to (7.5,-0.5);
	      \draw[very thick] (7.5,-0.5) to (5.5,-0.5);
	      \draw[very thick] (5.5,-0.5) to (5.5,-1.5);
	      \draw[very thick] (5.5,-1.5) to (5.5-1,-1.5);
	      \draw[very thick] (5.5-1,-1.5) to (5.5-1,-1.5-1);
	      \draw[very thick] (5.5-1,-1.5-1) to (5.5-1-2,-1.5-1);
	      \draw[very thick] (5.5-1-2,-1.5-1) to (5.5-1-2,-1.5-1+1);
	      \draw[very thick] (5.5-1-2,-1.5-1+1) to (5.5-1-2+1,-1.5-1+1);
	      \draw[very thick] (5.5-1-2+1,-1.5-1+1) to (5.5-1-2+1,-1.5-1+1+1) to (5.5-1-2+1+1,-1.5-1+1+1) to (5.5-1-2+1+1,-1.5-1+1+1+1);

	      \draw[very thick] (0.5,-2.5) to (0.5,-2.5-1) to (0.5+2,-2.5-1) to (0.5+2,-2.5-1+1) to (0.5+2-1,-2.5-1+1);
		\draw[very thick] (0.5, -3.5) to (0.5, -4.5) to (0.5+3, -4.5) to (0.5+3, -4.5+1) to (0.5+3+2, -4.5+1) to (0.5+3+2, -4.5+1+1) to (0.5+3+2-1, -4.5+1+1);		
		\draw[very thick] (6.5, -0.5) to (6.5, -2.5) to (6.5+1, -2.5) to (6.5+1, -2.5+1) to (6.5+1+1, -2.5+1) to (6.5+1+1+1, -2.5+1) to (6.5+1+1+1, -2.5+1+2) ;
		\draw[very thick] (6.5+1+1+1-1, -2.5+1+2) to (6.5+1+1+1-1, -2.5+1+2-1) to (6.5+1+1+1-1-1, -2.5+1+2-1);
		\draw[] (6.5, -2.5) to (6.5-1, -2.5);
		\draw[] (8.5, -2.5+1) to (8.5, -2.5) to (8.5-1, -2.5);
		\draw[very thick] (10.5, 0.5) to (10.5+2, 0.5) to (10.5+2, 0.5+1) to (10.5+2-2, 0.5+1);	

	      \node (1) at (   1,  1) {$\bullet$};
	      \node (2) at (   2,  1) {$\bullet$};
	      \node (3) at (   3,  1) {$\bullet$};
	      \node (4) at (   4,  1) {$\bullet$};
	      \node (5) at (   5,  1) {$\bullet$};
	      \node (6) at (   6,  1) {$\bullet$};
	      \node (7) at (   7,  1) {$\bullet$};
	      \node (8) at (   8,  1) {$\bullet$};
	      \node (9) at (   9,  1) {$\bullet$};
	      \node (10) at (   10,  1) {$\bullet$};
	      \node (11) at (   11,  1) {$\bullet$};
	      \node (12) at (   12,  1) {$\bullet$};

	      \node (21) at (   1,  0) {$\bullet$};
	      \node (22) at (   2,  0) {$\bullet$};
	      \node (23) at (   3,  0) {$\bullet$};
	      \node (24) at (   4,  0) {$\bullet$};
	      \node (25) at (   5,  0) {$\bullet$};
	      \node (26) at (   6,  0) {$\bullet$};
	      \node (27) at (   7,  0) {$\bullet$};
	      \node (28) at (   8,  0) {$\bullet$};
	      \node (29) at (   9,  0) {$\bullet$};

	      \node (31) at (   1,  -1) {$\bullet$};
	      \node (32) at (   2,  -1) {$\bullet$};
	      \node (33) at (   3,  -1) {$\bullet$};
	      \node (34) at (   4,  -1) {$\bullet$};
	      \node (35) at (   5,  -1) {$\bullet$};
	      \node (36) at (   6,  -1) {$\bullet$};
	      \node (37) at (   7,  -1) {$\bullet$};
	      \node (38) at (   8,  -1) {$\bullet$};
	      \node (39) at (   9,  -1) {$\bullet$};

	      \node (41) at (   1,  -2) {$\bullet$};
	      \node (42) at (   2,  -2) {$\bullet$};
	      \node (43) at (   3,  -2) {$\bullet$};
	      \node (44) at (   4,  -2) {$\bullet$};
	      \node (45) at (   5,  -2) {$\bullet$};
	      \node (46) at (   6,  -2) {$\bullet$};
	      \node (47) at (   7,  -2) {$\bullet$};
	      \node (48) at (   8,  -2) {$\bullet$};

	      \node (51) at (   1,  -3) {$\bullet$};
	      \node (52) at (   2,  -3) {$\bullet$};
	      \node (53) at (   3,  -3) {$\bullet$};
	      \node (54) at (   4,  -3) {$\bullet$};
	      \node (55) at (   5,  -3) {$\bullet$};

	      \node (61) at (   1,  -4) {$\bullet$};
	      \node (62) at (   2,  -4) {$\bullet$};
	      \node (63) at (   3,  -4) {$\bullet$};
\end{tikzpicture}
}

\caption{A lattice forest on RPP.}  \label{on2}
\end{figure}

\section{Pieri and branching formulas}\label{spieri}

Let us fix some notation which we will repeatedly use in this section. For any skew shape $\mu/\lambda$ define:
\begin{itemize}
\item[] $r(\mu/\lambda)$ as the number of rows; 
\item[] $c(\mu/\lambda)$ as the number of columns;
\item[] $b(\mu/\lambda)$ as the number of connected components;
\item[] $i(\mu/\lambda) = |\mu/\lambda| - c(\mu/\lambda) - r(\mu/\lambda) + b(\mu/\lambda)$ as the number of boxes in inner part. 
\end{itemize}
For example, for a skew shape $\mu/\lambda = 665222/2222$ we have $r(\mu/\lambda) = 5,$ $c(\mu/\lambda) = 6,$ $b(\mu/\lambda) = 2,$ $i(\mu/\lambda) = 6$. See Figure \ref{fig11}.

\begin{figure}[h]
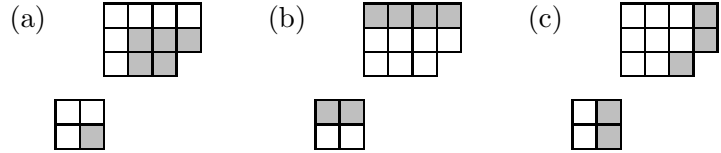

\ytableausetup{boxsize=0.8em} 
{
(a)
\begin{ytableau}
\none & \none &  ~ & ~ & ~ & ~ \\
\none & \none & ~ &  *(lightgray)  &  *(lightgray) & *(lightgray) \\
\none & \none & ~ & *(lightgray)  &  *(lightgray)\\
\none & \none \\
~ & ~\\
~ & *(lightgray)
\end{ytableau}
\qquad
(b)
\begin{ytableau}
\none & \none &  *(lightgray) & *(lightgray) & *(lightgray) & *(lightgray) \\
\none & \none & ~ &  ~ &  ~ &  ~ \\
\none & \none & ~ &   &  \\
\none & \none \\
*(lightgray) & *(lightgray) \\
~ & ~
\end{ytableau}
\qquad
(c)
\begin{ytableau}
\none & \none &  ~ & ~ & ~ & *(lightgray) \\
\none & \none & ~ &  ~ & ~ & *(lightgray) \\
\none & \none & ~ &   &  *(lightgray)\\
\none & \none \\
~ & *(lightgray) \\
~ & *(lightgray)
\end{ytableau}
}
\caption{A skew shape $\mu/\lambda = 665222/2222$ has 2 connected components, 6 columns, 5 rows and shadowed: (a) 6 boxes in inner part; (b) the upper boundary; (c) the right boundary.} \label{fig11}
\end{figure}

\subsection{Pieri type formulas for $G^{(\alpha, \beta)}_{\lambda}$}
\begin{theorem} \label{pieri}The following formulas hold.

Type 1:
\begin{align}
G^{(\alpha, \beta)}_{(k)} G^{(\alpha, \beta)}_{\lambda} = \sum_{\mu/\lambda \text{ hor. strip}} (\alpha + \beta)^{|\mu/\lambda| - k} \binom{r(\mu/\lambda) - 1}{|\mu/\lambda| - k} G^{(\alpha, \beta)}_{\mu},
\end{align}
\begin{align}
G^{(\alpha, \beta)}_{(1^k)} G^{(\alpha, \beta)}_{\lambda} = \sum_{\mu/\lambda \text{ vert. strip}} (\alpha + \beta)^{|\mu/\lambda| - k} \binom{c(\mu/\lambda) - 1}{|\mu/\lambda| - k} G^{(\alpha, \beta)}_{\mu}.
\end{align}

Type 2:
\begin{align}
h_k\left(\frac{x}{1 + \alpha x} \right) G^{(-\alpha, -\beta)}_{\lambda} &= \sum_{\nu:\ c(\nu/\lambda) = k} (\alpha + \beta)^{|\nu/\lambda| - k} G^{(-\alpha, -\beta)}_{\nu}\\
e_k\left(\frac{x}{1 + \alpha x} \right) G^{(-\alpha, -\beta)}_{\lambda} &= \sum_{\nu/\lambda \text{ vert. strip}} (\alpha + \beta)^{|\nu/\lambda| - k} \binom{|\nu/\lambda| - c(\nu/\lambda)}{k - c(\nu/\lambda)} G^{(-\alpha, -\beta)}_{\nu}
\end{align}

Type 3:
\begin{align}
h_{k} G^{(-\alpha, -\beta)}_{\lambda} = \sum_{\mu} v^{k}_{\mu/\lambda}(\alpha, \beta) G^{(-\alpha, -\beta)}_{\mu},
\end{align}
\begin{align}\label{ge}
e_{k} G^{(-\alpha, -\beta)}_{\lambda} = \sum_{\mu} \bar v^{k}_{\mu/\lambda}(\alpha, \beta) G^{(-\alpha, -\beta)}_{\mu}, 
\end{align}
where 
\begin{equation}\label{vi}
v^{k}_{\mu/\lambda}(\alpha, \beta) = \begin{cases}
    \beta^{r(\mu/\lambda) - b(\mu/\lambda)}(\alpha + \beta)^{i(\mu/\lambda)} \alpha^{c(\mu/\lambda) - k} \binom{c(\mu/\lambda) - b(\mu/\lambda)}{c(\mu/\lambda) - k}, & \text{ if } \lambda \subseteq \mu,\\
    0, & \text{ otherwise}.
\end{cases}
\end{equation}
and $\bar v^{k}_{\mu/\lambda}(\alpha, \beta)$ defined similarly as $ v^{k}_{\mu/\lambda}(\alpha, \beta)$ but with $\alpha, \beta$ and $r, c$ being simultaneously switched. 
\end{theorem}

\begin{remark}
Type 1 formulas are finite sums and other types are infinite in general.
\end{remark}

Before proving Theorem \ref{pieri} we prepare some lemmas.

\begin{lemma}\label{vc} The coefficients of Type 3 formulas satisfy the identity
{
\begin{align}\label{vk}
&v^{k}_{\mu/\lambda}(\alpha, \beta) = \sum_{\mu/\nu \text{ vert. strip}} \alpha^{c(\nu/\lambda) - k} \binom{c(\nu/\lambda)}{k} (\alpha + \beta)^{|\nu/\lambda| - c(\nu/\lambda)} (-\alpha)^{c(\mu/\nu)} \beta^{|\mu/\nu| - c(\mu/\nu)}.
\end{align}
}
\end{lemma}

\begin{proof}
For a skew shape $\mu/\lambda$ define its {\it upper boundary} as the horizontal strip containing all boxes for which no box lies strictly above each of them. Similarly, define its {\it right boundary} as the vertical strip of all boxes for which no box lies strictly to the right. See Figure \ref{fig11}.

Let $\nu$ be a partition so that $\mu/\nu$ is a vertical strip. We interpret the r.h.s. of \eqref{vk} with weighted fillings of the shape $\mu/\lambda$ having the following properties:
\begin{itemize}
\item[(a)] each box has one of five weights $\{\alpha,-\alpha, \beta,  \alpha + \beta, 1\}$;
\item[(b)] the upper boundary of $\nu/\lambda$ (which has $c(\nu/\lambda)$ elements) has $k$ elements $1$ and $c(\nu/\lambda) - k$ elements $\alpha$;
\item[(c)] other boxes in $\nu/\lambda$ have weight $\alpha + \beta$;
\item[(d)] in the remaining part $\mu/\nu$, the bottom elements of each column have weight $-\alpha$ and each of the remaining boxes has weight $\beta$.
\end{itemize}

The weight of any such tableau is the product of weights of its entries. It is easy to see that for different $\nu$ we obtain different tableau and the total sum of weights of all such tableau gives the r.h.s. of \eqref{vk}. 
We now explain how to cancel most of the elements in these tableaux so that the formula will match \eqref{vi}. 

Consider locally how varies (when $\nu$ runs) the total weight of a single column of the right boundary of $\mu/\lambda$. 
If no box in this column belongs to $\nu$ then the weight of the column is $-\alpha \beta^{h-1}$ where $h$ is the height of the column. When $\nu$ occupies $i$ ($1 \le i \le h - 1$) topmost boxes of the column, the weight is given by $-\alpha \beta^{h - 1 - i}(\alpha + \beta)^{i-1} - \alpha^2 \beta^{h - 1 - i}(\alpha + \beta)^{i-1}$ depending on whether the topmost box has weight $1$ or $\alpha$. If $\nu$ occupies the entire column, the weight is given by $(\alpha + \beta)^{h-1} + \alpha(\alpha + \beta)^{h-1}$ (again we sum depending on whether the topmost box has weight $1$ or $\alpha$). Therefore, the total weight (when $\nu$ varies on this column and is fixed everywhere else) is given by 
\begin{align*}
&-\alpha \beta^{h-1} - \sum_{i = 1}^{h-1}(\alpha \beta^{h - 1 - i}(\alpha + \beta)^{i-1} + \alpha^2 \beta^{h - 1 - i}(\alpha + \beta)^{i-1}) + (\alpha + \beta)^{h-1} + \alpha(\alpha + \beta)^{h-1}\\
&= -\alpha \beta^{h-1} + (\beta^{h-1} - (\alpha + \beta)^{h-1}) + \alpha (\beta^{h-1} - (\alpha + \beta)^{h-1}) + (\alpha + \beta)^{h-1} + \alpha(\alpha + \beta)^{h-1}\\
&= \beta^{h-1}.
\end{align*}
If the topmost box of this column does not belong to the upper boundary of $\mu/\lambda,$ then we have the sum with the same result
\begin{align*}
&- \sum_{i = 1}^{h-1}\alpha \beta^{h - 1 - i}(\alpha + \beta)^{i-1} + (\alpha + \beta)^{h-1} = \beta^{h-1}.
\end{align*}

This means that we can transform this column into the column whose topmost box has weight $1$ and every other box has weight $\beta$, so that the total weight will be conserved. We can do this procedure on every column of the right boundary of $\mu/\lambda$ and therefore every column will have this property. Finally, note that these resulted tableaux correspond to the defining formula \eqref{vi} of $v^{k}_{\mu/\lambda}(\alpha, \beta)$.
\end{proof}

\begin{lemma}\label{hh}
\begin{align}
h_k\left(x \right) &=\sum_i (-\alpha)^i e_i\left(\frac{x}{1 + \alpha x} \right) \sum_j h_{j}\left(\frac{x}{1 + \alpha x} \right) \alpha^{j - k} \binom{j}{k}
\end{align}
\end{lemma}

\begin{proof}
The proof is a standard manipulation with the generating series
$\sum_{k \ge 0} h_k(x) t^k = \prod_{i} \frac{1}{1 - t x_i}$
applying the substitutions $x \to \frac{x}{1 - \alpha x}$ and then $x \to \frac{x}{1 + \alpha x}$. 
\end{proof}

\begin{proof}[Proof of Theorem \ref{pieri}]
For $\alpha = 0, \beta = \pm 1$, Lenart \cite{lenart} proved Pieri formulas, which can easily be restated with a $\beta$ parameter: 
{\small
\begin{align*}
G^{\beta}_{(k)} G^{\beta}_{\lambda} = \sum_{\mu/\lambda \text{ hor. strip}} \beta^{|\mu/\lambda| - k} \binom{r(\mu/\lambda) - 1}{|\mu/\lambda| - k} G^{\beta}_{\mu}, \quad
G^{\beta}_{(1^k)} G^{\beta}_{\lambda} = \sum_{\mu/\lambda \text{ vert. strip}} \beta^{|\mu/\lambda| - k} \binom{c(\mu/\lambda) - 1}{|\mu/\lambda| - k} G^{\beta}_{\mu},
\end{align*}
\begin{align*}
h_k G^{-\beta}_{\lambda} = \sum_{\nu:\ c(\nu/\lambda) = k} \beta^{|\nu/\lambda| - k} G^{-\beta}_{\nu}, \quad
e_k G^{-\beta}_{\lambda} = \sum_{\nu/\lambda \text{ vert. strip}} \beta^{|\nu/\lambda| - k} \binom{|\nu/\lambda| - c(\nu/\lambda)}{k - c(\nu/\lambda)} G^{-\beta}_{\nu}
\end{align*}
}
Type 1, 2 formulas for $G^{(\alpha, \beta)}_{\lambda}$ are followed then from the latter identities via substitutions $\beta \to \alpha + \beta$, $x \to \frac{x}{1 \pm \alpha x}$. 

We now prove Type 3 formulas. 
Applying Lemma \ref{hh} and then Type 2 formulas, we have 
\begin{align*}
&h_{k} G^{(-\alpha, -\beta)}_{\lambda} \\
&= \sum_{i} (-\alpha)^i e_i\left(\frac{x}{1 + \alpha x} \right) \sum_{j \ge k} \alpha^{j - k} \binom{j}{k} h_j\left(\frac{x}{1 + \alpha x} \right) G^{(-\alpha, -\beta)}_{\lambda}\\
&= \sum_{i} (-\alpha)^i e_i\left(\frac{x}{1 + \alpha x} \right) \sum_{j \ge k} \alpha^{j - k} \binom{j}{k} \sum_{\nu: c(\nu/\lambda) = j} (\alpha + \beta)^{|\nu/\lambda| - j} G_{\nu}^{(-\alpha, -\beta)}\\
&= \sum_{\nu} (\alpha + \beta)^{|\nu/\lambda| - c(\nu/\lambda)} \alpha^{c(\nu/\lambda) - k}\binom{c(\nu/\lambda)}{k} \sum_{i} (-\alpha)^i e_i\left(\frac{x}{1 + \alpha x} \right) G_{\nu}^{(-\alpha, -\beta)}\\
&= \sum_{\nu} (\alpha + \beta)^{|\nu/\lambda| - c(\nu/\lambda)} \alpha^{c(\nu/\lambda) - k}\binom{c(\nu/\lambda)}{k}\\
    &\qquad\times\sum_{i} (-\alpha)^i \sum_{\mu/\nu \text{ vert. strip}} (\alpha + \beta)^{|\mu/\nu| - i}\binom{|\mu/\nu| - c(\mu/\nu)}{i - c(\mu/\nu)} G_{\mu}^{(-\alpha, -\beta)}\\
&= \sum_{\nu} (\alpha + \beta)^{|\nu/\lambda| - c(\nu/\lambda)} \alpha^{c(\nu/\lambda) - k}\binom{c(\nu/\lambda)}{k}\sum_{\mu/\nu \text{ vert. strip}} G_{\mu}^{(-\alpha, -\beta)} (-\alpha)^{c(\mu/\nu)}\\
    &\qquad\times\sum_{i \ge c(\mu/\nu)} (\alpha + \beta)^{|\mu/\nu| - i} (-\alpha)^{i - c(\mu/\nu)}\binom{|\mu/\nu| - c(\mu/\nu)}{i - c(\mu/\nu)}\\
&= \sum_{\nu} (\alpha + \beta)^{|\nu/\lambda| - c(\nu/\lambda)} \alpha^{c(\nu/\lambda) - k}\binom{c(\nu/\lambda)}{k}\\
    &\qquad\times\sum_{\mu/\nu \text{ vert. strip}} G_{\mu}^{(-\alpha, -\beta)} (-\alpha)^{c(\mu/\nu)} (\alpha + \beta - \alpha)^{|\mu/\nu| - c(\mu/\nu)}\\
&= \sum_{\mu} G_{\mu}^{(-\alpha, -\beta)}\sum_{\mu/\nu \text{ vert. strip}}(\alpha + \beta)^{|\nu/\lambda| - c(\nu/\lambda)} \alpha^{c(\nu/\lambda) - k}\binom{c(\nu/\lambda)}{k}  (-\alpha)^{c(\mu/\nu)} \beta^{|\mu/\nu| - c(\mu/\nu)}\\
&= \sum_{\mu} v^{k}_{\mu/\lambda}(\alpha, \beta) G_{\mu}^{(-\alpha, -\beta)}.
\end{align*}
The last step uses Lemma \ref{vc}. The second Type 3 formula \eqref{ge} now implies by applying the involution $\omega$ to what we just proved. 
\end{proof}

\subsection{Branching formulas for $g^{(\alpha, \beta)}_{\lambda}$ and proof of Theorem \ref{cg}}

Let us first compute the weight generating function for rim border tableaux (RBT, Definition~\ref{rrpp}) at single variable $z$, and as a result it will correspond to a formula for $g^{(\alpha, \beta)}_{\lambda/\mu}(z)$. 

\begin{lemma}\label{gl}
For $\mu \subseteq \lambda$, we have
$$
\sum_{T \in RBT(\lambda/\mu)} w_T(\alpha, \beta) z^{\#\{\text{rim hooks in } T\}}=\beta^{r(\lambda/\mu) - b(\lambda/\mu)} (\alpha + \beta)^{i(\lambda/\mu)} z^{b(\lambda/\mu)} (z+\alpha)^{c(\lambda/\mu) - b(\lambda/\mu)}.
$$ 
\end{lemma}

\begin{proof}
We use Definition~\ref{rrpp} of RBT and try to put an element ($z$) in a shape $\lambda/\mu$. Each connected component of $\lambda/\mu$ contains at least one rim hook (from the border), which gives the factor $z^{b(\lambda/\mu)}$. The factor $(\alpha + \beta)^{i(\lambda/\mu)}$ arises from the inner part, also counted in $w_T(\alpha, \beta)$. All rim hooks are produced by the arbitrary choice from the upper border (consisting of $c(\lambda/\mu) - b(\lambda/\mu)$ remaining elements) of factors $z$ or $\alpha$. The remaining factor $\beta^{r(\lambda/\mu) - b(\lambda/\mu)}$ corresponds to the heights of rim hooks. 
\end{proof}

\begin{theorem}
We have
$$
g^{(\alpha, \beta)}_{\lambda}(x_1, \ldots, x_n, x) = \sum_{\mu} g^{(\alpha, \beta)}_{\mu}(x_1, \ldots, x_n) g^{(\alpha, \beta)}_{\lambda/\mu}(x),
$$
where the function $g^{(\alpha, \beta)}_{\lambda/\mu}(x)$ of single variable $x$ is defined as follows\footnote{The function $g^{(\alpha, \beta)}_{\lambda/\mu}(x)$ given by that explicit formula coincides exactly with combinatorial definition for $g^{(\alpha, \beta)}_{\lambda}$ extended for a skew shape.}:
$$
g^{(\alpha, \beta)}_{\lambda/\mu}(x) = 
\begin{cases}
\beta^{r(\lambda/\mu) - b(\lambda/\mu)} (\alpha + \beta)^{i(\lambda/\mu)} x^{b(\lambda/\mu)} (x+\alpha)^{c(\lambda/\mu) - b(\lambda/\mu)}, & \text{ if } \mu \subseteq \lambda,\\
0, & \text{otherwise}.
\end{cases}
$$
\end{theorem}

\begin{proof}
Take the Cauchy identity and use Type 3 Pieri formula for $G^{(\alpha, \beta)}_{\lambda}$ (Theorem \ref{pieri}), 
{
\begin{align*}
\sum_{\lambda} G^{(-\alpha, -\beta)}_{\lambda}(y) &g^{(\alpha, \beta)}_{\lambda}(x_1, \ldots, x_n, x) \\
&= \prod_{1 \le i \le n, 1 \le j} \frac{1}{1-x_i y_j} \prod_{1 \le j} \frac{1}{1 - x y_j}\\
&= \prod_{1 \le i \le n, 1 \le j} \frac{1}{1-x_i y_j} \sum_{0 \le k} x^k h_k(y)\\
&= \sum_{\mu} g^{(\alpha, \beta)}_{\mu}(x_1, \ldots, x_n) G^{(-\alpha, -\beta)}_{\mu}(y) \sum_{k \ge 0} x^k h_k(y)\\
&= \sum_{\mu} g^{(\alpha, \beta)}_{\mu}(x_1, \ldots, x_n) \sum_{k \ge 0} x^k \sum_{\lambda} v^k_{\lambda/\mu}(\alpha, \beta) G^{(-\alpha, -\beta)}_{\lambda}(y)\\
&= \sum_{\lambda} G^{(-\alpha, -\beta)}_{\lambda}(y) \sum_{\mu} g^{(\alpha, \beta)}_{\mu}(x_1, \ldots, x_n) \sum_{k} x^k v^k_{\lambda/\mu}(\alpha, \beta)
\end{align*}
}
Using \eqref{vi} we obtain
\begin{align*}
\sum_{k} x^k v^k_{\lambda/\mu}(\alpha, \beta) &= \sum_{k} x^k \beta^{r(\lambda/\mu) - b(\lambda/\mu)}(\alpha + \beta)^{i(\lambda/\mu)} \alpha^{c(\lambda/\mu) - k} \binom{c(\lambda/\mu) - b(\lambda/\mu)}{c(\lambda/\mu) - k}\\
&=\beta^{r(\lambda/\mu) - b(\lambda/\mu)} (\alpha + \beta)^{i(\lambda/\mu)} x^{b(\lambda/\mu)} (x+\alpha)^{c(\lambda/\mu) - b(\lambda/\mu)},
\end{align*}
which completes the proof.
\end{proof}

\begin{proof}[Proof of Theorem \ref{cg}]
Combinatorial formula for $g_{\lambda}^{(\alpha, \beta)}$ now follows from the previous Theorem and Lemma \ref{gl} by iteratively adding new variables. So RBT of shape $\lambda$ is constructed as $\emptyset = \lambda^{(0)} \subseteq \cdots \subseteq \lambda^{(k)} = \lambda$, where $\lambda^{(i + 1)}/\lambda^{(i)}$ is an RBT of single variable. 
\end{proof}

By combinatorial formula we can extend the polynomials $g_{\lambda/\mu}^{(\alpha, \beta)}$ to any skew shape $\lambda/\mu$. It is then easy to obtain the following formulas.

\begin{proposition}
[Branching formulas] The following properties hold
$$
g^{(\alpha, \beta)}_{\lambda}(x, x') = \sum_{\mu} g^{(\alpha, \beta)}_{\mu}(x) g^{(\alpha, \beta)}_{\lambda/\mu}(x'),
$$
where $x, x'$ are two sets of variables;
$$
g^{(\alpha, \beta)}_{\lambda}(x_1, \ldots, x_n, x) = \sum_{\lambda/\mu \text{ connected}} \alpha^{c(\lambda/\mu) - 1}\beta^{r(\lambda/\mu) - 1} (\alpha + \beta)^{i(\lambda/\mu)} x g_{\lambda/\mu}(x_1, \ldots, x_n).
$$
\end{proposition}

\subsection{Branching formulas for $G^{(\alpha, \beta)}_{\lambda}$}
We now state branching formulas for $G^{(\alpha, \beta)}_{\lambda}$. Similar formulas were given in \cite{buch} for $G_{\lambda}$. 
For a partition $\lambda = (\lambda_1, \lambda_2, \lambda_3, \ldots)$, denote $\bar\lambda = (\lambda_2, \lambda_3, \ldots)$.

\begin{proposition}\label{bg}
We have
$$
G^{(\alpha, \beta)}_{\lambda}(x_1, \ldots, x_n, x) = \sum_{\lambda/\mu \text{ hor. strip}}  G^{(\alpha, \beta)}_{\lambda/\mu}(x) G^{(\alpha, \beta)}_{\mu}(x_1, \ldots, x_n),
$$
where for a horizontal strip $\lambda/\mu$ we have
$$
G^{(\alpha, \beta)}_{\lambda/\mu}(x) =\left(\frac{x}{1 - \alpha x}\right)^{|\lambda/\mu|} \left(\frac{1 + \beta x}{1 - \alpha x} \right)^{r(\mu/\bar\lambda)}.
$$
\end{proposition}

\begin{proof}
The proof follows from the operator definition (\eqref{gskew}, Section \ref{dfg})
\begin{align*}
G^{(\alpha, \beta)}_{\lambda}(x_1, \ldots, x_n, x) &= \langle C(x) C(x_n) \cdots C(x_1) \cdot\varnothing,\lambda\rangle\\
&=\langle C(x) \sum_{\mu} G^{(\alpha, \beta)}_{\mu}(x_1, \ldots, x_n) \cdot \mu,\lambda\rangle\\
&= \sum_{\mu} G^{(\alpha, \beta)}_{\lambda/\mu}(x) G^{(\alpha, \beta)}_{\mu}(x_1, \ldots, x_n).
\end{align*}
Note that for a single variable $x$, $G^{(\alpha, \beta)}_{\lambda/\mu}(x) = 0$ if $\lambda/\mu$ is not a horizontal strip. Otherwise, $r(\mu/\bar\lambda)$ corresponds to the number of removable boxes of $\mu$ for which there is no box of $\lambda$ below them. In other words, the operators $u_i d_i$ can be applied here and then it is not hard to compute from the operator expansion \eqref{gskew} that $G^{(\alpha, \beta)}_{\lambda/\mu}(x) = \left(1 + \frac{(\alpha + \beta) x}{1 - \alpha x} \right)^{r(\mu/\bar\lambda)} \left(\frac{x}{1 - \alpha x}\right)^{|\lambda/\mu|}.$
\end{proof}

\subsection{The elements $g^{(\alpha, \beta)}_{(k)}, g^{(\alpha, \beta)}_{(1^k)}$}
\begin{proposition} We have the following formulas and the generating series
\begin{align}\label{42}
g^{(\alpha, \beta)}_{(k)} = \sum_{i = 1}^k \alpha^{k - i} \binom{k - 1}{i - 1} h_i, \qquad g^{(\alpha, \beta)}_{(1^k)} = \sum_{i = 1}^k  \beta^{k - i} \binom{k - 1}{i - 1} e_i.
\end{align}
\begin{align}\label{43}
1 + \sum_{k \ge 1} \left(\frac{t}{1 + \alpha t} \right)^k g^{(\alpha, \beta)}_{(k)} = \prod_{j \ge 1} \frac{1}{1 - x_j t}, \quad 1 + \sum_{k \ge 1} \left(\frac{t}{1 + \beta t} \right)^k g^{(\alpha, \beta)}_{(1^k)} = \prod_{j \ge 1} {(1 + x_j t)}.
\end{align}
\end{proposition}

\begin{proof}
The formulas for $g^{(\alpha, \beta)}_{(k)}, g^{(\alpha, \beta)}_{(1^k)}$ are easy to derive from combinatorial interpretations. The generating series are then implied from these formulas. 
\end{proof}

\begin{corollary}
From \eqref{43}, for any $n \in \mathbb{Z}_{\ge 0}$ we have 
\begin{equation}\label{ggk}
\sum_{i} g^{(\alpha, \beta)}_{(i)}(x) g^{(\beta, \alpha)}_{(1^{n - i})}(-x) = \delta_{n,0}.
\end{equation}
\end{corollary}

\begin{corollary}
From \eqref{42}, the elements $g^{(\alpha, \beta)}_{(k)}$ (and $g^{(\alpha, \beta)}_{(1^k)}$) are free generators of the polynomial ring $\Lambda$, we have 
$$\Lambda \otimes \mathbb{Z}[\alpha] \cong \mathbb{Z}[\alpha][g^{(\alpha, \beta)}_{(1)}, g^{(\alpha, \beta)}_{(2)}, \ldots], \qquad \Lambda \otimes \mathbb{Z}[\beta] \cong \mathbb{Z}[\beta][g^{(\alpha, \beta)}_{(1)}, g^{(\alpha, \beta)}_{(1^2)}, \ldots].$$
\end{corollary}

\subsection{Pieri formulas for $g^{(\alpha, \beta)}_{\lambda}$} 

\begin{proposition} 
The following formulas hold:

Type 1:
 $$g^{(\alpha, \beta)}_{(k)} g^{(\alpha, \beta)}_{\mu} = \sum_{\lambda/\mu \text{ hor. strip}} (-(\alpha + \beta))^{k - |\lambda/\mu|} \binom{r(\mu/\bar\lambda)}{k - |\lambda/\mu|} g^{(\alpha, \beta)}_{\lambda},$$
$$g^{(\alpha, \beta)}_{(1^k)} g^{(\alpha, \beta)}_{\mu} = \sum_{\lambda/\mu \text{ vert. strip}} (-(\alpha + \beta))^{k - |\lambda/\mu|} \binom{c(\mu'/\bar\lambda')}{k - |\lambda/\mu|} g^{(\alpha, \beta)}_{\lambda}.$$

Type 2:
$$h_{k} g^{(\alpha, \beta)}_{\mu} = \sum_{\lambda/\mu \text{ hor. strip}} q_{\lambda/\mu}{(\alpha, \beta)} g^{(\alpha, \beta)}_{\lambda}, \quad e_{k} g^{(\alpha, \beta)}_{\mu} = \sum_{\lambda/\mu \text{ vert. strip}} q_{\lambda'/\mu'}{(\beta, \alpha)} g^{(\alpha, \beta)}_{\lambda},$$
where for a horizontal strip $\lambda/\mu$ we define
$$
q_{\lambda/\mu}{(\alpha, \beta)} = \sum_{\ell} (-(\alpha + \beta))^{\ell - |\lambda/\mu|} \binom{r(\mu/\bar\lambda)}{\ell - |\lambda/\mu|}(-\alpha)^{k  - \ell} \binom{k - 1}{\ell - 1}
$$
\end{proposition}

\begin{proof}
Type 1,2 formulas can be obtained from the branching formulas for $G^{(\alpha, \beta)}_{\lambda}$ (Proposition~\ref{bg}). Note that Type 1 formulas are also followed from the known comultiplication formulas for $G^{(-\alpha, -\beta)}_{\lambda}$.
\end{proof}

\section{Schur expansions}\label{schur}

In this section we describe expansions of the (dual) bases $\{G^{(-\alpha, -\beta)}_{\lambda}\}, \{g^{(\alpha, \beta)}_{\lambda}\}$ in the basis $\{ s_{\lambda}\}$ of Schur functions. We give determinantal formulas for connection constants and combinatorial interpretations using nonintersecting lattice paths. In particular, we show that $g^{(\alpha, \beta)}_{\lambda}$ are Schur-positive (for $\alpha, \beta > 0$). 

\subsection{Cauchy-Binet formula for partitions} We will repeatedly use the following adaptation of the Cauchy-Binet determinantal formula.

\begin{lemma}\label{cb}
Let $a^{(i)}_{p, q}, b^{(i)}_{p, q}$ (for all $p, q \in \mathbb{Z}$ and $i \in \mathbb{Z}_{> 0}$) be elements of a commutative ring. 

For any partitions $\nu \subseteq \mu \subseteq \lambda$ and  $t \ge \ell(\lambda), \ell(\mu)$, let 
$$ 
F_{\lambda/\mu} = \det \left[a^{(i)}_{\lambda_i - i, \mu_j - j}\right]_{1 \le i,j \le t}, \quad
G_{\mu/\nu} = \det\left[b^{(j)}_{\mu_i - i, \nu_j - j}\right]_{1 \le i,j \le t}.
$$
Let
$
H_{\lambda/\nu} := \sum_{\mu} F_{\lambda/\mu} G_{\mu/\nu},
$
then
$$
H_{\lambda/\nu} = \det\left[c^{(i,j)}_{\lambda_i - i, \nu_j - j}\right]_{1 \le i,j \le t}, \quad \text{ where }\quad c^{(i,j)}_{p, q} = \sum_{k} a^{(i)}_{p, k} b^{(j)}_{k, q}.
$$
\end{lemma}

\subsection{Lattice paths and supplementary tableaux.}
In this subsection we define the coefficients $f^{(\alpha, \beta)}_{\mu/\nu}$ which appear in Schur expansions of $G^{(\alpha, \beta)}_{\lambda}, g^{(\alpha, \beta)}_{\lambda}$. 

{\it Type 1 grid.} Consider the lattice grid $\mathbb{Z}^2$ with the following assignment of edge weights:
\begin{itemize}
\item $w[(x,y) \to (x,y+1)] = 1$ for $x,y \in \mathbb{Z}$, all up steps;
\item $w[(x,y) \to (x+1,y)] = \alpha$ for $x < 0, y \ge 0$, right steps in 2nd quadrant;
\item $w[(x,y) \to (x+1,y+1)] = \beta$ for $x \ge 0, y \ge 0$, diagonal steps in 1st quadrant;
\item $w[(x,y) \to (x+1,y)] = \alpha + \beta$ for $x < 0, y < 0$, right steps in 3rd quadrant;
\item $w[(x,y) \to (x+1,y+1)] =\alpha+ \beta$ for $x \ge 0, y < 0$, diagonal steps in 4th quadrant;
\end{itemize}
and all other weights are $0$. See Figure \ref{fig3} (a).

\begin{figure}
{\small
\begin{tikzpicture}[scale = 0.5]

          \draw[gray,very thick, ->] (0, -3.5) to (0,3.5);
          \draw[gray, very thick] (-3.5,0) to (0,0);
          \draw[gray,dashed] (3.5,0) to (0,0);
          \foreach \x in {1,...,3} \draw[gray, thin] (\x, -3.5) to (\x,3.5);
          \foreach \x in {1,...,3} \draw[gray, thin] (-\x, -3.5) to (-\x,3.5);
          
          \foreach \x in {1,...,3} \draw[gray, thin] (-3.5, \x) to (0,\x);
          \foreach \x in {1,...,3} \draw[gray, thin] (-3.5, -\x) to (0,-\x);

          \foreach \x in {0,...,3} \draw[gray, thin] (0, \x) to (3.5-\x,3.5);

          \foreach \x in {1,...,3} \draw[gray, thin] (0, -\x) to (3.5,3.5-\x);
          \foreach \x in {1,...,3} \draw[gray, thin] (\x-0.5, -3.5) to (3.5,-\x+0.5);

          \draw[red, very thick] (-2, 0) to (-1, 0);
          \node at (-1.5,0.3) {$\alpha$};

          \draw[blue, very thick] (-2, -2) to (-1, -2);
          \node at (-1.5,-1.6) {$\alpha + \beta$};

          \draw[red, very thick] (1.0, 0) to (2.0, 1);
          \node at (1.4,0.9) {$\beta$};

          \draw[blue, very thick] (1.0, -2) to (2.0, -1);
          \node at (0.8,-1.3) {$\alpha+\beta$};
\end{tikzpicture}
\qquad\qquad
\begin{tikzpicture}[scale = 0.5]
          \draw[gray, ->] (0, -3.5) to (0,3.5);
          \draw[gray] (-3.5,0) to (0,0);
          \draw[gray,dashed] (3.5,0) to (0,0);
          \foreach \x in {1,...,3} \draw[gray, thin] (\x, -3.5) to (\x,3.5);
          \foreach \x in {1,...,3} \draw[gray, thin] (-\x, -3.5) to (-\x,3.5);
          
          \foreach \x in {1,...,3} \draw[gray, thin] (-3.5, \x) to (0,\x);
          \foreach \x in {1,...,3} \draw[gray, thin] (-3.5, -\x) to (0,-\x);

          \foreach \x in {0,...,3} \draw[gray, thin] (0, \x) to (3.5-\x,3.5);

          \foreach \x in {1,...,3} \draw[gray, thin] (0, -\x) to (3.5,3.5-\x);
          \foreach \x in {1,...,3} \draw[gray, thin] (\x-0.5, -3.5) to (3.5,-\x+0.5);
          
          \vertex[fill] at (1-3,1-1) {};
          \vertex[fill] at (2-3,1-2) {};
          \vertex[fill] at (3-3,1-3) {};

          \vertex[fill] at (1-2,2-1) {};
          \vertex[fill] at (2-2,2-2) {};
          \vertex[fill] at (3-1,3-1-1) {};
          
          \draw[very thick] (-2,0) to (1-3+1,1-1) to (1-3+1,1-1+1);
          \draw[very thick] (2-3,1-2) to (2-3 + 1,1-2) to (2-3+1,1-2+1);
          \draw[very thick] (3-3,1-3) to (3-3+1,1-3+1) to (3-3+1,1-3+1+1) to (3-3+1+1,1-3+1+1+1);
          
          \node at (-1.5,0.3) {$\alpha$};
          \node at (-0.9,-0.6) {\scriptsize$\alpha+\beta$};
          \node at (1.4,0.9) {$\beta$};
          \node at (1.3,-1.6) {\scriptsize$\alpha+\beta$};
\end{tikzpicture}
}

\caption{(a) A fragment of Type 1 lattice grid with sample weights; (b) A nonintersecting lattice path system $\mathcal{P}$ for $\mu = 333, \nu = 221$ and weight $w(\mathcal{P}) = \alpha (\alpha + \beta)^2 \beta$; here $f^{}_{\mu/\nu}(\alpha, \beta) = \alpha(\alpha + \beta) (3\alpha^2 + 9 \alpha\beta + 4\beta^2)$.}  \label{fig3}
\end{figure}
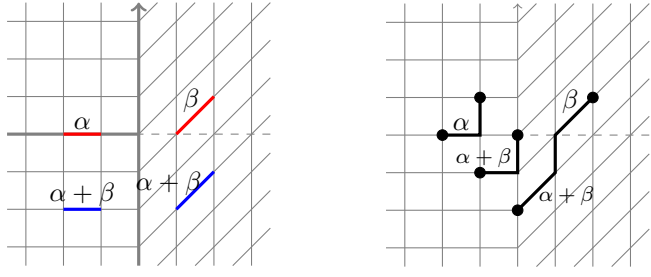

Let $\mu/\nu$ be a skew shape, $\ell = \ell(\mu)$ the length of $\mu$ and $d = d(\nu)$ the number of boxes on the main diagonal of $\nu$. 

Consider in this grid the system of nonintersecting lattice paths from the set of points $A = (A_1, \ldots, A_{\ell})$ to the set of points $B = (B_1, \ldots, B_{\ell})$, where $A_i = (i - \mu_i, 1 - i)$, $B_i = (i - \nu_i, \nu_i - i)$  for $i = 1, \ldots, d$ and $A_i = (i - \mu_i, 1 - \mu_i)$, $B_i = (i - \nu_i, i - \nu_i - 1)$ for $i = d+1, \ldots, \ell$. 
See examples in Figure \ref{fig3} (b) and Figure \ref{fig33} (a). 

For every lattice path system $\mathcal{P}$ from $A$ to $B$ the {\it weight} $w$ of $\mathcal{P}$ is defined as the product of weights of its edges. Define
\begin{equation}
f^{}_{\mu/\nu}(\alpha, \beta) := \sum_{\mathcal{P}: A \to B} w(\mathcal{P}),
\end{equation}
so that the sum runs over all nonintersecting path systems $\mathcal{P}$ from $A$ to $B$. In particular, $f^{}_{\mu/\nu}(\alpha, \beta)$ is a polynomial in $\alpha, \beta$ with positive integer coefficients. 

\begin{example}
For a lattice path system in Figure \ref{fig3} (b), it is easy to compute that $f^{}_{333/221}(\alpha, \beta) = \alpha(\alpha + \beta) (3\alpha^2 + 9 \alpha\beta + 4\beta^2)$.
\end{example}

One can obtain the symmetry $f^{}_{\mu/\nu}(\alpha, \beta) = f^{}_{\mu'/\nu'}(\beta, \alpha)$. In next subsection we will show that the numbers $f^{}_{\mu/\nu}(\alpha, \beta)$ are connection constants in the Schur expansion of $g^{(\alpha, \beta)}_{\lambda}.$ We now collect some technical details which will be useful later.

\begin{proposition}\label{prop8} We have
$$
f^{}_{\mu/\nu}(\alpha, \beta) = \det\left[ f^{(i)}_{\mu_i - i,\nu_j - j} \right]_{1 \le i,j\le \ell(\mu)},
$$
where
\begin{equation}\label{f}
f^{(i)}_{p,q} = 
    \begin{cases}
        \sum_{k} (\alpha + \beta)^k \binom{p+i-1}{k} \alpha^{p - q - k} \binom{p - k}{q}, & \text{ if } q \ge 0, i \le d;\\
        \sum_{k} (\alpha + \beta)^k \binom{p+i-1}{k} \beta^{-p - q - k} \binom{-q-1}{-p-q-k}, & \text{ if } q < 0, i > d;\\
        \sum_{k} (\alpha + \beta)^p \binom{p+k}{k} \sum_{m} (\alpha + \beta)^m \binom{i - k - 1}{m} \beta^{-q-m} \binom{-q-1}{-q-m}, & \text{ if } q < 0, i \le d;\\
        0, & \text{ otherwise.}
    \end{cases}
\end{equation}
\end{proposition}

\begin{proof}
From the given (Type 1) lattice grid it is easy to check that the total weight of paths going from $A_i$ to $B_j$ is exactly $f^{(i)}_{\mu_i - i,\nu_j - j}$ with $f^{(i)}_{p,q}$ given by the formula above (depending in which regions the points $A_i$ and $B_j$ lie). Then by the Lindstr\"om-Gessel-Viennot Lemma \cite{gv}, the determinant $\det\left[ f^{(i)}_{\mu_i - i,\nu_j - j} \right]_{1 \le i,j\le \ell}$ computes the number of nonintersecting lattice path systems from the set of points $A$ to the set $B$.
\end{proof}

{\it Type 2 grid.} The given path systems can be implemented on another transformed grid, see Figure \ref{fig33} (right). On this grid the left half-plane remains the same. The right half-plane allows only moves up and right; the weight of the right steps under $y = -x$ is $\alpha + \beta$ and the weight of the right steps up to that line is $\beta$. The points $A_i$ have coordinates $(i - \mu_i, 1 - i)$ and $B_i$ $(i - \nu_i, -1)$ for $i = d+1, \ldots, \ell.$ It is easy to see that this system generates the same weights and repeats nonintersecting path systems as on the previous grid (we just moved down some of the points $A_i$ and $B_i$ without affecting the total weight in each nonintersecting path system, see Figure \ref{fig33}). 

\begin{figure}
{\small
\begin{tikzpicture}[scale = 0.3]

          \draw[gray,thick,->] (0, -7.5) to (0,7.5);
          \draw[gray, thick] (-7.5,0) to (0,0);
          \draw[gray,dashed] (7.5,0) to (0,0);
          
          \draw[black,dashed] (-7.5,7.5) to (0,0);
          \draw[black,dashed] (1,0) to (7.5,6.5);
          
          \foreach \x in {1,...,7} \draw[gray, thin] (\x, -7.5) to (\x,7.5);
          \foreach \x in {1,...,7} \draw[gray, thin] (-\x, -7.5) to (-\x,7.5);
          
          \foreach \x in {1,...,7} \draw[gray, thin] (-7.5, \x) to (0,\x);
          \foreach \x in {1,...,7} \draw[gray, thin] (-7.5, -\x) to (0,-\x);

          \foreach \x in {0,...,7} \draw[gray, thin] (0, \x) to (7.5-\x,7.5);

          \foreach \x in {1,...,7} \draw[gray, thin] (0, -\x) to (7.5,7.5-\x);
          \foreach \x in {1,...,7} \draw[gray, thin] (\x-0.5, -7.5) to (7.5,-\x+0.5);
          
          \vertex[fill] at (-6,0) {};
          \vertex[fill] at (-4,-1) {};
          \vertex[fill] at (-1,-2) {};
          \vertex[fill] at (-0,-3) {};
          \vertex[fill] at (2,-2) {};
          \vertex[fill] at (4,-1) {};
          \vertex[fill] at (6,0) {};

          \vertex[fill] at (-3,3) {};
          \vertex[fill] at (-1,1) {};
          \vertex[fill] at (-0,0) {};
          \vertex[fill] at (2,1) {};
          \vertex[fill] at (4,3) {};
          \vertex[fill] at (6,5) {};
          \vertex[fill] at (7,6) {};
          
          \draw[very thick] (-6,0) to (-6+1,0) to (-6+1,0+1) to (-6+1+1,0+1) to (-6+1+1+1,0+1) to (-6+1+1+1,0+1+1) to (-6+1+1+1,0+1+1+1);
          
          \draw[very thick] (-4,-1) to (-4+1,-1) to (-4+1,-1+1) to (-4+1+1,-1+1) to (-4+1+1+1,-1+1) to (-4+1+1+1,-1+1+1);
          
          \draw[very thick] (-1,-2) to (-1,-2+1) to (-1+1,-2+1) to (-1+1,-2+1+1);
          
          \draw[very thick] (-0,-3) to (-0+1,-3+1) to (-0+1,-3+1+1) to (-0+1,-3+1+1+1) to (-0+1+1,-3+1+1+1+1);
          
          \draw[very thick] (2,-2) to (2,-2+1) to (2,-2+1+1) to (2+1,-2+1+1+1) to (2+1,-2+1+1+1+1) to (2+1+1,-2+1+1+1+1+1);
          
          \draw[very thick] (4,-1) to (4+1,-1+1) to (4+1,-1+1+2) to (4+1+1,-1+1+2+1) to (4+1+1,-1+1+2+1+2);
          
          \draw[very thick] (6,0) to (6,2) to (7,3) to (7,6);
          
          \node at (-3.5,-3.5) {$\alpha + \beta$};
          \node at (3.5,-3.5) {$\alpha + \beta$};
          \node at (-4.5,2.5) {$\alpha$};
          \node at (4.2,1.8) {$\beta$};
          
          \draw[thick, <->] (8.1,0) to (11,0);
\end{tikzpicture}
\begin{tikzpicture}[scale = 0.3]

          \draw[gray,thick,->] (0, -7.5) to (0,7.5);
          \draw[gray, thick] (-7.5,0) to (0,0);
          \draw[gray,dashed] (7.5,0) to (0,0);
          
          \draw[black,dashed] (-7.5,7.5) to (7.5,-7.5);
          \foreach \x in {1,...,7} \draw[gray, thin] (\x, -7.5) to (\x,7.5);
          \foreach \x in {1,...,7} \draw[gray, thin] (-\x, -7.5) to (-\x,7.5);
          \draw[gray, thin] (2,-7.5) to (2,7.5);
          \foreach \x in {1,...,7} \draw[gray, thin] (-7.5, \x) to (0,\x);
          \foreach \x in {1,...,7} \draw[gray, thin] (-7.5, -\x) to (7.5,-\x);

          \vertex[fill] at (-6,0) {};
          \vertex[fill] at (-4,-1) {};
          \vertex[fill] at (-1,-2) {};
          \vertex[fill] at (-0,-3) {};
          \vertex[fill] at (2,-2-2) {};
          \vertex[fill] at (4,-1-4) {};
          \vertex[fill] at (6,0-6) {};

          \vertex[fill] at (-3,3) {};
          \vertex[fill] at (-1,1) {};
          \vertex[fill] at (-0,0) {};
          \vertex[fill] at (2,-1) {};
          \vertex[fill] at (4,-1) {};
          \vertex[fill] at (6,-1) {};
          \vertex[fill] at (7,-1) {};
          
          \draw[very thick] (-6,0) to (-6+1,0) to (-6+1,0+1) to (-6+1+1,0+1) to (-6+1+1+1,0+1) to (-6+1+1+1,0+1+1) to (-6+1+1+1,0+1+1+1);
          
          \draw[very thick] (-4,-1) to (-4+1,-1) to (-4+1,-1+1) to (-4+1+1,-1+1) to (-4+1+1+1,-1+1) to (-4+1+1+1,-1+1+1);
          
          \draw[very thick] (-1,-2) to (-1,-2+1) to (-1+1,-2+1) to (-1+1,-2+1+1);
          
          \draw[very thick] (-0,-3) to (-0+1,-3) to (-0+1,-3+1) to (-0+1,-3+1+1) to (-0+1+1,-3+1+1);
          
          \draw[very thick] (2,-4) to (2,-4+2) to (2+1,-4+2) to (2+1,-4+2+1) to (2+1+1,-4+2+1);
          
          \draw[very thick] (4,-1-4) to (4+1,-1-4) to (4+1,-1-4+2) to (4+1+1,-1-4+2) to (4+1+1,-1-4+2+2);
          
          \draw[very thick] (6,-6) to (6,-6+2) to (6+1,-6+2) to (6+1,-6+2+3);

          \node at (-3.5,-3.5) {$\alpha + \beta$};
          \node at (2.2,-5.5) {$\alpha + \beta$};
          \node at (-4.5,2.5) {$\alpha$};
          \node at (4.2,-2.2) {$\beta$};      
\end{tikzpicture}
}

\caption{A skew shape $\mu/\nu = 7644321/43321$. A nonintersecting lattice path system on Type 1 grid associated to $\mu/\nu$ with weight $\alpha^5 \beta^{5} (\alpha + \beta)^4$; and a transformed (via a weight preserving bijection) nonintersecting lattice path system on a Type 2 grid.}  \label{fig33}
\end{figure}
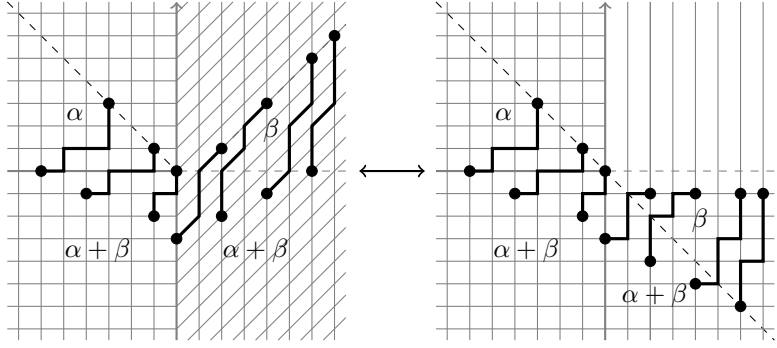

\begin{remark}
From lattice path interpretation (Type 1 or 2 grid) it is standard to define tableaux interpretations with some (restricted) properties which we omit here.
\end{remark}

We will also need the following two supplementary types of tableaux. 

\begin{definition}[$f_{\mu/\nu}$, elegant tableaux \cite{lenart, lp}]\label{elegant}
Denote by $f_{\mu/\nu}$ the number of semistandard tableaux of skew shape $\mu/\nu$ so that the row $i$ (of $\mu$) contains integers from $[1,i-1]$. In particular, if $f_{\mu/\nu} > 0$, the first row of $\mu/\nu$ is empty.
\end{definition}

Elegant tableaux arise from nonintersecting lattice path interpretation \cite{lenart} and the following determinantal formulas hold
\begin{equation}\label{ff}
f_{\mu/\nu}  = \det\left[ \binom{\mu_i - \nu_j + j - 2}{\mu_i - i -\nu_j + j} \right]_{1 \le i,j \le \ell(\mu)} = \det\left[\binom{\mu'_i - 1}{\mu'_i - i - \nu'_j + j} \right]_{1 \le i,j \le \ell(\mu')}. 
\end{equation}
Note that $f^{}_{\mu/\nu}(0,1) = f_{\mu/\nu}$. 

\begin{definition}[$\psi_{\lambda/\mu}$, dual hook tableaux \cite{molev}]\label{dht}
Suppose that $\lambda$ and $\mu$ have the same number of boxes $b$ on the main diagonal. Denote by $\psi_{\lambda/\mu}$ the number of fillings of $\lambda/\mu$ so that the elements in the first $b$ rows strictly decrease in rows and weakly in columns, and all elements in row $i$ are from the set $\{0, -1, \ldots, i - \lambda_i + 1 \}$ for $i = 1, \ldots, b$; elements of the first $b$ columns strictly increase in columns and weakly increase in rows, and all elements in the column $j$ are from the set $\{1, 2, \ldots, \lambda'_j - 1 \}$ for $j = 1, \ldots, b$.
\end{definition}

Dual hook tableaux can also be described using nonintersecting lattice paths and they satisfy the following determinantal formula \cite{molev}
\begin{align*}
\psi_{\lambda/\nu} &=  \det\left[ \binom{\lambda_i - i}{\lambda_i - i -\nu_j + j} \right]_{1 \le i,j \le d} \det\left[ \binom{i - \nu_j + j - 1}{\lambda_i - i -\nu_j + j} \right]_{d+1 \le i,j},
\end{align*}
which can be composed into a single determinant 
\begin{align}\label{psii}
\psi_{\lambda/\nu}&= \det\left[ \binom{\lambda_i - i}{\lambda_i - i -\nu_j + j} \right]_{1 \le i,j \le \ell(\lambda)}.
\end{align}

The following (nonpositive) formulas for $f_{\nu/\mu}(\alpha, \beta)$ will be useful for us in proving the Schur expansions. 

\begin{lemma}\label{ll} The following formulas hold:
\begin{equation}\label{ll1}
f_{\nu/\mu}(\alpha, \beta) = \sum_{\lambda} (\alpha + \beta)^{|\nu/\lambda|} \alpha^{|\lambda/\mu|} (-1)^{n(\lambda/\mu)} f_{\nu/\lambda} \psi_{\lambda/\mu},
\end{equation}
where $\lambda$ and $\mu$ have the same number $b$ of boxes on the main diagonal (see Definition~\ref{dht} of $\psi$), $n(\lambda/\mu)$ is the number of boxes in the first $b$ columns of $\lambda/\mu$;
\begin{equation}\label{ll2}
f_{\mu/\nu}(\alpha, \beta) = \det\left[\widetilde{f}^{(i)}_{\mu_i - i,\nu_j - j} \right]_{1 \le i,j \le \ell},
\end{equation}
where
\begin{equation}
\widetilde{f}^{(i)}_{p,q} = \sum_{k = -\infty}^{\infty} \binom{p+i- k - 1}{i-1}\binom{k}{k - q} (\alpha + \beta)^{p - k} \alpha^{k - q}, 
\end{equation}
or equivalently
$$
\sum_{k \ge 0}\widetilde{f}^{(i)}_{q+k,q} t^k = \frac{1}{(1 - (\alpha + \beta)t)^i}\frac{1}{(1 - \alpha t)^{q + 1}}.$$
\end{lemma}

\begin{figure}
{\small
\begin{tikzpicture}[scale = 0.3]
          \draw[gray,thick,->] (0, -7.5) to (0,7.5);
          \draw[gray, thick] (-7.5,0) to (0,0);
          \draw[gray,dashed] (7.5,0) to (0,0);
          
          \draw[black,dashed] (-7.5,7.5) to (0,0);
          \draw[black,dashed] (1,0) to (7.5,6.5);
          
          \foreach \x in {1,...,7} \draw[gray, thin] (\x, -7.5) to (\x,7.5);
          \foreach \x in {1,...,7} \draw[gray, thin] (-\x, -7.5) to (-\x,7.5);
          
          \foreach \x in {1,...,7} \draw[gray, thin] (-7.5, \x) to (0,\x);
          \foreach \x in {1,...,7} \draw[gray, thin] (-7.5, -\x) to (7.5,-\x);

          \foreach \x in {0,...,7} \draw[gray, thin] (0, \x) to (7.5-\x,7.5);

          \foreach \x in {1,...,7} \draw[gray, thin] (\x, 0) to (7.5,7.5-\x);
          \vertex[fill] at (-6,0) {};
          \vertex[fill] at (-4,-1) {};
          \vertex[fill] at (-1,-2) {};
          \vertex[fill] at (-0,-3) {};
          \vertex[fill] at (2,-2-2) {};
          \vertex[fill] at (4,-1-4) {};
          \vertex[fill] at (6,0-6) {};

          \vertex[fill] at (-3,3) {};
          \vertex[fill] at (-1,1) {};
          \vertex[fill] at (-0,0) {};
          \vertex[fill] at (2,1) {};
          \vertex[fill] at (4,3) {};
          \vertex[fill] at (6,5) {};
          \vertex[fill] at (7,6) {};
          
          \draw[very thick] (-6,0) to (-6+1,0) to (-6+1,0+1) to (-6+1+1,0+1) to (-6+1+1+1,0+1) to (-6+1+1+1,0+1+1) to (-6+1+1+1,0+1+1+1);
          
          \draw[very thick] (-4,-1) to (-4+1,-1) to (-4+1,-1+1) to (-4+1+1,-1+1) to (-4+1+1+1,-1+1) to (-4+1+1+1,-1+1+1);
          
          \draw[very thick] (-1,-2) to (-1,-2+1) to (-1+1,-2+1) to (-1+1,-2+1+1);
          
          \draw[very thick, dashed] (2,-4) to (2,-4+2);
          \draw[very thick, dashed] (2,-4+2) to (2+1,-4+2) to (2+1,-4+2+1);
          \draw[very thick, dashed] (2+1,-4+2+1) to (2+1,-4+2+1+3) to (2+1+1,-4+2+1+3+1);

          \node at (-3.5,-3.5) {$\alpha + \beta$};
          \node at (4.3,-3.5) {$\alpha + \beta$};
          \node at (-4.5,2.5) {$\alpha$};
          \node at (2.0,2.5) {$-\alpha$};
\end{tikzpicture}
}

\caption{Type 3 lattice grid with some negative diagonal weights $-\alpha$ in the 1st quadrant, $\mu/\nu = 7644321/43321$.}  \label{fig333}
\end{figure}

\begin{proof}
Note that using determinantal formulas \eqref{ff}, \eqref{psii}, the formula \eqref{ll1} is equivalent to the determinantal formula \eqref{ll2} by applying the Cauchy-Binet formula indexed by partitions. 

We will perform one more transformation of the lattice grid (now with some negative weights) whose nonintersecting path systems compute the same function $f_{\mu/\nu}(\alpha, \beta)$. 

Consider the new {\it Type 3} lattice grid as in Figure \ref{fig333}. Which is essentially the same as the Type 2 grid, but diagonal steps in its 1st quadrant have negative weights $-\alpha$ and all right steps in 4th quadrant have weight $\alpha + \beta$. Points $A_i$ have the same coordinates, but the points $B_i$ on the right half-plane move again as in Type 1 lattice grid, i.e. they have the coordinates $(i - \nu_i, i - \nu_i - 1)$. 

By the following two cases we obtain that 
$$
w(A_i \to B_j) = \widetilde{f}^{(i)}_{\mu_i - i,\nu_j - j}.
$$
If $\nu_j - j \ge 0$, then
$$
\widetilde{f}^{(i)}_{\mu_i - i,\nu_j - j} = \sum_k \binom{\mu_i - k - 1}{i-1}\binom{k}{k - \nu_j + j} (\alpha + \beta)^{\mu_i - i - k} \alpha^{k - \nu_j + j}.
$$
If $ j - \nu_j > 0$, then
$$
\widetilde{f}^{(i)}_{\mu_i - i,\nu_j - j} = \sum_k \binom{k + \mu_i - 1}{i - 1} \binom{j -\nu_j - 1}{j - \nu_j - k} (\alpha + \beta)^{k + \mu_i - i} (-\alpha)^{j - \nu_j - k}.
$$

It is then not hard to see that the total weighted sum from $A$ to $B$ on this new Type 3 grid gives the same function $f_{\mu/\nu}(\alpha, \beta)$. 
\end{proof}

\subsection{Schur expansions for $G^{(\alpha, \beta)}_{\lambda}$, $g^{(\alpha, \beta)}_{\lambda}$}

\begin{figure} 
{Table 1. Schur expansions of $g^{(\alpha, \beta)}_{\lambda}$ for some $\lambda$.}

{\scriptsize
\begin{align*}
g^{(\alpha, \beta)}_{(k)} =\ &\sum_{i = 1}^k \binom{k - 1}{i - 1} h_i \alpha^{k - i} \qquad g^{(\alpha, \beta)}_{(1^k)} = \sum_{i = 1}^k \binom{k - 1}{i - 1} e_i \beta^{k - i}\\
g^{(\alpha, \beta)}_{21} =\ &\alpha \beta s_{1} + \beta s_{2} + \alpha s_{11} + s_{21}\\
g^{(\alpha, \beta)}_{31} =\ &\alpha^2 \beta s_{1} + 2\alpha \beta s_{2} + \alpha^2 s_{11}  + \beta s_{3} + 2\alpha s_{21} + s_{31}\\
g^{(\alpha, \beta)}_{22} =\ &\alpha \beta (\alpha + \beta) s_{1} + \beta (\alpha + \beta) s_{2} + \alpha(\alpha + \beta) s_{11} + (\alpha + \beta) s_{21}  + s_{22}\\
g^{(\alpha, \beta)}_{32} =\ &\alpha^2\beta(\alpha + \beta) s_{1} +  2\alpha\beta(\alpha+\beta) s_{2} + \alpha^2 (\alpha + \beta) s_{11} + \beta(\alpha+\beta) s_{3} + 2\alpha(\alpha+\beta) s_{21} \\
	&+ (\alpha + \beta) s_{31} + 2\alpha s_{22} + s_{32}\\
g^{(\alpha, \beta)}_{321} =\ &\alpha^2 \beta^2 (\alpha + \beta) s_{1} + 2\alpha \beta^2(\alpha + \beta) s_{2} + 2\alpha^2 \beta(\alpha+\beta) s_{11} + \beta^2(\alpha+\beta) s_{3} + 4\alpha\beta(\alpha + \beta) s_{21} \\
	&+ \alpha^2(\alpha + \beta) s_{111}+ 2\beta(\alpha + \beta) s_{31} + 2\alpha (\alpha + \beta) s_{211} + 4\alpha \beta s_{22} + 2\beta s_{32} + 2\alpha s_{221} + (\alpha + \beta)s_{311}+ s_{321}\\
g^{(\alpha, \beta)}_{33} =\ &\alpha^2\beta(\alpha+\beta)^2 s_{1} + 2\alpha\beta(\alpha+\beta)^2 s_{2} + \alpha^2(\alpha+\beta)^2 s_{11} + \beta(\alpha +\beta)^2 s_{3} + 2\alpha (\alpha + \beta)^2 s_{21} \\ 
	&+ (\alpha + \beta)^2 s_{31} + \alpha (3\alpha + 2\beta) s_{22} + (2\alpha + \beta) s_{32} + s_{33}\\
g^{(\alpha, \beta)}_{333} =\ & \alpha^2\beta^2(\alpha+\beta)^4 s_{1} + 2\alpha\beta^2(\alpha+\beta)^4 s_{2} + 2\alpha^2\beta(\alpha+\beta)^4 s_{11} + \beta^2(\alpha+\beta)^4 s_{3} + 4\alpha\beta(\alpha + \beta)^4 s_{21} \\
	&+ \alpha^2(\alpha+\beta)^4 s_{111} + 2\beta(\alpha + \beta)^4 s_{31} + 2\alpha (\alpha+\beta)^4 s_{211} + 2\alpha\beta(\alpha+\beta)(3\alpha^2 + 7\alpha\beta + 3\beta^2) s_{22} \\
	&+ \beta(\alpha+\beta)(4\alpha^2 + 9\alpha\beta + 3\beta^2) s_{32} + (\alpha+\beta)^4 s_{311} + \alpha(\alpha+\beta)(3\alpha^2 + 9\alpha\beta + 4\beta^2) s_{221}\\ 
	&+ 2\beta(\alpha+\beta)(\alpha + 2\beta) s_{33} + 2(\alpha + \beta)(\alpha^2 + 3\alpha\beta + \beta^2) s_{321} + 2\alpha(\alpha+\beta)(2\alpha + \beta) s_{222} \\
	&+(\alpha+\beta)(\alpha + 3\beta) s_{331}+(\alpha+\beta)(3\alpha + \beta) s_{322}+2(\alpha+\beta)s_{332} + s_{333}
\end{align*}
}
\end{figure}

\begin{theorem}\label{schur1} The following dual formulas hold
\begin{equation}
s_{\mu} = \sum_{\nu} f_{\nu/\mu}{(\alpha, \beta)} G^{(-\alpha, -\beta)}_{\nu}, \qquad g^{(\alpha, \beta)}_{\mu} = \sum_{\nu} f_{\mu/\nu}{(\alpha, \beta)} s_{\nu}.
\end{equation}
\end{theorem}

\begin{proof}
First recall the following expansion given in \cite{lenart} (restated with parameter $\beta$ here)
\begin{equation}
s_{\lambda} = \sum_{\mu} (-\beta)^{|\mu/\lambda|} f_{\mu/\lambda}  G^{\beta}_{\mu},
\end{equation}
where $f_{\mu/\lambda}$ is the number of elegant tableaux of $\mu/\lambda$ (see Definition \ref{elegant}). 
Therefore,
\begin{equation}\label{x}
s_{\lambda}\left(\frac{x}{1 + \alpha x} \right) = \sum_{\nu} (\alpha + \beta)^{|\nu/\lambda|} f_{\nu/\lambda}   G^{(-\alpha, -\beta)}_{\nu}
\end{equation}
We have (from \cite{molev} Theorem 3.20, specializing all $a_i = \alpha$)  
\begin{equation}\label{x2}
s_{\mu} = \sum_{\lambda} (-1)^{n(\lambda/\mu)} \alpha^{|\lambda/\mu|} \psi_{\lambda/\mu} s_{\lambda}\left(\frac{x}{1 + \alpha x} \right),
\end{equation}
where $\psi_{\lambda/\mu}$ is the number of dual hook tableaux (see Definition \ref{dht}); here $\lambda, \mu$ have the same number $b$ of boxes on the main diagonal, $n(\lambda/\mu)$ is the number of boxes in the lower component of $\lambda/\mu$.
Combining \eqref{x}, \eqref{x2} we thus have 
\begin{align*}
s_{\mu} &= \sum_{\lambda} (-1)^{n(\lambda/\mu)} \alpha^{|\lambda/\mu|} \psi_{\lambda/\mu} \sum_{\nu}  (\alpha + \beta)^{|\nu/\lambda|} f_{\nu/\lambda}   G^{(-\alpha, -\beta)}_{\nu} \\
&= \sum_{\nu} G^{(-\alpha, -\beta)}_{\nu} \sum_{\lambda} (\alpha + \beta)^{|\nu/\lambda|} (-1)^{n(\lambda/\mu)} \alpha^{|\lambda/\mu|} f_{\nu/\lambda} \psi_{\lambda/\mu} \\
&= \sum_{\nu} f_{\nu/\mu}{(\alpha, \beta)} G^{(-\alpha, -\beta)}_{\nu}.
\end{align*}
The last step followed by Lemma \ref{ll}.
The Schur expansion of the dual $g^{(\alpha, \beta)}_{\lambda}$ polynomials implies from the duality of families $\{g^{(\alpha, \beta)}_{\lambda}\}$ and $\{G^{(-\alpha, -\beta)}_{\lambda}\}$. 
\end{proof}

\begin{corollary}
The dual polynomials $g^{(\alpha, \beta)}_{\lambda}$ are Schur-positive ($f_{\mu/\nu}(\alpha, \beta) \in \mathbb{Z}_{\ge 0}[\alpha, \beta]$). 
\end{corollary}

\section{Jacobi-Trudi type identities}\label{jt}

In this section we prove Jacobi-Trudi type determinantal identities for the functions $G^{(\alpha, \beta)}_{\lambda}$, $g^{(\alpha, \beta)}_{\lambda}$.

\begin{theorem}
The following determinantal formulas hold:

\begin{equation}\label{56}
G^{(\alpha, \beta)}_{\lambda}(x_1, \ldots, x_n) = \det\left[ \widetilde{h}^{(i)}_{\lambda_i - i + j} \right]_{1 \le i,j \le n},  
\end{equation}
where
$$
\widetilde{h}^{(i)}_{p} = \sum_{k} (\alpha + \beta)^k \binom{i-1}{k} h_{p+k}\left(\frac{x_1}{1 - \alpha x_1}, \ldots, \frac{x_n}{1 - \alpha x_n}  \right),
$$
\begin{equation}
G^{(\alpha, \beta)}_{\lambda}(x_1, \ldots, x_n) = \det\left[ \widetilde{e}^{(i)}_{\lambda'_i - i + j} \right]_{1 \le i,j \le n},  
\end{equation}
$$
\widetilde{e}^{(i)}_{p} = \sum_{k} (\alpha + \beta)^k \binom{i-1}{k} e_{p+k}\left(\frac{x_1}{1 + \beta x_1}, \ldots, \frac{x_n}{1 + \beta x_n} \right).
$$
\end{theorem}

\begin{proof}
For $G^{\beta}_{\lambda}$, the formula \eqref{56} was given in \cite{kir1} and hence implies here via $\beta \to \alpha + \beta$ and $x_i \to \frac{x_i}{1 - \alpha x_i}$. The second formula is a dual version. 
\end{proof}

\begin{theorem} 
\label{jtg} The following determinantal identities hold:
\begin{equation}
g^{(\alpha, \beta)}_{\lambda} = \det \left[\widetilde{g}^{(i)}_{\lambda_i - i, j} \right]_{1 \le i,j \le \ell(\lambda)}, \qquad \widetilde{g}^{(i)}_{p, j} = \sum_{k} \widetilde{f}^{(i)}_{p, k} h_{k+j},
\end{equation}
\begin{equation}
g^{(\beta, \alpha)}_{\lambda} = \det \left[\bar{g}^{(i)}_{\lambda'_i - i, j} \right]_{1 \le i,j \le \ell(\lambda')}, \qquad \bar{g}^{(i)}_{p, j} = \sum_{k} \widetilde{f}^{(i)}_{p, k} e_{k+j},
\end{equation}
where 
$$
\widetilde{f}^{(i)}_{p, q} = \sum_{k = -\infty}^{\infty} \binom{p+i- k - 1}{i-1}\binom{k}{k - q} (\alpha + \beta)^{p - k} \alpha^{k - q}.
$$
\end{theorem}

Note: Instead of $\widetilde{f}^{(i)}_{p, k}$ (Lemma~\ref{ll}) we can take positive expressions ${f}^{(i)}_{p, k}$ from Proposition~\ref{prop8}, \eqref{f}.

\begin{proof}
The formulas imply from the Cauchy-Binet formula,  Schur expansion in Theorem~\ref{schur1}, and determinantal expressions for the coefficients $f_{\mu/\nu}(\alpha, \beta)$ (Lemma~\ref{ll}, Proposition~\ref{prop8}). The second (dual) formula implies by applying the involution $\omega$.
\end{proof}

\begin{corollary}
Let
\begin{align}
 h^{(m)}_{n}(\beta) &:= h_{n}(\underbrace{\beta, \ldots, \beta}_{m \text{ times}}, x_1, x_2, \ldots) = \sum_{k \ge 0} \beta^k \binom{m + k -1}{k} h_{n - k}, \\ 
 e^{(m)}_{n}(\beta) &:= e_{n}(\underbrace{\beta, \ldots, \beta}_{m \text{ times}}, x_1, x_2, \ldots) = \sum_{k \ge 0} \beta^k \binom{m}{k} e_{n - k}.
\end{align}
For $(\alpha, \beta) = (0,\beta)$ the determinantal formulas above can be refined to the following
\begin{align}
g^{(0,\beta)}_{\lambda} = \det\left[ h^{(i -1)}_{\lambda_i - i + j}(\beta) \right]_{1 \le i,j \le \ell(\lambda)} = \det\left[ e^{(\lambda'_i -1)}_{\lambda'_i - i + j}(\beta) \right]_{1 \le i,j \le \ell(\lambda')}
\end{align}
($g^{(0,1)} = g_{\lambda}$ and for $g_{\lambda}$ the last formula was given in \cite{sz}); for $(\alpha, \beta) = (\alpha,0)$ we have
\begin{align}
g^{(\alpha,0)}_{\lambda} &= \omega(g^{(0, \alpha)}_{\lambda'})= \det\left[ \omega(h^{(i -1)}_{\lambda'_i - i + j}(\alpha)) \right]_{1 \le i,j \le \ell(\lambda')} = \det\left[ \omega(e^{(\lambda_i -1)}_{\lambda_i - i + j}(\alpha)) \right]_{1 \le i,j \le \ell(\lambda)}
\end{align}
For $\alpha + \beta = 0$ we have
\begin{align}
g^{(\alpha, -\alpha)}_{\lambda} &= \det\left[\sum_{k}\binom{\lambda_i - i}{k} \alpha^k h_{k + j} \right]_{1 \le i,j \le \ell(\lambda)} = \det\left[\sum_{k}\binom{\lambda'_i - i}{k} \alpha^k e_{k + j} \right]_{1 \le i,j \le \ell(\lambda')}
\end{align}
\end{corollary}

\begin{remark}[On Giambelli hook formulas]
Jacobi-Trudi identities usually imply Giambelli formulas (determinantal formulas indexed by hooks), which is not (directly) the case here for the $G, g$ functions. 
For example, $G^{(0,1)}_{22}(x_1, x_2) = x_1^2 x_2^2,$ but
$$
\det\begin{bmatrix} G^{(0, 1)}_{21}(x_1, x_2) & G^{(0, 1)}_{2}(x_1, x_2) \\ G^{(0, 1)}_{11}(x_1, x_2) & G^{(0, 1)}_{1}(x_1, x_2) \end{bmatrix} = x_1^2 x_2^2 (x_1 + 1)(x_2 + 1). 
$$
Similarly for dual $g$ polynomials computation shows that $g^{(\alpha,\beta)}_{22}=(\alpha + \beta) g^{(\alpha, \beta)}_{21} + s_{22},$ but
$$
\det\begin{bmatrix} g^{(\alpha, \beta)}_{21} & g^{(\alpha, \beta)}_{2} \\ g^{(\alpha, \beta)}_{11} & g^{(\alpha, \beta)}_{1} \end{bmatrix} = g^{(\alpha, \beta)}_{21} g^{(\alpha, \beta)}_{1} - g^{(\alpha, \beta)}_{2} g^{(\alpha, \beta)}_{11} = s_{22}.
$$
(Schur expansions of $g^{(\alpha, \beta)}_{\lambda}$ are provided in Table 1.)
Lascoux and Naruse in \cite{ln} gave the following analog of Giambelli's formula for $g_{\lambda}$
$$
g_{\lambda} = \det\left[\widetilde{g}^{(i,j)}_{(p_i | q_j)} \right]_{1 \le i,j \le d},
$$
where $\lambda = (p_1, \ldots, p_d | q_1, \ldots, q_d)$ written in Frobenius notation, and 
$$
\widetilde{g}^{(i,j)}_{(p | q)} = \sum_{k, s} \binom{k + i - 2}{k} \binom{s + j - 2}{s} g_{(p - k | q - s)} + \sum_{t} \binom{p + i - t}{p}\binom{q + j - t}{q}.
$$
So it would be interesting to see what versions of Giambelli identities hold for $G^{(\alpha, \beta)}_{\lambda}, g^{(\alpha, \beta)}_{\lambda}$.
\end{remark}

\subsection{Combinatorial proofs of determinantal formulas}
Roughly speaking, if connection constants in Schur expansions have determinantal expressions, then by the Cauchy-Binet formula, the whole expression is also some determinant (the Schur function satisfies the usual Jacobi-Trudi formula). So, if combinatorial lattice interpretation is known for connection coefficients, this approach also provides a combinatorial proof for the corresponding determinantal identity (the Cauchy-Binet formula can be proved using the Lindstr\"om-Gessel-Viennot Lemma \cite{gv} by `gluing' two graphs together, sinks of the first to the sources of the second). 

For the dual stable Grothendieck polynomials $g_{\lambda},$ the determinantal formula 
\begin{equation}
g_{\lambda} =\det\left[ e^{(\lambda'_i -1)}_{\lambda'_i - i + j}(1) \right]_{1 \le i,j \le \ell(\lambda')} =  \det\left[\sum_{k = 0}^{\lambda'_i - 1} \binom{\lambda'_i-1}{k} e_{\lambda'_i - i + j - k} \right]_{1 \le i,j \le \ell(\lambda')}
\end{equation}
is implied from Theorem \ref{jtg}. 

To illustrate the idea how to pass to combinatorial constructions, we now show the lattice path interpretation which proves this formula for $g_{\lambda}$. Something similar can be designed for a general case using the lattice path interpretation of the coefficients $f^{(\alpha, \beta)}_{\mu/\nu}$. 

Let us prove the identity for a finite number of variables $x_1, \ldots, x_n$. Consider the lattice grid with allowed up and diagonal unit steps. All up steps have weights $1$, the upper half-plane  has diagonal weights $x_{n - j + 1}$ for passing from $y$-coordinate $j - 1$ to $j$, and the lower half-plane has diagonal weights $1$. 

Let $\ell = \ell(\lambda')$ and consider nonintersecting lattice path systems from the points $A = (A_1, \ldots, A_{\ell})$ to the points $B = (B_1, \ldots, B_{\ell})$, where $A_i$ has coordinates $(-\lambda'_i + i - 1, 1 - \lambda'_i)$ and $B_j$ has coordinates $(j - 1, n)$. See Figure \ref{fff}. It is easy to see that 
$$
w(A_i \to B_j) = \sum_{k = 0}^{\lambda'_i - 1} \binom{\lambda'_i-1}{k} e_{\lambda'_i - i + j - k}.
$$
Hence, $\det\left[w(A_i \to B_j) \right]$ is a weighted sum over all nonintersecting lattice path systems from $A$ to $B$. 

\begin{figure}
\begin{tikzpicture}[scale = 0.4]
          \draw[gray, thick, dashed] (0,0) to (7.5,0);
          \vertex[fill] at (5+1,4) {};

          \vertex[fill] at (0+1,-4) {};

          \draw[very thick, blue] (0+1,-4) to (0+2+1,-4+2) to (0+2+1,-4+2+1) to (0+2+1+1,-4+2+1+1);
          \draw[very thick] (0+2+1+1,-4+2+1+1) to (0+2+1+1,-4+2+1+1+1) to (0+2+1+1+1,-4+2+1+1+1+1) to (0+2+1+1+1,-4+2+1+1+1+1+1) to (0+2+1+1+1+1,-4+2+1+1+1+1+1+1);
          
          \node at (0.2,-4) {$A_i$};
          \node at (6.7, 4) {$B_j$};
          \draw[<->] (1,0) to (4,0);
          \node at (2.2,0.5) {\scriptsize$\lambda'_i - 1 - k$};
          \node at (6.7, 1.5) {$e_{k + 1 - i + j}$};
          
          \draw[<->] (1,0) to (1,-4);
          \node at (0,-2) {\scriptsize$\lambda_i' - 1$};
          
          \node at (4.2,-2) {$\binom{\lambda'_i - 1}{k}$};
\end{tikzpicture}
\begin{tikzpicture}[scale = 0.4]
          \draw[gray,thick,->] (5, -4.5) to (5,5.0);
          \draw[gray, thick, dashed] (0,0) to (7.5,0);
          \foreach \x in {0,...,7} \draw[gray, thin] (\x, -4.5) to (\x,4.5);
          \foreach \x in {0,...,4} \draw[gray, thin] (0, \x) to (4.5-\x,4.5);
          \foreach \x in {1,...,3} \draw[gray, thin] (0, -\x) to (4.5+\x,4.5);
          \draw[gray, thin] (0, -4) to (7.5,3.5);
          \draw[gray, thin] (0.5, -4.5) to (7.5,2.5);
          \draw[gray, thin] (1.5, -4.5) to (7.5,1.5);
          \draw[gray, thin] (2.5, -4.5) to (7.5,0.5);
          \draw[gray, thin] (3.5, -4.5) to (7.5,-0.5);
          \draw[gray, thin] (4.5, -4.5) to (7.5,-1.5);
          \draw[gray, thin] (5.5, -4.5) to (7.5,-2.5);
          \draw[gray, thin] (6.5, -4.5) to (7.5,-3.5);
                                        
          \vertex[fill] at (5,4) {};
          \vertex[fill] at (6,4) {};
          \vertex[fill] at (7,4) {};
          
          \vertex[fill] at (0,-4) {};
          \vertex[fill] at (3,-2) {};
          \vertex[fill] at (5,-1) {};
          
          \draw[very thick, blue] (0,-4) to (0+2,-4+2) to (0+2,-4+2+1) to (0+2+1,-4+2+1+1);
          \draw[very thick] (0+2+1,-4+2+1+1) to (0+2+1,-4+2+1+1+1) to (0+2+1+1,-4+2+1+1+1+1) to (0+2+1+1,-4+2+1+1+1+1+1) to (0+2+1+1+1,-4+2+1+1+1+1+1+1);
          
          \draw[very thick, blue] (3,-2) to (3+1,-2+1) to (3+1,-2+1+1);
          \draw[very thick] (3+1,-2+1+1) to (3+1+1,-2+1+1+1) to (3+1+1,-2+1+1+1+1) to (3+1+2,-2+1+2+2) to (3+1+2,-2+1+2+2+1);
          
          \draw[very thick, blue] (5,-1) to (5+1,-1+1);
          \draw[very thick] (5+1,-1+1) to (5+1,-1+1+1) to (5+1,-1+1+1+1) to (5+1+1,-1+1+1+1+1) to (5+1+1,-1+1+1+1+1+1);
          
          \node at (4.4,3.9) {1};
          \node at (5.4,2.9) {2};
          \node at (6.4,2.9) {2};
          \node at (3.4,1.9) {3};
          \node at (4.4,0.9) {4};
          
          \node at (2.4, -0.1) {\color{blue}{1}};
          \node at (5.4, -0.1) {\color{blue}{1}};
          \node at (3.4, -1.1) {\color{blue}{2}};
          \node at (1.4, -2.1) {\color{blue}{3}};
          \node at (0.4, -3.1) {\color{blue}{4}};
\end{tikzpicture}
\begin{tikzpicture}[scale=0.4]
\draw[thick,<->] (1,0.8) to (2.9,0.8);
\node at (5.2,0)
{
	\ytableausetup{boxsize=1.4em} 
	\begin{ytableau}
	{1} & {2} & {2} \\
	{3} & {4} & {\color{blue}1} \\
	{\color{blue}1} & {\color{blue}2} \\
	{\color{blue}3} \\
	{\color{blue}4}
	\end{ytableau}
};
\end{tikzpicture}

\caption{Path enumeration from $A_i(-\lambda'_i + i - 1, 1 - \lambda'_i)$ to $B_j(j - 1, n)$. A nonintersecting lattice path system for $\lambda' = (5,3,2)$ and its tableau interpretation which decomposes into SSYT (black entries) and an elegant filling (blue entries). }
\label{fff}
\end{figure}

The Schur expansion  
\begin{equation}\label{gs}
g_{\lambda} = \sum_{\mu} f_{\lambda/\mu} s_{\mu}
\end{equation}
was proved in \cite{lp} bijectively by mapping each RPP with a given weight to a pair of SSYT of shape $\mu$ (with the same weight) and an elegant filling (see Definition \ref{elegant}) of shape $\lambda/\mu$. 

Note that each nonintersecting lattice path system from $A$ to $B$ (geometrically one can see here that each $A_i$ should go to $B_i$) decomposes into two parts: below the $y = 0$ line (`underwater') and the upper part having weights $x_i$. (For instance, in Figure \ref{fff} we have the weight $x_1 x_2^2 x_3 x_4$.) The upper part can be translated into an SSYT of some shape $\mu \subseteq \lambda$; and the `underwater' part into an elegant tableau of the remaining shape $\lambda/\mu$ which repeats the decomposition \eqref{gs}. So, $g_{\lambda} = \det\left[ w(A_i \to B_j)\right]_{1 \le i,j \le \ell}$. $\square$

Instead of counting elegant tableaux with all weights $1$, we can similarly give weights to numbers in these fillings, or equivalently give some weights to the `underwater' part. So, if we put the weight $t_i$ to the number $i$ in an elegant tableaux or to diagonal steps `underwater', the resulting function will correspond to the {\it refined} dual stable Grothendieck polynomials $\widetilde{g}_{\lambda/\mu}(x,t)$ introduced by Galashin, Grinberg, and Liu in \cite{ggl} where they gave Bender-Knuth type involutions for proving that $g_{\lambda/\mu}$ is a symmetric function. 

From our setting, the Schur decomposition rewrites with $t = (t_1, t_2, \ldots)$ parameters as follows
\begin{equation}
\widetilde{g}_{\lambda}(x,t) = \sum_{\mu} f_{\lambda/\mu}(t) s_{\mu},
\end{equation}
where $f_{\lambda/\mu}(t) = \sum_{\text{elegant tableaux of } \lambda/\mu} \prod t_i^{a_i}$ ($a_i$ is the number of times $i$ occurs in an elegant tableau). By the same argument as the proof above (but with the $t$ parameters added), we obtain the following Jacobi-Trudi formula 
\begin{equation}
\widetilde{g}_{\lambda}(x,t)  = \det \left[ \sum_{k = 0}^{\lambda'_i - 1} e_{k}(t_1, \ldots, t_{\lambda'_i - 1}) e_{\lambda'_i - i + j - k}(x) \right]_{1 \le i,j \le \ell(\lambda')}
\end{equation}
which was recently conjectured by Darij Grinberg (personal communication) in a more general setting for skew shapes.

\subsection{More formulas for $G_{\lambda}$} 
The functions $G_{\lambda}$ also satisfy the following Jacobi-Trudi type formula
\begin{equation}
G_{\lambda} = \det\left[\sum_{k \ge 0} \binom{\lambda'_i - 1 + k}{k} e_{\lambda'_i - i + j + k} \right]_{1 \le i,j \le \ell(\lambda')},
\end{equation}
which can be proved similarly as the formula for $g_{\lambda}$ above, taking into account the known Schur expansions of $G_{\lambda}$ \cite{lenart}. Hence we also obtain the following formulas for $G^{(\alpha, \beta)}_{\lambda}$
\begin{equation}
G^{(\alpha, \beta)}_{\lambda} = \det\left[\sum_{k \ge 0} \binom{\lambda'_i - 1 + k}{k} (\alpha + \beta)^k e_{\lambda'_i - i + j + k}\left(\frac{x}{1 - \alpha x}\right) \right]_{1 \le i,j \le \ell(\lambda')}.
\end{equation}

\section{Grothendieck polynomials indexed by permutations}\label{ggw}

The polynomials $G_{\lambda}$ arise as a special case of stable Grothendieck polynomials $G_{w}$ indexed by permutations $w \in S_n$. The symmetric group $S_n$ acts on $\mathbb{Z}[x_1,\ldots, x_n]$ by permuting the indices of variables. Let $s_i = (i, i+1)$ denote a simple transposition in $S_n$. Define the operators $\pi_i$ acting on polynomials in $\mathbb{Z}[x_1, \ldots, x_n]$ as 
$$
\pi_i := \frac{1 - s_i}{x_i - x_{i+1}} (1 -  x_{i+1})
$$
and for any permutation $w \in S_n$ with a reduced decomposition $w = s_{i_1} \ldots s_{i_{\ell}}$ set 
$
\pi_{w} := \pi_{i_1} \cdots \pi_{i_{\ell}},
$
which is independent of the choice of the reduced word (since the operators $\pi_i$ satisfy the braid relations). 

The {\it Grothendieck polynomials} $\mathfrak{G}^{}_w$ ($w \in S_n$) introduced by 
Lascoux and Sch\"utzenberger \cite{ls} is defined as follows: 
$$
\mathfrak{G}_{w_0} := x_1^{n-1} x_{2}^{n-2} \cdots x_{n-1}, \qquad \mathfrak{G}_{w} := \pi_{w^{-1}w_0} \mathfrak{G}_{w_0}, 
$$
where $w_0 = n (n-1) \ldots 2 1 \in S_n$ is the longest permutation.  
The double Grothendieck polynomials $\mathfrak{G}_w = \mathfrak{G}_w(x; y)$ of two sets of variables $x, y$ are defined in the same way, but 
$$
\mathfrak{G}_{w_0}(x;y) = \prod_{i + j \le n}(x_i + y_j - x_i y_j).
$$
Note that double Grothendieck polynomials satisfy the duality $\mathfrak{G}_{w}(x;y) = \mathfrak{G}_{w^{-1}}(y;x)$.

For any permutation $w \in S_n$ and $m \in \mathbb{Z}_{\ge 0}$, let $1^m \times w \in S_{n + m}$ be a permutation which fixes $1, \ldots, m$ and maps $i$ to $w(i - m) + m$ for $i > m$. Fomin and Kirillov \cite{fk, fk1} showed that the polynomials $\mathfrak{G}_{1^m \times w}$ eventually stabilize to the symmetric power series
$$
G_{w} := \lim_{m \to \infty} \mathfrak{G}_{1^m \times w}
$$
called {\it stable Grothendieck polynomials} (indexed by permutations). 

When $w$ is a {\it Grassmannian permutation}, i.e. has a single descent at some position $d$, the stable Grothendieck polynomial $G^{}_{\lambda}$ is indexed by an integer partition $\lambda = (\lambda_1 \ge \cdots \ge \lambda_d)$, where $\lambda_i = w(d+1 - i)  - (d + 1 - i)$. 

Various dualities hold for the double versions: $G_{\lambda}(x;y) = G_{\lambda'}(y;x)$, $G_{w}(x;y) = G_{w^{-1}}(y;x)$ and $G_{w}(x;y) = G_{w_0 w w_0}(y;x)$, where $w_0 = n \cdots 21 \in S_n$ is the longest permutation \cite{buch}. The series $G_{w}(1 - e^x; 1 - e^{-y})$ are {\it super-symmetric}, i.e. if for any $i$ we put $x_i = y_i = t,$ the resulting function is independent of $t$ (hence, considering $G$ of one set of variables is enough). 

Similarly, we may define the symmetric functions $G^{(\alpha, \beta)}_w(x_1, x_2, \ldots)$. First, extend the Grothendieck polynomials 
to the polynomials $\mathfrak{G}^{(\alpha + \beta)}_{w}$ via the operators $$\tilde\pi_i = \frac{1 - s_i}{x_i - x_{i+1}}(1 + (\alpha + \beta) x_{i+1}).$$ Taking the stable limit we define the power series $G^{(\alpha + \beta)}_w$ and finally substituting $x_i \to \frac{x_i}{1 - \alpha x_i}$ we obtain the functions $G^{(\alpha, \beta)}_w(x_1, x_2, \ldots)$. 

Note that in the double versions $G_{w}(x;y)$ we can put $y_j \to \frac{y_j}{1 - \beta y_j}$ to obtain the double versions of $G^{(\alpha, \beta)}_{w}(x;y)$ that are {\it super-symmetric} functions. 

\begin{theorem} The functions $G^{(\alpha, \beta)}_{w}(x_1, x_2, \ldots)$ ($w \in S_n$) have the following properties:
\begin{itemize}
\item[(i)] $G^{(\alpha, \beta)}_{w}$ is a finite linear combination of the elements $\{G^{(\alpha, \beta)}_{\lambda} \}$;
\item[(ii)] $\omega(G^{(\alpha, \beta)}_{w}) = G^{(\beta, \alpha)}_{w^{-1}} = G^{(\beta, \alpha)}_{w_0 w w_0},$ where $w_0 \in S_n$ is the longest permutation.
\end{itemize}
\end{theorem}

\begin{proof}
We use Buch's result (\cite{buch}, Theorem 6.13) that 
\begin{equation}\label{gw}
G_{w} = \sum_{\lambda} a_{w, \lambda} G_{\lambda} \in \Gamma
\end{equation}
can be expressed as a finite linear combination of $G_{\lambda}.$ It is then easy to see (from the definition) that the same holds for 
$$G^{(\alpha, \beta)}_w = \sum_{\lambda} a^{(\alpha, \beta)}_{w, \lambda} G^{(\alpha, \beta)}_{\lambda} \in \Gamma^{(\alpha, \beta)},$$ 
where $a^{(\alpha, \beta)}_{w, \lambda} = (\alpha + \beta)^{|\lambda| - \ell(w)} a_{w,\lambda}$.

Hence to prove (ii), we have to show that $\alpha_{w, \lambda} = \alpha_{w^{-1},\lambda'}$ and $\alpha_{w, \lambda} = \alpha_{w_0w w_0,\lambda'}$. It is known that $\mathfrak{G}_{w}(x;y) = \mathfrak{G}_{w^{-1}}(y;x)$  
and therefore, $G_{w}(x;y) = G_{w^{-1}}(y;x)$ (since in the limit we have $(1^m \times w)^{-1} = 1^m \times w^{-1}$). It is also known that $G_{w}(x;y) = G_{w_0 w w_0}(y;x)$ (Fomin's lemma \cite{buch}). Combining these properties with $G_{\lambda}(x;y) = G_{\lambda'}(y;x)$ and \eqref{gw} we conclude both symmetries $\alpha_{w, \lambda} = \alpha_{w^{-1},\lambda'}$ and $\alpha_{w, \lambda} = \alpha_{w_0w w_0,\lambda'}$.
\end{proof}

\section{Theory of canonical bases}\label{skl}
We now discuss the way how canonical stable Grothendieck polynomials can be put in context of the Kazhdan-Lusztig theory of canonical bases \cite{kl}. 

First we propose a general scheme how canonical bases can be derived for bases of symmetric functions. For any element $a \in \mathbb{Z}[\alpha, \beta]$,\footnote{One can take a more familiar specialization $(\alpha, \beta) = (q, q^{-1})$.} let $a \to \overline{a}$ be the involution switching $\alpha$ and $\beta$. Let $\{ u_{\lambda}\}$ be a $\mathbb{Z}[\alpha, \beta]$-basis of $\Lambda$. For $f \in \Lambda$ define the {\it bar} involution $f \to \overline{f}$ given by e.g. $\overline{f} = \sum \overline{a_{\lambda}} s_{\lambda}$ for $f=\sum a_\lambda s_\lambda$. Suppose that 
\begin{equation}\label{wu}
\omega(\overline{u_{\lambda'}}) \in u_{\lambda} + \sum_{\mu < \lambda} \mathbb{Z}[\alpha, \beta] u_{\mu}
\end{equation}
for a (partial) order $<$ on partitions (say, compatible with the weight of partition, e.g. lexicographic, or by inclusion of diagrams). (Note that the standard involution $\omega$ commutes with the {bar} involution.) 

Define 
$$\mathbb{Z}[\alpha > \beta] := \{\sum_{j < i < \infty} a_{ij} \alpha^i \beta^j : a_{ij} \in \mathbb{Z} \},\quad \mathbb{Z}[\alpha \ge \beta] := \{\sum_{j \le i < \infty} a_{ij} \alpha^i \beta^j : a_{ij} \in \mathbb{Z} \} \subset \mathbb{Z}[\alpha, \beta]$$ 
(i.e. polynomials with terms having powers of $\alpha$ greater (or equal) than powers of $\beta$). 

\begin{proposition}\label{kazl1}
There is at most one $\mathbb{Z}[\alpha, \beta]$-basis $\{C_{\lambda} \}$ of $\Lambda$ satisfying:
\begin{itemize}
\item[(i)] $\omega(C_{\lambda}) = \overline{C_{\lambda'}}$;
\item[(ii)] $C_{\lambda} \in u_{\lambda} + \sum_{\mu < \lambda} \mathbb{Z}[\alpha > \beta] u_{\mu}$. 
\end{itemize}
\end{proposition}

Such basis $\{C_{\lambda} \}$, if exists, is a candidate for being a {\it canonical basis} with the respect to the given {\it pre-canonical structure} consisting of the {\it standard} basis $\{ u_{\lambda} \}$ satisfying \eqref{wu} and the involution $\omega$ composed with the {\it bar} and index transposition. 

The proof of this property is similar to the proof of uniqueness of Kazhdan-Lusztig canonical bases (for Hecke algebras) \cite{kl}. We show this argument in a more specific situation in the proof of Theorem \ref{gab} below. 

Consider now instead of the generic basis $\{u_{\lambda} \}$, the $\mathbb{Z}[\beta]$-basis of $\Lambda$ given by the dual stable Grothendieck polynomials $\{g^\beta_{\lambda} \}$. Since, 
$g^\beta_{\lambda} = s_{\lambda} + \text{lower degree terms,}$
we have 
$$\omega(\overline{g^{\beta}_{\lambda'}}) = g^{(\alpha, 0)}_{\lambda} = g^\beta_{\lambda} + \sum_{\mu \subset \lambda} q_{\lambda \mu}(\alpha, \beta) g^\beta_{\mu}, \qquad q_{\lambda \mu}(\alpha, \beta) \in \mathbb{Z}[\alpha, \beta].$$
For example,
\begin{align*}
g^{(\alpha, 0)}_{21} &= g^\beta_{21} + \alpha g^\beta_{11} - \beta g^\beta_{2} - \alpha \beta g^\beta_{1}\\
g^{(\alpha, 0)}_{22} &= g^\beta_{22} + (\alpha - \beta) g^\beta_{21} + \alpha^2 g^\beta_{11} - \alpha\beta g^\beta_{2} - \alpha \beta (\alpha - \beta) g^\beta_{1}. 
\end{align*}

The dual canonical polynomials $g^{(\alpha, \beta)}_{\lambda}$ satisfy the condition (i): we have $\omega(g^{(\alpha, \beta)}_{\lambda}) = \overline{g^{(\alpha, \beta)}_{\lambda'}}$. Let us consider their expansion in the (standard) $\{g^\beta_{\lambda}\}$ basis. Let
$$g^{(\alpha, \beta)}_{\lambda} = g^\beta_{\lambda} + \sum_{\mu \subset \lambda} p_{\lambda \mu}(\alpha, \beta) g^\beta_{\mu}, \qquad p_{\lambda \mu}(\alpha, \beta) \in \mathbb{Z}[\alpha, \beta].$$
For example, some of the $g^\beta_{\lambda}$ expansions look as follows:
\begin{align*}
g_{21}^{(\alpha, \beta)} &= g^\beta_{21} + \alpha g^\beta_{11} \\
g_{22}^{(\alpha, \beta)} &= g^\beta_{22} + \alpha g^\beta_{21} + \alpha(\alpha + \beta) g^\beta_{11}\\
g_{31}^{(\alpha, \beta)} &= g^\beta_{31} + 2\alpha g^\beta_{21} + \alpha^2 g^\beta_{11}\\
g_{32}^{(\alpha, \beta)} &= g^\beta_{32} + 2\alpha g^\beta_{22} + \alpha g^\beta_{31} + 2\alpha^2 g^\beta_{21} + \alpha^2(\alpha + \beta) g^\beta_{11}\\
g_{33}^{(\alpha, \beta)} &= g^\beta_{33} + 2\alpha g^\beta_{32} + \alpha(3\alpha + 2\beta) g^\beta_{22} + \alpha^2 g^\beta_{31} + \alpha^2(2\alpha + \beta) g^\beta_{21} + \alpha^2 (\alpha + \beta)^2 g^\beta_{11}.
\end{align*}
Here we see that $g^{(\alpha, \beta)}_{\lambda}$ do not fall into the characterization (ii) of Proposition \ref{kazl1}, e.g. $p_{(22),(11)}(\alpha, \beta) = \alpha^2 + \alpha\beta \not\in \mathbb{Z}[\alpha > \beta]$. However, we can change the condition (ii) and get a similar characterization, if the polynomials $p_{\lambda \mu}(\alpha, \beta)$ satisfy some additional properties which we ask below. 

Say that $p \in \mathbb{Z}[\alpha, \beta]$ has a {\it free} term, if it contains a term $(\alpha \beta)^i$ for some $i \ge 1$. 

\begin{problem}
Which of the following {\it positivity} properties hold for the polynomials $p_{\lambda \mu}(\alpha, \beta)$? 
\begin{itemize}
\item[(a)] $p_{\lambda \mu}(\alpha, \beta) \in \mathbb{Z}_{\ge 0}[\alpha, \beta]$ (i.e. polynomials in $\alpha, \beta$ with positive integer coefficients);
\item[(b)] $p_{\lambda \mu}(\alpha, \beta) \in \mathbb{Z}_{}[\alpha \ge \beta]$;
\item[(c)] at most one of $p_{\lambda \mu}(\alpha, \beta)$,  $p_{\lambda' \mu'}(\alpha, \beta)$ has a free term.
\end{itemize}
\end{problem}

If (a), (b), (c) hold, then $g^{(\alpha, \beta)}_{\lambda}$ is a unique polynomial satisfying these conditions and self-duality (i). Otherwise, it is perhaps needed to describe for which $\lambda, \mu$, the polynomial $p_{\lambda \mu}(\alpha, \beta)$ has a free term. Then if (b), (c) hold, $\{g^{(\alpha, \beta)}_{\lambda}\}$ is again a unique basis falling into the characterization. (The coefficients $p_{\lambda \mu}(\alpha, \beta)$ are some analogs of Kazhdan-Lusztig polynomials.)

Note that (for stable Grothendieck polynomials) there is no other candidate for canonical basis with exchange coefficients in $\mathbb{Z}[\alpha > \beta]$; i.e. when we construct these canonical polynomials, we are forced to consider $\mathbb{Z}[\alpha \ge \beta]$. For example, computing the coefficients $p_{\lambda \mu}(\alpha, \beta)$ recursively, we obtain that $p_{(22), (11)}(\alpha, \beta) - p_{(22), (2)}(\beta, \alpha) = \alpha^2 + \alpha\beta$ and so if we know positivity properties, we may conclude that $p_{(22), (11)}(\alpha, \beta) = \alpha^2 + \alpha \beta$ and $p_{(22), (2)}(\beta, \alpha) = 0$, which gives a unique $g^{(\alpha, \beta)}_{22}$.

However we can replay the same characterization on the ring {\it generators} $g^{(\alpha, \beta)}_{(k)}, g^{(\alpha, \beta)}_{(1^k)}$. 

\begin{theorem}\label{gab}
There is a unique set $\{C_k, k \in \mathbb{Z}_{> 0} \}$ of generators of $\Lambda$ with coefficients in $\mathbb{Z}[\alpha, \beta]$ satisfying:
$$
C_k \in g^\beta_{(k)} + \sum_{i < k} \mathbb{Z}[\alpha > \beta] g^\beta_{(i)}, \qquad \omega(\overline{C_k}) \in g^\beta_{(1^k)} + \sum_{i < k} \mathbb{Z}[\alpha > \beta] g^\beta_{(1^i)}.
$$ The elements $C_k = g^{(\alpha, \beta)}_{(k)}$ coincide with the dual canonical stable Grothendieck polynomials indexed by a single row. 
\end{theorem}

\begin{proof}
Let us first prove that there is at most one such set of elements $\{C_k \}$. Recall that
\begin{equation}\label{geh}
g^{\beta}_{(k)} = h_{k}, \qquad g^{\beta}_{(1^k)} = \sum_{i = 1}^k  \beta^{k - i} \binom{k - 1}{i - 1} e_i.
\end{equation}
We will prove by induction on $k-i$ that the coefficients $p_{k,i}(\alpha, \beta), p'_{k,i}(\alpha, \beta) \in \mathbb{Z}[\alpha > \beta]$ given by 
\begin{equation}\label{aaa}
C_k = g^\beta_{(k)} + \sum_{i < k} p_{k,i}(\alpha, \beta) g^\beta_{(i)}, \qquad \omega(\overline{C_k}) = g^\beta_{(1^k)} + \sum_{i < k} p'_{k,i}(\alpha, \beta) g^\beta_{(1^i)}.
\end{equation}
are unique. First for $i = k$, $p_{k,k}(\alpha, \beta) = p'_{k,k}(\alpha, \beta) = 1.$ Suppose that we know all coefficients $p_{k,j}, p'_{k,j}$ for $i < j \le k$. Combining both expansions in \eqref{aaa} we obtain
\begin{equation}\label{a4}
{C_k} =g^\beta_{(k)} + \sum_{i < k} p_{k,i}(\alpha, \beta) g^\beta_{(i)}  
= \omega(g^{\alpha}_{(1^k)}) + \sum_{i < k} p'_{k,i}(\beta, \alpha) \omega(g^{\alpha}_{(1^i)}),
\end{equation}
where by \eqref{geh},
$$
\omega(g^{\alpha}_{(1^k)}) = \sum_{i = 1}^k  \alpha^{k - i} \binom{k - 1}{i - 1} h_i = \sum_{i = 1}^k  \alpha^{k - i} \binom{k - 1}{i - 1} g^\beta_{(i)}.
$$
Hence, from \eqref{a4} 
\begin{equation}\label{a5}
g^\beta_{(k)} + \sum_{i < k} p_{k,i}(\alpha, \beta) g^\beta_{(i)} = \sum_{i = 1}^k  \alpha^{k - i} \binom{k - 1}{i - 1} g^\beta_{(i)} + \sum_{j < k} p'_{k,j}(\beta, \alpha) \sum_{i = 1}^j  \alpha^{j - i} \binom{j - 1}{i - 1} g^\beta_{(i)}.
\end{equation}
Since, $h_{i} = g^{\beta}_{(i)}$ is a homogeneous component, the coefficients in the term $[g^{\beta}_{(i)}]$ from both sides of \eqref{a5} coincide, from that   
we obtain the recurrence relation
$$
p_{k,i}(\alpha, \beta) -  p'_{k,i}(\beta, \alpha) = \alpha^{k - i} \binom{k - 1}{i - 1} + \sum_{i < j < k} p'_{k,j}(\beta, \alpha) \alpha^{j - i} \binom{j - 1}{i - 1}.
$$
The r.h.s of the latter equation is known by the induction hypothesis, and so both polynomials $p_{k,i}(\alpha, \beta)$, $p'_{k,i}(\beta, \alpha)$ are also uniquely recovered, since they belong to $\mathbb{Z}[\alpha > \beta]$, $\mathbb{Z}[\beta > \alpha]$ and split the terms accordingly (we do not have any cancellation of terms). 

Finally, one can easily see that the (algebraically independent) elements $C_k$ given explicitly by the formula
$$
C_k =  \sum_{i = 1}^k \alpha^{k - i} \binom{k - 1}{i - 1} h_i = g^{(\alpha, \beta)}_{(k)}
$$
are indeed the unique elements satisfying the given conditions. 
\end{proof}

\begin{corollary}
The set $\{g^{(\alpha, \beta)}_{(k)}, g^{(\alpha, \beta)}_{(1^k)} | k \in \mathbb{Z}_{> 0} \}$ is the unique set of elements of $\Lambda$ satisfying:
\begin{itemize}
\item[(i)] $\omega(g^{(\alpha, \beta)}_{(k)}) = g^{(\beta, \alpha)}_{(1^k)}$;
\item[(ii)] $g^{(\alpha, \beta)}_{(k)} \in g^\beta_{(k)} + \sum_{i < k} \mathbb{Z}[\alpha > \beta] g^\beta_{(i)}$ and $g^{(\alpha, \beta)}_{(1^k)} \in g^\beta_{(1^k)} + \sum_{i < k} \mathbb{Z}[\alpha > \beta] g^\beta_{(1^i)}$.
\end{itemize}
\end{corollary}

Furthermore, the homomorphism $\varphi :\Lambda \to \Lambda$ defined on the generators by $$\varphi(h_k) = \sum_{i = 1}^k \alpha^{k - i} \binom{k - 1}{i - 1} h_i = g^{(\alpha, \beta)}_{(k)}$$ 
satisfies $\varphi(g^{(\alpha + \beta)}_{\lambda}) = g^{(\alpha, \beta)}_{\lambda}$. Therefore, under these properties we have the unique (dual) canonical polynomials $g^{(\alpha, \beta)}_{\lambda}$. In addition, the dual family of functions $G^{(\alpha, \beta)}_{\lambda}$ (via the Hall inner product) is also unique. 

\section{Concluding remarks, special cases}\label{final}

\subsection{Dual Hopf algebras}

There are Hopf algebra structures with the bases $\{G^{(\alpha, \beta)}_{\lambda}\},$ $\{g^{(\alpha, \beta)}_{\lambda}\}$. Let $\Gamma^{(\alpha, \beta)} = \bigoplus_{\lambda} \mathbb{Z}[\alpha, \beta] \cdot G^{(\alpha, \beta)}_{\lambda}$ with completion $\hat\Gamma^{(\alpha, \beta)}$ and $\bar\Gamma^{(\alpha, \beta)} = \bigoplus_{\lambda} \mathbb{Z}[\alpha,\beta] \cdot g^{(\alpha, \beta)}_{\lambda}$. Note that $\Gamma^{(0,0)} = \bar\Gamma^{(0,0)} = \Lambda$ as self-dual Hopf algebras with the Schur basis $\{s_{\lambda}\}.$ One can show that $\hat\Gamma^{(\alpha, \beta)}$ and $\bar\Gamma^{(\alpha, \beta)} = \bigoplus_{\lambda} \mathbb{Z}[\alpha,\beta] \cdot g^{(\alpha, \beta)}_{\lambda}$ are Hopf algebras with the antipodes given by $S(G^{(\alpha, \beta)}_{\lambda}(x)) = G^{(\beta, \alpha)}_{\lambda'}(-x),$ and $S(g^{(\alpha, \beta)}_{\lambda}(x)) = g^{(\beta, \alpha)}_{\lambda'}(-x),$ respectively. Then, $\hat\Gamma^{(\alpha, \beta)}$ and $\bar\Gamma^{(-\alpha, -\beta)}$ will be dual Hopf algebras via the Hall inner product. Furthermore, $\bar\Gamma^{(\alpha, \beta)} \cong \Lambda$ as an abstract Hopf algebra (with different bases). 

\subsection{Dual noncommutative operators}
Let us recall noncommutative operators which build stable Grothendieck polynomials (Section \ref{dfg}). In case of the polynomials $G_{\lambda}, g_{\lambda}$, the operators 
$$\widetilde{u}_i = u_i(1 - d_i), \quad \text{ and }\quad \widetilde{d}_i = (1-d_i)^{-1}d_i = d_i + d_i^2 + \cdots$$ 
(written in terms of the Schur operators $u_i, d_i$, see section \ref{dfg}) seem to be dual to each other.  The operators $\widetilde{u}$ build $G_{\lambda/\mu}$ and $\widetilde{d}$ build the dual polynomials $g_{\lambda/\mu}$ (so $(1-d_i)^{-1}d_i$ is an implementation of column deleting (adding) operators used in \cite{lp}.)

\subsection{Some specializations} 
By combinatorial formulas, the polynomials $g^{(\alpha, 0)}_{\lambda} = \omega(g^{(0, \alpha)}_{\lambda'})$ can be expressed as generating series similar to Schur polynomials. We have
\begin{equation}
g^{(\alpha, 0)}_{\lambda} = \sum_{T \in SSYT(\lambda)} (x | \alpha)^T, \qquad (x |  \alpha)^T = \prod_{i \in T} x_{i}^{r_i} (x_{i} + \alpha)^{a_i - r_i}, 
\end{equation}
where $r_i$ is a number of rows which contain $i$ and $a_i$ is the total number of entries $i$ in $T$ (this obviously reduces to Schur polynomials when $\alpha = 0$). 
Combining this formula with the Jacobi-Trudi identity, one can see that
$$
g^{(1,0)}_{\lambda}((-1)^{n}) = (-1)^{|\lambda|}f^\lambda_n = \det\left[\binom{-n + \lambda_i - 1}{\lambda_i - i + j} \right]_{1 \le i,j \le \ell(\lambda)},
$$
where $f^\lambda_n$ is the number of row and column strict Young tableaux of shape $\lambda$ filled with numbers from $\{ 1, \ldots, n\}$. 

The hook-content formula (e.g. \cite{ec2}) counts SSYT having maximal entry at most $n$
$$
s_{\lambda}(1^n) = \det\left[\binom{n}{\lambda'_i - i + j} \right]_{1 \le i,j \le \ell(\lambda').} = \prod_{(i,j) \in \lambda} \frac{n + j - i}{h_{ij}},
$$
where $h_{ij} = \lambda_i - i + \lambda'_j - j + 1$ is the hook-length of the box $(i,j)$. 
The formula derived from the Jacobi-Trudi identity for $g_{\lambda}$
$$
g^{(0,1)}_{\lambda}(1^n) = g_{\lambda}(1^n) = \det\left[\binom{n + \lambda'_i - 1}{\lambda'_i - i + j} \right]_{1 \le i,j \le \ell(\lambda').}
$$
gives the number of RPP having maximal entry at most $n$. It does not always factorise as nicely as the formula above, for example
$$
g_{(532)}(1^n) = \frac{n(n+1)^2(n+2)^2(n+3)(n+4)(15n^3 + 58n^2 + 71n + 24)}{120960}.
$$
But in a special case when $\lambda = (k^m)$ has a rectangular shape $m \times k$, it corresponds to the hook-content formula with the shift $n \to n + m - 1$, i.e. the number of RPP of shape $m \times k$ with maximal entry at most $n$ is 
$$g_{(k^m)}(1^n) = \prod_{(i,j) \in (k^m)} \frac{n + m  - 1 + j - i}{h_{ij}} = \prod_{i = 1}^{m} \prod_{j = 1}^k \frac{n + m - 1 + j - i}{m + k  - i - j + 1}.$$
More generally, for a rectangular shape $$g_{(k^m)}(x_1, \ldots, x_n) = s_{(k^m)}(x_1, \ldots, x_n, 1^m)$$ is a specialization of Schur function (this can be seen e.g. from our lattice path construction; note that for any $\lambda$, $g_{\lambda}$ is a specialization of flagged Schur function \cite{ln}). 

The polynomials $G^{\beta}_{\lambda}$ have some nice specializations when $\lambda$ is a single row or column.
For example, 
$$
G^{\beta}_{(k)}(1, q, \ldots, q^{n-1}) + \beta G^{\beta}_{(k+1)}(1, q, \ldots, q^{n-1}) = \left[{n + k-1 \atop k} \right] \prod_{j =0}^{n-1} (1 + \beta q^{j}),
$$
where $\left[{n\atop k} \right]$ is the $q$-binomial coefficient;
when $n\to \infty$,
$$
G^{\beta}_{(k)}(1, q, q^2, \ldots) + \beta G^{\beta}_{(k+1)}(1, q, q^2, \ldots) = \frac{\prod_{j \ge 0} (1 + \beta q^{j})}{\prod_{j \ge 1}(1 - q^j)}.
$$

\subsection{The case $\alpha + \beta = 0$, deformed Schur functions}
For $\beta = -\alpha$ we have
$$
G^{(\alpha, -\alpha)}_{\lambda}(x) = s_{\lambda}\left(\frac{x}{1 - \alpha x}\right), \qquad \omega(G^{(\alpha, -\alpha)}_{\lambda}(x)) 
= s_{\lambda'}\left(\frac{x}{1 + \alpha x}\right).
$$
By Lemma \ref{ll} we have
$$
f_{\nu/\mu}{(\alpha, -\alpha)} = \alpha^{|\lambda/\mu|} (-1)^{n(\lambda/\mu)} \psi_{\lambda/\mu}
$$
which is nonzero when $\lambda$ and $\mu$ have the same number of boxes on the main diagonal.
Hence, the dual polynomials $g$ have the following Schur expansion
$$g^{(\alpha, -\alpha)}_{\lambda} = \sum_{\mu} f_{\lambda/\mu}{(\alpha, -\alpha)} s_{\mu} = \sum_{\mu} \alpha^{|\lambda/\mu|} (-1)^{n(\lambda/\mu)} \psi_{\lambda/\mu} s_{\mu},$$
where $\mu \subset \lambda$ and $\mu$, $\lambda$ have the same number of boxes on the main diagonal; the diagram $\lambda/\mu$ decomposes into two subdiagrams and $n(\lambda/\mu)$ is the number of boxes in a lower part; recall that $\psi_{\lambda/\mu}$ is the number of the so-called dual hook tableaux \cite{molev}, see Definition~\ref{dht}. In particular, $g^{(\alpha, -\alpha)}_{\lambda} = s_{\lambda}$ (for generic $\alpha$) iff  $\lambda$ has a square shape. 

From combinatorial presentation (RBT) of $g^{(\alpha, \beta)}_{\lambda}$, the property $\alpha + \beta = 0$ means that we do not count tableaux with some inner parts. These tableaux have property that no $2 \times 2$ square contains the same number. Let us call them {\it rim tableaux} (RT). 
We then tile these RT by rim hooks consisting of the same letter such that no two rim hooks are connected horizontally. The weight $w(\rho)$ of any such rim hook $\rho$ is given by $(-1)^{ht(\rho) - 1}\alpha^{|\rho| - 1} ,$ where $ht(\rho)$ is the height of $\rho$. So we have the expression
$$
g^{(\alpha, -\alpha)}_{\lambda} = \sum_{T} w(T) \prod_{i \ge 1} x_i^{\#\{\text{rim hooks in $T$ containing }i\}}, 
$$
where the sum runs through all tilings of RT of shape $\lambda$ into rim hooks as described above and $w(T)$ is the product of weights of rim hooks in $T$.

\end{document}